\documentclass[11pt]{article}
\usepackage[a4paper]{geometry}

\usepackage{amsmath}
\usepackage{amssymb}
\usepackage{amsthm}
\usepackage{mathrsfs}
\usepackage{bbm}
\usepackage{empheq}
\usepackage[loose]{subfigure}
\usepackage{epsfig}
\usepackage{graphicx}
\usepackage{psfrag}
\usepackage[usenames,dvipsnames]{pstricks}
\usepackage{pst-plot}
\usepackage[colorlinks]{hyperref}
\usepackage{tabls}
\usepackage{paralist}
\usepackage{hyperref}

\DeclareSymbolFontAlphabet{\Bbb}{AMSb}

\newlength{\fixboxwidth}
\newcommand{\COMMENT}[1]{}
\newcommand{\E}{\mathbb{E}}

\newcommand{\Dmaps}{\mathfrak{D}}
\newcommand{\Dmap}{\mathbb{D}}

\newcommand{\Dspace}{\mathcal{D}}

\newcommand{\supp}{\operatorname{supp}}

\newcommand{\eins}{\mathbbm{1}}

\newcommand{\one}{\mathbbm{1}}

\newcommand{\R}{\mathbb{R}}

\newcommand{\e}{\varepsilon}

\newtheorem{thm}{Theorem}[section]
\newtheorem{prop}[thm]{Proposition}
\newtheorem{lem}[thm]{Lemma}

\newtheorem{cor}[thm]{Corollary}
\theoremstyle{definition}

\newtheorem{rmk}[thm]{Remark}

\def \d         { \delta }
\hypersetup{
    linkcolor=blue,
}

\title{Brittleness of Bayesian inference and new Selberg formulas\footnotetext{	
Houman Owhadi:  owhadi@caltech.edu\\Clint Scovel: corresponding author,  clintscovel@gmail.com
}
}

\author{Houman Owhadi \& Clint Scovel \\
California Institute of Technology}

\date{\today}

\makeatletter
\@addtoreset{equation}{section}
\@addtoreset{figure}{section}
\@addtoreset{table}{section}
\makeatother

\renewcommand{\thefigure}{\arabic{section}.\arabic{figure}}

\makeatletter
\renewcommand{\p@subfigure}{\thefigure}
\makeatother

\newcounter{mycount}

\begin{document}

\maketitle

 The incorporation of priors \cite{BayesOUQ} in the
 Optimal Uncertainty Quantification (OUQ) framework \cite{OSSMO:2011}  reveals
  brittleness in Bayesian inference; a model may share an arbitrarily large number of finite-dimensional marginals with, or be arbitrarily close (in Prokhorov or total variation metrics) to, the data-generating distribution and still make the largest possible prediction error after conditioning on an arbitrarily large number of  samples.
The initial purpose of this paper is to unwrap this  brittleness mechanism by providing (i)  a quantitative version of the Brittleness Theorem of \cite{BayesOUQ} and (ii) a detailed and comprehensive analysis of its application to the revealing example of estimating the mean of a random variable on the unit interval $[0,1]$ using priors that exactly capture the distribution of an arbitrarily large number of Hausdorff moments.

However, in doing so, we discovered that the free parameter associated with Markov and Kre\u{\i}n's  canonical representations of truncated Hausdorff moments
generates  reproducing kernel identities corresponding to  reproducing kernel Hilbert spaces of  polynomials.
 Furthermore, these reproducing identities lead to biorthogonal systems of Selberg integral formulas.

This process of discovery appears to be generic: whereas Karlin and Shapley used Selberg's integral formula to first compute the volume of the Hausdorff moment space (the polytope defined by the first $n$ moments of a probability measure on the interval $[0,1]$), we observe that the computation of that volume along with higher order moments of the uniform measure on the moment space, using different finite-dimensional representations of subsets of the infinite-dimensional set of probability measures on $[0,1]$
 representing the first $n$ moments, leads to families of equalities corresponding to classical and new Selberg identities.

\newpage

\tableofcontents

\section{Introduction}

 Optimal Uncertainty Quantification (OUQ) \cite{OSSMO:2011} provides a framework
for the computation of optimal bounds on quantities of interest -given a set of available information
and specified assumptions. Although this framework is neither frequentist nor Bayesian,
 in that it is simply expressed in terms of optimization  over measures and functions, a natural question arises; what happens when we introduce priors into OUQ?
In Owhadi et al.~\cite{BayesOUQ}, this program was initiated  through the introduction of a further set of assumptions, namely, the
assumptions regarding the prior  on the specified assumption set.
A corresponding reduction theory for optimization problems over measures on  spaces of measures is established, facilitating
 the computation of optimal bounds on prior and posterior values
and the analysis of the consequences of conditioning on observed data.
However, the completion of this program reveals  Brittleness theorems \cite[Thm.~4.13, Thm.~6.4, Thm.~6.9]{BayesOUQ}
 for Bayesian Inference -mild assumptions  are sufficient to demonstrate that, given a set of priors, conditioning on observations
can produce arbitrary results, regardless of the sample size.

Although it is known from the results of Diaconis and Freedman that the Bayesian method may fail to converge or may converge towards the wrong solution (i.e., be inconsistent) if the underlying probability mechanism allows an infinite number of possible outcomes \cite{DiaconisFreedman:1986} and that in these non-finite-probability-space situations, this lack of convergence (commonly referred to as \emph{Bayesian inconsistency}) is the rule rather than the exception \cite{DiaconisFreedman:1998}, it is also known, from the Bernstein-Von Mises Theorem \cite{Bernstein:1964, vonMises:1964} (see also LeCam \cite{LeCam:1953}), that consistency (convergence upon observation of sample data) does indeed hold, under some regularity conditions, if the data-generating distribution of the sample data belongs to the finite
 dimensional family of distributions parameterized by the model. Furthermore, although it is also known that this convergence may fail under model misspecification \cite{White:1982, Grunwald:2004, Grunwald:2006, Christophe:2002, Christophe:2008, Kleijn:2012, Lian:2009, Gustafson:2001} (i.e. when  the data-generating distribution does not belong to the
  family of distributions parameterized by the model), it is natural to wonder whether a ``close enough'' model has good convergence properties: see e.g.~\cite{Draper:1995, Samaniego:2010, Draper:2013} and in particular G.\ E.\ P.\ Box's  question \cite[p.~74]{Box:1987} ``Remember that all models are wrong; the practical question is how wrong do they have to be to not be useful?''

 The Brittleness theorems \cite[Thm.~4.13, Thm.~6.4, Thm.~6.9]{BayesOUQ} suggest that there
 may be no such thing as a ``close enough'' model if Box's question is answered
  in the classical framework of  Bayesian sensitivity analysis (where given the data and a class of priors one computes optimal bounds on posterior values); indeed, if ``closeness'' is defined  (i) as sharing an arbitrarily large  finite  number of finite-dimensional marginals or
 (ii) using the Prokhorov or total variation metrics, then the posterior values of such ``close'' models may be as distant as possible after conditioning on an arbitrarily large number of sample data.

The primary motivation for this paper is to unwrap the mechanism causing  this brittleness by providing (i)  a quantitative version of the Brittleness Theorem \cite[Thm.~4.13]{BayesOUQ} and (ii) a detailed and comprehensive analysis of its application to the informative example from \cite[Ex.~4.16]{BayesOUQ} of estimating the mean of a random variable on the unit interval  using priors that exactly capture the distribution of an arbitrary large number of Hausdorff moments.
In this  example, the probability distribution $\mu^\dagger$ of $X$ is an unknown element of the set of all possible probability distributions on $[0,1]$, i.e.
 $\mu^\dagger \in \mathcal{A}:=\mathcal{M}([0,1])$. The set of prior probability distributions $\pi$ on $\mu \in \mathcal{A}$ (i.e. $\pi \in \mathcal{M}(\mathcal{A})$) is defined as the set of priors $\pi$ under which the  vector of truncated Hausdorff moments $(\E_\mu[X],\ldots,\E_\mu[X^n])$ is uniformly distributed on the  truncated Hausdorff moment set
$M^n\subset \R^n$ defined as the set of $q=(q_1,\ldots,q_n)\in \R^n$ such that there exists a  probability measure $\mu$ on $[0,1]$ with $\E_\mu[X^i]=q_i$ for $i\in \{1,\ldots,n\}$  ($M^n$ is the polytope of $\R^n$ corresponding to the set of possible values for the first $n$ moments of a measure of probability on the interval $[0,1]$)).
In this case, the computation of optimal bounds on posterior values leads naturally to the calculation of the  Lebesgue volume of certain
  subsets of the set $M^{n}$ of truncated Hausdorff moments.

Curiously, whereas Karlin and Shapley \cite{KarlinShapley} used  Selberg's  integral formula to first compute
the volume of the truncated Hausdorff moment space $M_n$, inadvertantly stimulating the development of the theory of the Selberg integral formulas\footnote{In  discussing the history and  importance of the Selberg integral formulas,    Forrester and Waardan
\cite[Pg.~3]{ForresterWarnaar} mention their first application: ``For over thirty years
the Selberg integral went essentially unnoticed. It was used only once—-in the
special case $ \alpha = \beta = 1, \gamma = 2$—-in a study by S. Karlin and L.S. Shapley relating
to the volume of a certain moment space, published in 1953.''}, it appears that computing the volume of
 the truncated
 Hausdorff moment space $M^n$ using different finite-dimensional representations of $M^n$ in the infinite-dimensional space  $\mathcal{M}\big([0,1]\big)$
  reveals a new family of Selberg integral formulas (see Theorems \ref{thm_selberg}, \ref{thm_selberg2}, \ref{thm_selberg_explicit} and Corollary \ref{cor_selberg2}).
This process of discovery  appears to be generic and we will now describe its main principles.

  Let $\Psi$ be the function mapping each measure $\mu \in \mathcal{M}\big([0,1]\big)$ into its first $n$ moments
\begin{equation}
\Psi(\mu):=\big(\mathbb{E}_{X\sim \mu}[X], \mathbb{E}_{X\sim \mu}[X^2],\ldots,\mathbb{E}_{X\sim \mu}[X^n]\big).
\end{equation}
 Note that
\begin{equation}
M^n:=\Psi\Big(\mathcal{M}\big([0,1]\big)\Big).
\end{equation}
The classical and new Selberg identities are obtained by computing the volume of $M^n$ using different finite dimensional representations in $\mathcal{M}\big([0,1]\big)$. These finite dimensional representations are obtained by restricting $\Psi$ to convex sums of Diracs, i.e. to  measures $\mu \in \mathcal{M}\big([0,1]\big)$ of the form
\begin{equation}\label{eq:musumdiracs}
\mu=\sum_{j=1}^N \lambda_j \delta_{t_j}
\end{equation}
 where $0\leq t_1<\cdots<t_N\leq 1$ and $\lambda_1,\ldots,\lambda_N>0$ with $\sum_{j=1}^N \lambda_j=1$. Note that if $\mu$ is of the form \eqref{eq:musumdiracs}, then $\Psi(\mu)=(q_1,\ldots,q_n)$ with $q_i=\sum_{j=1}^N \lambda_j t_j^i$.

For each measure $\mu$ of the form \eqref{eq:musumdiracs}, we define $i(\mu)$, the index of $\mu$, as the number of support points (Diracs) of $\mu$, counting interior points with weight $1$ and boundary points with weight $1/2$. We call $\mu$ ``principal'' if $i(\mu)=(n+1)/2$, ``canonical'' if $i(\mu)=(n+2)/2$, ``upper'' if support points include $1$, ``lower'' if support points do not include $1$.
Then Theorem \ref{thm_principal} asserts that each  $q \in Int(M^{n})$ has a unique
upper and lower principal representation. Since  the volume of $M_n$ is independent of the representation used to compute it, computing that volume with a lower and an upper representation leads to an equality corresponding to classical Selberg identities.

Now let $t_{*}\in (0,1)$. Theorem \ref{thm_canonical} asserts that  every point in the interior of $M^{n}$ has a unique canonical representation
 whose support contains $t_{*}$, and when $t_{*}=0$ or $1$, then there exists a unique principal representation whose support contains $t_{*}$.
Since  the volume of $M_n$,  and the higher order moments of the uniform measure restricted to $M_n$,
 are independent of the representation used to compute them, computing these "moment moments"
  for all possible values of $t^*$ leads to a family of equalities corresponding to new  Selberg integral formulas and Reproducing Kernel Hilbert Spaces.
Consequently, the free parameter $t_*$ associated with Markov and Kre\u{\i}n's  canonical representations of truncated Hausdorff moments (see Section \ref{sec_IntegralGeometry}) which, along with their principal representations, so handily provides us with the means to prove the quantitative Brittleness Theorem \ref{thm_shiva_singlesample}, is found to generate  reproducing kernel identities
 corresponding to  reproducing kernel Hilbert spaces of
 polynomials (see sections \ref{sec_mean} and \ref{sec_RKHS}).
 Furthermore, these reproducing identities lead to biorthogonal systems of
Selberg integral formulas described in Theorems \ref{thm_selberg}, \ref{thm_selberg2} and \ref{thm_selberg_explicit} (see also Corollary \ref{cor_selberg2}).

  Moreover, although not done here, this process can easily be generalized in simple ways. For example,
 the argument is valid using any measure on the moment space, not just the uniform measure, and so
 the introduction of alternatives for which the integrals can likewise be computed, can be used.
In addition, it also appears possible that this process can be repeated with multiple free parameters
 $t_{*,1},\ldots,t_{*,k}$ to obtain even richer classes of (new) Selberg integral formulas.

\section{OUQ with Priors}\label{sec:ouqwithpriors}
To understand  OUQ  one simply starts with \u{C}eby\u{s}ev
  \cite[Pg.~4]{Krein}
`` Given: length, weight, position of the centroid and moment of inertia of a material rod with a density
varying from point to point. It is required to find the most accurate limits for the weight of a certain
segment  of this rod.'' According to  Kre\u{\i}n \cite{Krein}, although \u{C}eby\u{s}ev did solve this problem, it was his student
 Markov
who supplied the proof in his thesis. See Kre\u{\i}n \cite{Krein}  for an account of the
 history of this subject
   along with substantial contributions by Kre\u{\i}n.
We take this mindset and apply it to more complex problems, extending the base space to functions and measures,
and, instead of developing sophisticated mathematical solutions, develop optimization problems and reductions,
 so that their solution may be implemented on a computer,
as in Bertsimas and Popescu's \cite{BertsimasPopescu:2005} convex optimization approach to
\u{C}eby\u{s}ev  inequalities, and the Decision Analysis framework of Smith \cite{Smith}.
In addition to the determination of optimal bounds as a function of available information and assumptions, the OUQ methodology
has the substantial benefit of demanding that different components of an organization work together to
come up with information and assumptions that, together, they believe in.

Let us begin with a  general formulation of OUQ with priors, where the base assumptions are sets of (function, measure) pairs and
the secondary assumptions are sets of
 priors, that is, sets of probability measures defined on the base assumption set. Later, when we apply to Bayesian inference, we will restrict
to a base  assumption set consisting of a set of measures and a secondary assumption consisting of a set
probability measures on the base assumption set.
 To that end,
let $\mathcal{X}$ be a topological space, $\mathcal{M}(\mathcal{X})$ the space of Borel probability measures on $\mathcal{X}$, and let $\mathcal{G} \subseteq \mathcal{F}(\mathcal{X})$
 be a subset  of the real-valued measurable functions $\mathcal{F}
(\mathcal{X})$ on  $\mathcal{X}$.  Let $\mathcal{A}$ be an arbitrary subset of $\mathcal{G} \times \mathcal{M}(\mathcal{X})$, and let $\Phi \colon \mathcal{G} \times \mathcal{M}(\mathcal{X}) \to \R$
be a function producing a quantity of interest.  In the context of uncertainty quantification one is interested in estimating $\Phi(f^\dagger,\mu^\dagger)$, where $(f^\dagger,\mu^\dagger)\in \mathcal{G} \times \mathcal{M}(\mathcal{X})$ corresponds to an \emph{unknown reality}.  If $\mathcal{A}$ represents all that is known about $(f^\dagger,\mu^\dagger)$ (in the sense that $(f^\dagger,\mu^\dagger)\in \mathcal{A}$ and that any $(f, \mu)\in \mathcal{A}$ could, a priori, be $(f^\dagger,\mu^\dagger)$ given the available information) then \cite{OSSMO:2011} shows that the  quantities
\begin{eqnarray}
	\label{eq:defma1}
	\mathcal{U}(\mathcal{A})&:=& \sup_{(f, \mu)\in \mathcal{A}} \Phi(f, \mu)\\
	\label{eq:defma2}
	\mathcal{L}(\mathcal{A})&:=& \inf_{(f, \mu)\in \mathcal{A}} \Phi(f, \mu)
\end{eqnarray}
determine  the inequality
\begin{equation}
\label{ineq_OUQ}
  \mathcal{L}(\mathcal{A}) \leq \Phi(f^\dagger,\mu^\dagger) \leq \mathcal{U}(\mathcal{A})
\end{equation}
to be optimal with respect to the available information (i.e.~$(f^\dagger,\mu^\dagger) \in \mathcal{A}$)  as follows: First, it is simple to see that the inequality \eqref{ineq_OUQ}  follows  from
 $(f^\dagger,\mu^\dagger) \in \mathcal{A}$.
Moreover, for any $\e >0$  there exists a $(f,\mu) \in \mathcal{A}$ such that
\[ \mathcal{U}(\mathcal{A})-\e < \Phi(f,\mu) \leq   \mathcal{U}(\mathcal{A}).\] Consequently since all that we know about
$(f^\dagger,\mu^\dagger)$ is that $(f^\dagger,\mu^\dagger) \in \mathcal{A}$, it follows that the upper bound
$\Phi(f^\dagger,\mu^\dagger) \leq \mathcal{U}(\mathcal{A})$ is the best obtainable given that information. The lower bound is clearly optimal in the same sense.

        A classical example of a quantity of interest is the
 validation and certification quantity $\Phi(f,\mu):=\mu[f\geq a]$ where $a$ is a safety margin. In the certification context one is interested in showing that
$\mu^\dagger[f^{\dagger}\geq a]\leq \epsilon$ where $\epsilon$ is a safety certification threshold (\emph{i.e.}\ the maximum acceptable $\mu^\dagger$-probability of the system $f^{\dagger}$ exceeding the safety margin $a$).  If $\mathcal{U}(\mathcal{A}) \leq \epsilon$, then the system associated with $(f^\dagger,\mu^\dagger)$ is safe (given the information represented by $\mathcal{A}$).  If $\mathcal{L}(\mathcal{A}) > \epsilon$, then the system associated with $(f^\dagger,\mu^\dagger)$ is unsafe.  If $\mathcal{L}(\mathcal{A}) \leq \epsilon < \mathcal{U}(\mathcal{A})$, then the safety of the system cannot be decided without making further assumptions or gathering further information.

Although the OUQ optimization problems \eqref{eq:defma1} and \eqref{eq:defma2} are extremely large, we have shown in \cite{OSSMO:2011} that an important subclass enjoys significant and practical finite-dimensional reduction properties.  In particular, for assumption sets corresponding to
  linear inequality constraints on generalized moments, the search can be reduced to one over probability measures that are products of finite convex combinations of Dirac masses with explicit upper bounds on the number of Dirac masses.

To incorporate priors,
we define a prior $\pi$ to be
a probability measure $\pi \in \mathcal{M}(\mathcal{A})$, and  define the value
$\bar{\Phi}(\pi)$ of $\pi$ through the extended quantity of interest $\bar{\Phi}:\mathcal{M}(\mathcal{A}) \rightarrow \R$
defined by
\begin{equation*}
\label{def_interest_can} \bar{\Phi}(\pi):=\E_{\pi}[\Phi], \quad \pi \in \mathcal{M}(\mathcal{A}).
\end{equation*}
We will defer the nontrivial and not uninteresting topics of measurability to when we  analyze
  the full OUQ with priors framework, but note that
 Ressel \cite{Ressel_some} has established important and relevant results for us already, in particular the measurability of
the validation and certification quantity of interest discussed above under mild conditions.

We call the value $\E_{\pi}[\Phi]$ the
{\em prior value}, and  for a family of priors $\Pi \subset \mathcal{M}(\mathcal{A})$
we note that
 the values
\begin{eqnarray}
        \label{eq:UPi}
        \mathcal{U}(\Pi)& := & \sup_{\pi \in \Pi} \E_{\pi} \big[\Phi\big]\\
        \label{eq:LPi}
        \mathcal{L}(\Pi)& := &\inf_{\pi \in \Pi} \E_{\pi} \big[\Phi \big]
\end{eqnarray}
form a natural generalization of the notations
$\mathcal{U}(\mathcal{A})$ and  $\mathcal{L}(\mathcal{A})$.
Moreover, in the same way that $\mathcal{U}(\mathcal{A})$ and $\mathcal{L}(\mathcal{A})$ are optimal
upper and lower bounds on $\Phi(f^\dagger,\mu^\dagger)$ given the information that
$(f^\dagger,\mu^\dagger) \in \mathcal{A}$,  $\mathcal{U}(\Pi)$ and $\mathcal{L}(\Pi)$ are optimal upper and lower bounds on $\E_{\pi}\big[\Phi\big]$ given the information that $\pi \in \Pi$.

For conditioning on sample data in an observation space $\mathcal{D}$,
we  begin by defining a {\em  data map}
\[\Dmap:\mathcal{A}\rightarrow \mathcal{M}(\mathcal{D})\]
which specifies that $\Dmap(f,\mu)\in \mathcal{M}(\mathcal{D})$ generates the data
when the truth is  $(f,\mu) \in \mathcal{A}$. Then, given a
prior
 $\pi \in \mathcal{M}(\mathcal{A})$, we  define a probability measure
 \[\pi\odot \Dmap \in
 \mathcal{M}\bigl(\mathcal{A} \times \mathcal{D}\bigr)\]
through
\begin{equation*}\label{eq:palm}
	\pi\odot \Dmap \big[ A \times B \big] = \E_{(f.\mu) \sim \pi} \big[ \one_{A}(f, \mu) \Dmap(f, \mu)[B] \big],\quad A \in \mathcal{B}(\mathcal{A}),\, B \in \mathcal{B}(\Dspace)\, ,
\end{equation*}
 where $\one_{A}$ is the indicator function of the set $A$:
\[
	\one_{A}(f,\mu) :=
	\begin{cases}
		1, & \text{ $(f,\mu) \in A$,} \\
		0, & \text{ $(f,\mu) \notin A$.}
	\end{cases}
\]
We denote the resulting $\Dspace$-marginal by $ \pi\cdot\Dmap \in \mathcal{M}(\Dspace)$ which satisfies
\begin{equation*}
        \label{eq:cdotexp2}
        \pi\cdot\Dmap[B]:=\E_{(f,\mu)\sim \pi}\big[\Dmap(f,\mu)[B]\big]\, .
\end{equation*}

Given an observation $d \in \mathcal{D}$, to simultaneously avoid the ill-definedness of regular conditional probabilities and incorporate uncertainty in the observation process, we consider conditioning
on an open subset $B$ containing $d$ such that $\pi\cdot \Dmap[B]>0$. The naturality of this positivity condition is fully discussed in  \cite{BayesOUQ}, in particular it is easy to show
 that if $B$ is  an open ball of center $\delta$ around the data $d$ (noted $B_\delta$) and if the data is randomized and distributed according to $\pi\cdot \Dmap$, then the probability of the event $\pi\cdot \Dmap[B_\delta]>0$ is one.
It is also shown in \cite{BayesOUQ} that if the probability of the data is uniformly bounded, in the Bayesian model class $\mathcal{A}$, from above and below by that of a reference measure (e.g., for all $(f,\mu)\in \mathcal{A}$, $\frac{1}{\alpha }\Dmap(f_0,\mu_0) \leq \Dmap(f,\mu)[B_\delta]\leq \alpha \Dmap(f_0,\mu_0)$ for some reference measure $\Dmap(f_0,\mu_0)$) then learning and robustness appear as antagonistic properties:
when $\alpha=1$, the data is equiprobable under all measures in the model class, posterior values are equal to prior values, the method is robust but learning is not possible, and as $\alpha$ deviates from, learning becomes possible (posterior values depend on the data) but the method becomes increasing brittle (the range of posterior values converges towards that  of the quantity of interest $\Phi$).

  To simplify notation, we henceforth drop the notational dependence
of the set $B$ on the point $d$.
The conditional expectation, given a prior $\pi$ and data map $\Dmap$, conditioned on a subset $ B \in \mathcal{B}(\Dspace)$ such that $\pi\cdot \Dmap[B]>0$, is 
\begin{equation*}\label{eq:defcondexp}
	\E_{\pi \odot \Dmap}\big[ \Phi\big| B\big]=\frac{\E_{(f,\mu)\sim \pi}\big[  \Phi(f,\mu)
 \Dmap(f,\mu)[B] \big]}{\E_{(f,\mu)\sim\pi}\big[\Dmap(f,\mu)[B]\big]}\, .
\end{equation*}

 To represent uncertainty regarding the data generating process, instead of a single
data map $\Dmap:\mathcal{A}\rightarrow \mathcal{M}(\Dspace)$, we instead
 specify a
set \[\Dmaps =\bigl\{\Dmap:\mathcal{A}\rightarrow \mathcal{M}(\Dspace)\bigr\}\]
of data maps
 and represent our assumptions regarding the data
 with the statement
$\Dmap \in \Dmaps$.
Therefore, having specified a set $\Pi$ of priors, and a set $\Dmaps$ of data maps,
 for an open subset $B \subseteq \Dspace$,
we define the set of all possible resulting product measures to be
\begin{equation*}
\label{def_B}
	\Pi\odot_B\Dmaps := \Bigl\{ \pi \odot \Dmap: \pi\in  \Pi,\, \Dmap \in  \Dmaps, \, (\pi \cdot \Dmap)[B]> 0 \Bigr\}\, .
\end{equation*}
The notations $ \mathcal{U}(\Pi)$ and $ \mathcal{L}(\Pi)$  of \eqref{eq:UPi} and \eqref{eq:LPi}
extend naturally to
 these conditional expectations as
\begin{equation*}\label{eq:Upostes}
	\mathcal{U}(\Pi\odot_B\Dmaps):=\sup_{\pi\odot \Dmap \in \Pi\odot_B\Dmaps}\E_{\pi \odot \Dmap}\big[ \Phi\big| B\big]
\end{equation*}
\begin{equation*}\label{eq:Lpostes}
	\mathcal{L}(\Pi\odot_B\Dmaps):=\inf_{\pi\odot \Dmap \in \Pi\odot_B\Dmaps}\E_{\pi \odot \Dmap}\big[ \Phi\big| B\big]\, ,
\end{equation*}
where we note that, just as for $\mathcal{U}(\mathcal{A})$ and $\mathcal{L}(\mathcal{A})$ and
$\mathcal{U}(\Pi)$ and $\mathcal{L}(\Pi)$,
   $\mathcal{U}(\Pi\odot_{B}\Dmaps)$ and
 $\mathcal{L}(\Pi\odot_{B}\Dmaps)$
are optimal upper and lower bounds on the posterior value $\E_{\pi\odot \Dmap}\big[ \Phi\big|  B\big]$, given
 the assumptions that  $\pi\in \Pi$,  $\Dmap \in \Dmaps$, and  $\pi \cdot \Dmap(B)>0$.

We are now prepared to discuss the brittleness theorems of the next section.
Indeed, it is easy to see that
\begin{equation}\label{eq:upilpi}
\mathcal{L}(\mathcal{A})\leq  \mathcal{L}(\Pi)  \leq  \mathcal{U}(\Pi)\leq   \mathcal{U}(\mathcal{A})
\end{equation}
and
\[\mathcal{L}(\mathcal{A})\leq  \mathcal{L}(\Pi\odot_B\Dmaps)  \leq  \mathcal{U}(\Pi\odot_B\Dmaps) \leq  \mathcal{U}(\mathcal{A}) \, .\]
What the Brittleness Theorem \ref{thm_shiva} will show is that, under mild conditions, regardless of where
 the values $\mathcal{L}(\Pi)$ and
$\mathcal{U}(\Pi)$ lie in \eqref{eq:upilpi}
we have
\[  \mathcal{L}(\Pi\odot_B\Dmaps) \approx \mathcal{L}(\mathcal{A}) \text{ and }\mathcal{U}(\Pi\odot_B\Dmaps) \approx \mathcal{U}(\mathcal{A})\,, \]
that is, conditioning on the observed data, one can obtain any value
between $\mathcal{L}(\mathcal{A})$ and $\mathcal{U}(\mathcal{A})$
for the posterior value
$ \E_{\pi \odot \Dmap}\big[ \Phi\big| B\big]$ for some admissible prior $\pi \in\Pi$ and data map
$\Dmap \in \Dmaps$.

\section{Quantification of Bayesian Brittleness}

The following theorem is the Main Brittleness result of \cite[Thm.~4.13]{BayesOUQ}:
\begin{thm}
        \label{thm:shiva0}
        Let $\mathcal{A}$ be a Suslin space, let $\mathcal{Q}$ be a separable and metrizable space, and let $\Psi \colon \mathcal{A} \to \mathcal{Q}$ be measurable.  Moreover, let $\mathfrak{Q} \subseteq \mathcal{M}(\mathcal{Q})$ be such that $\supp(\mathbb{Q}) \subseteq \Psi(\mathcal{A})$ for all $\mathbb{Q} \in \mathfrak{Q}$.  Suppose that, for all $\delta >0$, there exists some $\mathbb{Q}\in \mathfrak{Q}$ such that
        \begin{equation}
                \label{eq:dto0}
                \E_{q\sim \mathbb{Q}} \left[ \inf_{\mu\in  \Psi^{-1}(q)} \Dmap(\mu)[B] \right]=0
        \end{equation}
        and        \begin{equation}
                \label{eq:djkdjehjehj33}
                \mathbb{P}_{q\sim \mathbb{Q}} \left[ \sup_{\mu\in \Psi^{-1}(q),\, \Dmap(\mu)[B]>0}\Phi(\mu) > \sup_{\mu\in \mathcal{A}}\Phi(\mu) - \delta \right]>0 .
        \end{equation}
        Then
        \begin{equation}
                \label{eq:2qbisjhjycondddihjh}
                \mathcal{U}\big(\Psi^{-1}\mathfrak{Q}\big|B\big) =\mathcal{U}(\mathcal{A}) .
        \end{equation}
\end{thm}

The following generalization
of the  Theorem \ref{thm:shiva0} (\cite[Thm.~4.13]{BayesOUQ})  allows a weakening of its assumptions while
approximately obtaining its conclusion.
We require, as in \cite{BayesOUQ}, the data space $\Dspace$ to be metrizable. We
 select a consistent metric, and  for a
data point $d \in \mathcal{D}$,
let $B_{\d}(d)$ denote the open ball of metric radius $\d$ about $d$. To keep the notation simple we omit
reference to the base point $d$ and denote this family of open balls about $d$ by
$\{B_{\d}, \d \geq 0\}$, where $B_{0}=\emptyset$.
\begin{thm}
\label{thm_shiva}
For a  metrizable topological space $\mathcal{X}$, consider a topologized
 subset $\mathcal{G}\subset\mathcal{F}(\mathcal{X})$ and the space of probability measures $\mathcal{M}(\mathcal{X})$ equipped with the weak star topology.
Let $\mathcal{A} \subset \mathcal{G}\times \mathcal{M}(\mathcal{X})$ be Suslin, $\mathcal{Q}$ separable metrizable, and
$\Psi:\mathcal{A} \rightarrow \mathcal{Q}$ Borel measurable. Moreover, let
$\mathfrak{Q} \subset \mathcal{M}(\mathcal{Q})$ be such that
$\supp \, \mathbb{Q} \subset \Psi(\mathcal{A}), \mathbb{Q} \in \mathfrak{Q}$, and let
$\tau \geq 0$. Suppose there exists some
 $\mathbb{Q}\in \mathfrak{Q},\,\Dmap\in \Dmaps$ and  a continuous monotonically increasing function
$h:\R^{+} \rightarrow \R$  with $h(0)=0$ such that
\begin{equation}
\label{eq:dto0app}
 \mathbb{Q}\Bigl(\bigl\{q: \inf_{(f,\mu)\in  \Psi^{-1}(q)} \Dmap(f,\mu)[B_{\d}]\leq \tau\big\}\Bigr) \geq 1-h(\d), \quad \d >0\, .
\end{equation}
Fix $\d>0$.  If  $\e \geq 0$, $\e'>0$ and $\d'>0$ are three  real numbers  such that
\begin{equation}\label{eq:djkdjehjehj33app}
\begin{split}
\mathbb{Q}\Bigl( \Bigl\{ q: \sup_{(f,\mu)\in \Psi^{-1}(q),\, \Dmap(f,\mu)[B_{\d}]> \e}\Phi(f,\mu) >\sup_{(f,\mu)\in \mathcal{A}}\Phi(f,\mu)    -\delta' \Bigr\}\Bigr)\geq \e'
\end{split}
\end{equation}
and
\begin{equation}
\label{eq_solve}
h(\d)+\tau \leq  \frac{\e\d'\e'}{\mathcal{U}(\mathcal{A})-\mathcal{L}(\mathcal{A})}\, ,
\end{equation}
then we have
\begin{equation}
		\label{eq:2qbisjhjycondddihjh}
\mathcal{U}(\mathcal{A}) -2\d' \leq \mathcal{U}\big(\Psi^{-1}(\mathfrak{Q})\odot_{B_{\d}}\Dmaps\big)
\leq \mathcal{U}(\mathcal{A})\, .
\end{equation}
\end{thm}

If, for $\tau=0$, there exists a $\d^{*}>0$ such that
for all $\d'>0$ there exists some  $\mathbb{Q}\in \mathfrak{Q},\,\Dmap\in \Dmaps$
which satisfies \eqref{eq:dto0app} with a function $h$ such that
 $h(\d)=0,\d \leq \d^{*}$, and which satisfies
 \eqref{eq:djkdjehjehj33app}  with
 $\e=0$,
 then we recover the conditions and the  assertion of the Brittleness Theorem \cite[Thm.~4.13]{BayesOUQ}
for $B_{\d}, \d \leq \d^{*}$.

\begin{rmk}
The proof of Theorem \ref{thm_shiva} also leads to the following result. For a  metrizable topological space $\mathcal{X}$, consider a topologized
 subset $\mathcal{G}\subset\mathcal{F}(\mathcal{X})$ and the space of probability measures $\mathcal{M}(\mathcal{X})$ equipped with the weak star topology.
Let $\mathcal{A} \subset \mathcal{G}\times \mathcal{M}(\mathcal{X})$ be Suslin, $\mathcal{Q}$ separable metrizable, and
$\Psi:\mathcal{A} \rightarrow \mathcal{Q}$ Borel measurable. Moreover, let
$\mathfrak{Q} \subset \mathcal{M}(\mathcal{Q})$ be such that
$\supp \, \mathbb{Q} \subset \Psi(\mathcal{A}), \mathbb{Q} \in \mathfrak{Q}$. It holds true that for $\delta>0$
\begin{equation}
		\label{eq:2qbisjhjycondddihjhtaustar}
\mathcal{U}(\mathcal{A}) - v(\d) \leq \mathcal{U}\big(\Psi^{-1}(\mathfrak{Q})\odot_{B_{\d}}\Dmaps\big)
\leq \mathcal{U}(\mathcal{A})\, .
\end{equation}
where the function $v$ is defined by
\begin{equation}
\label{eq_solvetaustar}
\begin{split}
v(\d):=& 2 \inf \Bigg\{\delta'>0\Bigg| \delta'\geq  \big(\mathcal{U}(\mathcal{A})-\mathcal{L}(\mathcal{A})\big) \inf_{\mathbb{Q}\in \mathfrak{Q},\,\Dmap\in \Dmaps,\e> 0,\tau\geq 0 } \\&\frac{1-\mathbb{Q}\Bigl(\bigl\{q: \inf_{(f,\mu)\in  \Psi^{-1}(q)} \Dmap(f,\mu)[B_{\d}]\leq \tau\big\}\Bigr)+\tau}{\e \mathbb{Q}\Bigl( \Bigl\{ q: \sup_{(f,\mu)\in \Psi^{-1}(q),\, \Dmap(f,\mu)[B_{\d}]> \e}\Phi(f,\mu) >\sup_{(f,\mu)\in \mathcal{A}}\Phi(f,\mu)    -\delta' \Bigr\}\Bigr)}\Bigg\}\,
\end{split}\end{equation}
for $\d >0$.
\end{rmk}

\begin{rmk}
This brittleness is not a consequence of a  lack of compactness of the admissible set. Indeed,
in the following section, the primary space of measures  $\mathcal{M}(I)$ is compact in the weak topology, as is any closed moment subset, and  Theorem \ref{thm_shiva_singlesample} describes a brittleness result.
\end{rmk}
\begin{rmk}
It is true that this brittleness does not appear to be primarily due to the Bayesian methodology, but is valid more generally.
 See
Bahadur and Savage \cite{BahadurSavage} and Donoho \cite{Donoho} for similar results for statistical estimators, where it appears that the mechanism generating  the instability is analogous to that investigated here.
\end{rmk}

\subsection{Application to a revealing example}
To demonstrate that the assumptions of Theorem  \ref{thm_shiva} are mild, we now
use it to  extend the Brittleness result of \cite[Ex.~4.16]{BayesOUQ} to a simple but informative example.
Here one is interested in estimating the mean of a random variable $X$ with unknown distribution on the unit interval $I:=[0,1]$.
Since our quantity of interest is $\E_{\mu^\dagger}[X]$, where $\mu^\dagger$ is an unknown distribution on $I$,
 in the notations of Section \ref{sec:ouqwithpriors}, we have  $\mathcal{X}:=I$ (since $X$ is a random variable on $I$), $\mathcal{G}$ consists only
of the identity function (this example does not involve unknown functions of $X$), $\mathcal{A}:= \mathcal{M}(I)$ (the set of possible/admissible candidates for $\mu^\dagger$ is the set of all probability distributions $\mu$ on $I$), $\Phi(\mu):=\E_{t \sim \mu}[t]$ (our quantity of interest is the mean of the random variable $X$), $\mathcal{Q}:=\R^{n}$ and the map
 $\Psi:\mathcal{M}(I) \rightarrow \R^{n}$ is the map to the truncated Hausdorff moments
$\Psi(\mu) := \bigl(\E_{t \sim \mu}[t^{i}]\bigr)_{i=1,..,n}$  (our set of prior distributions is defined  by constraining the distribution of the
first $n$ Hausdorff moments in $\R^n$, for some fixed  $n$). Furthermore $\mathbb{Q}$ is the uniform
Borel measure on $\R^{n}$ restricted to the Hausdorff moment space $M^{n}:=\Psi(\mathcal{M}(I))$ and then normalized to be a probability measure, that is $\Pi \subset \mathcal{M}\bigl(\mathcal{M}(I)\bigr)$ is the set of prior distributions on $\mathcal{A}=\mathcal{M}(I)$ such that $\Psi \mu\in \mathcal{M}(M^{n})$ is uniformly  distributed on the the space
$M^{n}$ of first $n$ Hausdorff moments.

The Brittleness Theorem \ref{thm_shiva} implies (see \cite[Ex.~4.16]{BayesOUQ}) that if we observe (condition on) $k$  independent samples from $X$, i.e.  $\Dspace:=I^k$ and $\Dmap^{k}\mu:=\mu\otimes \cdots \otimes \mu$ ($k$-fold tensorization) and  $B$ is the $k$-fold product of small enough balls centered on the data then $\mathcal{L}(\Pi\odot_B\Dmaps) \approx \mathcal{L}(\mathcal{A})$ and $\mathcal{U}(\Pi\odot_B\Dmaps) \approx \mathcal{U}(\mathcal{A})$. In other words, although the set of prior values of $\E_{\mu}[X]$ is  the single point $\{\frac{1}{2}\}$, the optimal bounds on the posterior values of $\E_{\mu}[X]$ are zero and one irrespective of the number $n$ of constraints on marginals and the number $k$ of observed samples if the data is observed with sufficient precision.

The following theorem provides a rigorous and quantitative statement and proof of this implication for $k=1$. Although, for the sake of conciseness and clarity our analysis is provided in the $k=1$ case, it generalizes to the situation where $k$ is arbitrary. Indeed, although counterintuitive, one can show that brittleness for the single sample case is {\em more} difficult to
 obtain
 than for multiple samples. Since our main objective here is to unwrap and scrutinize the mechanism causing brittleness in Bayesian inference, we therefore chose to keep the presentation and our example as clear, concise, and simple as possible to illustrate the generic and pervasive nature of this brittleness.

Therefore, we will now  (i) consider the case of a single data point, i.e., $k=1$, $\Dspace:=I$,  and $\Dmap^{1}\mu:=\mu$ (ii) use  Theorem \ref{thm_shiva} to provide quantitative bounds on $\mathcal{U}(\Pi\odot_B\Dmaps)$ as a function $n$ of the number of marginal constraints defining the set of priors (iii) scrutinize the brittleness causing mechanism through the proof of the following theorem.

\begin{thm}
\label{thm_shiva_singlesample}
Let $\mathcal{A}:= \mathcal{M}(I)$, $\Phi(\mu)=\E_{t\sim \mu}[t]$, $\mathcal{D}=I$,  and
$\Psi:\mathcal{M}(I) \rightarrow \R^{n}$
denote the map to the truncated Hausdorff moments $\Psi(\mu)= \bigl(\E_{t \sim \mu}[t^{i}]\bigr)_{i=1,..,n}$.
Furthermore, let
$\mathbb{Q}$ denote the uniform
Borel measure on $\R^{n}$ restricted to the Hausdorff moment space $M^{n}:=\Psi(\mathcal{M}(I))$ and then normalized to be a probability measure.
Suppose that $\mathbb{Q} \in \mathfrak{Q}$ and $\Dmap^{1} \in \Dmaps$.
Then for $\delta>0$
we have
\begin{equation}
\label{shiva_single_sample_result}
1-4 e\big(\frac{2n \d}{e}\big)^\frac{1}{2n+1} \leq \mathcal{U}\big(\Psi^{-1}(\mathfrak{Q})\odot_{B_{\d}}\Dmaps\big) \leq 1
\end{equation}
\end{thm}
\begin{rmk}
Alternatively, Theorem \ref{thm_shiva_singlesample} asserts that for positive $\d,\d'$  satisfying
\[ \d \leq \frac{1}{4n} \bigl(\d'\bigr)^{2n+1}\bigl(2e\bigr)^{-2n}\]
we have
\begin{equation}
\label{shiva_single_sample_resultold}
1-2\d' \leq \mathcal{U}\big(\Psi^{-1}(\mathfrak{Q})\odot_{B_{\d}}\Dmaps\big) \leq 1 \, .
\end{equation}
\end{rmk}

\section{Volume Inequalities on the Hausdorff Moment Space}
\label{sec_volumeinequalities}
Karlin and Shapley
  \cite[Thm.~15.2]{KarlinShapley} (see also  \cite[Thm.~6.2]{KarlinStudden:1966})  computed the volume
 of the space of truncated Hausdorff moments $M^{n}$ of probability measures on the
unit interval to be
\begin{equation}
\label{karlinshapley}
 Vol\bigl(M^{n}\bigr)=\prod_{k=1}^{n}{\frac{\Gamma(k)\Gamma(k)}{\Gamma(2k)}}\, ,
\end{equation}
where $\Gamma$ is the Gamma function. To accomplish this, they used a Markov representation of truncated moment points, as described in
 Kre\u{\i}n \cite{Krein} (see also  \cite[Ch.~II]{KarlinStudden:1966}), combined with the change of variables formula, followed by the evaluation of a Selberg integral.

Here we will refine their analysis to obtain
 volume {\em inequalities} on the Hausdorff moment
space which are used in the application of the Brittleness Theorem \ref{thm_shiva}
to the proof of the  Brittleness Theorem \ref{thm_shiva_singlesample}.
Of the two main results, it is interesting to note
that the Mass Supremum Equality uses  the canonical representation of  moment points combined with Markov's
Maximal Mass Theorem  \cite[Thm.~2.1]{Krein} (see also \cite[Thm.~4.1]{KarlinStudden:1966}) to change the
``Inequality'' to ``Equality'', whereas
 the
Mass Infimum Inequality instead  uses the principal representation, as in Karlin and Shapley's proof of
the volume formula
\eqref{karlinshapley}. All this terminology will be defined in the following Section \ref{sec_IntegralGeometry} and comes from
Karlin and Studden \cite{KarlinStudden:1966}. This section will simply state the volume inequalities
 that we need for Theorem
\ref{thm_shiva_singlesample}.

To proceed, let us now fix terminology.
Let $I:=[0,1]$, and let $\mathcal{P}(I)$ be the set of Borel measures on $I$ and
 $\mathcal{M}(I)\subset \mathcal{P}(I)$ be the set of probability measures. Throughout we will assume the weak star topology for these measures. For the system of functions
\[  u_{i}(t):=t^{i}, t \in I, i=0,..,n\] the Hausdorff moments of a measure $\mu \in \mathcal{P}(I)$ is defined
as the vector $q \in \R^{n+1}$ with coordinates
$q_{i}=\E_{\mu}[u_{i}]=\E_{t \sim \mu}[t^{i}]$. It  is well known (see e.g.~\cite[Cor.~15.7]{AliprantisBorder:2006}) that the map
\[\Psi:\mathcal{P}(I)  \rightarrow \R^{n+1}\] defined by
$\Psi(\mu) := \bigl(\E_{t \sim \mu}[t^{i}], i=0,..,n\bigr)$ is affine and continuous.
Furthermore, let the Hausdorff moment space $\mathcal{M}^{n+1}\subset \R^{n+1}$ be the image $\mathcal{M}^{n+1}:=\Psi \mathcal{P}(I)$
 of the  measures, and let
$M^{n}$ defined by $\mathcal{M}^{n+1}=(1,M^{n})$ be  moments of a probability measures
 omitting the zero-th moment. Equivalently, let $P_{1}: \R \times \R^{n} \rightarrow \R^{n}$ denote the projection mapping $(x_0,x_1,\ldots,x_n)$ onto $(x_1,\ldots,x_n)$ and let $\Psi_{1}=:P_{1}\Psi$. Then
$M^{n}=\Psi_{1}\mathcal{M}(I)$. We will abuse notation by letting $\Psi$ also
denote the mapping $\Psi_{1}$ restricted to the first-to-$n$-th order  moments  of the probability measures
\[\Psi:\mathcal{M}(I)  \rightarrow M^{n} \subset \R^{n}\,,\]
and, for $q \in  \R^{n}$,  let
\[\Psi^{-1}q:=\bigl\{\mu \in \mathcal{M}(I): \Psi\mu=q\bigr\}\]
denote its set-valued inverse.

 It follows from continuity that
the moment set  $\mathcal{M}^{n+1}$ is a closed convex cone and  $M^{n}$  is a compact convex set.
Moreover, one can show that
\begin{equation}
\label{interior}
Int(\mathcal{M}^{n+1}) \cap (1,\R^{n})=(1, Int(M^{n})\bigr),
\end{equation}
see e.g.~\cite[Cor.~6.5.1]{Rockafellar}, so that a point $q$ is interior to $M^{n}$ if and only
if $(1,q)$ is interior to $\mathcal{M}^{n+1}$.
Let $Vol$ be the usual $n$-dimensional volume measure.  Then, since $M^{n}$ is convex, by  \cite[Lem.~1.8.1] {Bogachev_Gauss}
\begin{equation}
\label{boundarynull}
 Vol\bigl(Int(M^{n})\bigr)=Vol\bigl(M^{n}\bigr)\,  .
\end{equation}

Our first result is the Mass Supremum Equality.
\begin{lem}
\label{lem_mass_sup}
Let $t_{*} \in I$, $ 0 \leq \e \leq 1$,  and consider the set
$M^{n}_{\e} \subset M^{n}$ defined by
\[ M^{n}_{\e}:=\bigl\{q \in  M^{n}: \exists \mu \in\Psi^{-1}q:  \mu(\{t_{*}\})\geq \e.\bigr\}\]
Then we have
\[Vol\bigl(M^{n}_{\e}\bigr)=(1- \e)^{n}Vol\bigl(M^{n}\bigr)\, .
 \]
\end{lem}
\begin{rmk}
Lemma \ref{lem_mass_sup} is valid for any system $u_{i}:I \rightarrow \R, i=1,..,n$ of moment
functions
which form a T-system per \cite{KarlinStudden:1966}.
\end{rmk}

The second is  the Mass Infimum Inequality.
\begin{lem}
\label{lem_mass_inf}
Let $t_{*}\in I$,  $\d >0$, and consider the set
$M^{n}_{\d} \subset M^{n}$ defined by
\[ M^{n}_{\d}:=\bigl\{q \in  M^{n}: \exists \mu \in\Psi^{-1}q:  \mu\bigl(B_{\d}(t_{*})\bigr)=0.\bigr\}\]
Then we have
\[Vol\bigl(M^{n}_{\d}\bigr) \geq \Bigl(1- \d (2e)^{2n} \Bigr) Vol\bigl(M^{n}\bigr) \, .\]
\end{lem}

The third is the Mass of First Moment Inequality.
\begin{lem}
\label{massfirstmoment}
Let  $ 0 \leq \d \leq \frac{1}{2}$. Then we have
\[\d^{n} \bigl(2e\bigr)^{n} \geq \frac{Vol\bigl( q \in M^{n}: q_{1} \in [1-\d,1]\bigr)}{Vol(M^{n})} \geq \d^{n}\, .  \]
\end{lem}

\section{Integral Geometry of the Markov-Kre\u{\i}n Representations}
\label{sec_IntegralGeometry}
Here we will describe the Markov-Kre\u{\i}n representations of truncated moments and begin the development of their integral geometry. The history of this subject begins with \u{C}eby\u{s}ev and his student Markov's thesis,  followed by
work by Kre\u{\i}n and others,  where in
   \cite{Krein} one can find, not only an historical sketch, but substantial contributions by Kre\u{\i}n. Indeed,
it is clear from Karlin and Studden \cite{KarlinStudden:1966} that this subject owes a lot to Kre\u{\i}n.
Consequently, we refer to the (principal and canonical) representations that we use as
Markov-Kre\u{\i}n representations. It can be argued that the appropriate name should be
\u{C}eby\u{s}ev-Markov-Kre\u{\i}n representations but this name is too long and so we implicitly give credit to
\u{C}eby\u{s}ev.

Now, following Karlin and Studden  \cite[Chapters II \& IV]{KarlinStudden:1966},
 we describe the Markov-Kre\u{\i}n representations and determine their Jacobian determinants.
  We  finish this section
by setting up the change of variables approach, in preparation for
both the  proofs of the volume inequalities of Section \ref{sec_volumeinequalities} and  all that follows.
  To wit,
we define the index $i(t)$ of a strictly increasing  set $t$ of points  $0 \leq t_{1}  < t_{2} < \cdots < t_{N} \leq 1$
by counting the interior points with weight $1$ and boundary points  with weight $\frac{1}{2}$.
For a point $q \in M^{n}$ we say that a measure $\mu \in \mathcal{M}(I)$ is a representing measure for
$q$ if $\Psi(\mu)=q$ and
it is a weighted  sum of Dirac masses
\[
 \mu=\sum_{j=1}^{N}{\lambda_{j}\delta_{t_{j}}}, \quad \lambda_{j} >0, j=1,..,N
\]
for a strictly increasing set of points  $0 \leq t_{1}  < t_{2} < \cdots < t_{N} \leq 1$.
In that case, we have the formula
\[q_{i}= \bigl(\Psi(\mu)\bigr)_{i}=\sum_{j=1}^{N}{\lambda_{j}t_{j}^{i}}\, .\]
The index $i(\mu)$ of such a representing measure is defined to be the index $i(t)$ of its set of support points.

 A representing measure $\mu$ is called {\em principal} if $i(\mu)= \frac{n+1}{2}$ and {\em canonical} if
$i(\mu) \leq \frac{n+2}{2}$.
For $q \in Int(M^{n})$, \cite[Thm.~2.1]{KarlinStudden:1966} asserts
that $i(\mu)\geq \frac{n+1}{2}$ for any representing measure $\mu$ for $q$.
  A principal or canonical representation is called ``lower'' if its set of support points {\em does not} include the righthand endpoint $1$ and ``upper'' if it does.
The following two results will be our main tools.
The first is the  principal representation, see \cite[Cor.~3.1]{KarlinStudden:1966}.
\begin{thm}
\label{thm_principal}
Every point $q \in Int(M^{n})$ has a unique
upper and lower principal representation.
\end{thm}
The second is  the canonical representation which allows the specification of a predetermined  point  $t_{*}\in I$ in the support of the representing measure, see \cite[Thm.~3.1]{KarlinStudden:1966} combined with  \cite[Cor.~3.2]{KarlinStudden:1966} and
\cite[Cor.~3.1]{KarlinStudden:1966}.
\begin{thm}
\label{thm_canonical}
For $t_{*}\in (0,1)$,
  every point $q \in Int(M^{n})$ has a unique canonical representation
 whose support contains
$t_{*}$.
When $t_{*}=0$ or $1$, there exists a unique principal
representation whose support contains $t_{*}$.
\end{thm}
What Theorem \ref{thm_canonical} doesn't make clear is if the canonical representations converge to
these principal representation as $t_{*}$ tends to $0$ or $1$. They indeed do as we will see.
Let us define some notation that we will use henceforward. We consider two coordinate representations of the
interior of the regular unit simplex. In particular, let
\[T^{N}=\{(t_{1},..,t_{N}):  0 < t_{1} < t_{2} < \cdots < t_{N} < 1\}\]
denote the set of strictly increasing sequences of length $N$ in the interior to $I$ and
\[\Lambda^{N}=\{(\lambda_{1},..,\lambda_{N}):  \lambda_{j}>0,\, j=1,..,N,\, \sum_{j=1}^{N}{\lambda_{j}}<1\}\, .
\]
denote the interior to the positive orthant restricted to $\lambda\cdot \eins < 1$.
Sometimes it will be convenient to abuse this notation and  shift indices so that
\[\Lambda^{N}=\{(\lambda_{0},..,\lambda_{N-1}):  \lambda_{j}>0,\, j=0,..,N-1,\, \sum_{j=0}^{N-1}{\lambda_{j}}<1\}\, .
\]
We will often use the fact that
$I^{N}$ can be described by $N!$ copies of $T^{N}$ corresponding to permuting the sequence.

We use the notation $t$ for a vector with coordinates $t_{j}$ and similarly
$ \lambda$ for a vector with coordinates $\lambda_{j}$.
We use the superscripts $p$ for ``principal'' and $c$ for ``canonical'',
the subscripts $o$ for ``odd'', $e$ for ``even'',
 $l$ for ``lower'', and $u$ for ``upper''.
Finally, we purposefully ignore
multiples of $\pm 1$ in all our determinant calculations. With proper caution, this causes no harm since at the end of the day
we take the absolute value.

\subsection{Principal Representations}
\label{sec_principal}
Theorem \ref{thm_principal} asserts that each  $q \in Int(M^{n})$ has a unique
upper and lower principal representation. We now define these representations as maps and compute their Jacobian determinants. We state these propositions without proof, since these proofs are very similar to those for the
canonical representations of
Propositions and \ref{prop_Jco} and \ref{prop_Jce}.

First consider the odd case when $n=2m-1$.
 Then since $\frac{n+1}{2}=m$ is an integer, it follows that the support of any lower principal representation  contains neither endpoint and the support of any upper principal representation contains both endpoints. Consequently,
Theorem \ref{thm_principal} implies that
each point  in $Int(M^{2m-1})$ has a unique lower principal representation of the form
\begin{equation}
\label{rep_principal_odd}
 \mu=\sum_{j=1}^{m}{\lambda_{j}\delta_{t_{j}}}, \quad \lambda_{j} >0,\, j=1,..,m, \,\, \sum_{j=1}^{m}{\lambda_{j}}=1
\end{equation}
where $0 < t_{1} < t_{2} < \cdots < t_{m} < 1$.
Therefore,
 consider the bijection
\[ \phi^{p}_{ol}: \Lambda^{m-1} \times T^{m} \rightarrow Int(M^{2m-1})\]  defined by
\begin{eqnarray}
\label{def_phi_pol}
 \phi^{p}_{ol}(\lambda,t)
&= &\Psi  \Bigl( \sum_{j=1}^{m-1}{\lambda_{j}\delta_{t_{j}}} +(1-\sum_{j=1}^{m-1}{\lambda_{j}})\delta_{t_{m}}\Bigr)\notag\\
&=& \Bigl(\sum_{j=1}^{m-1}{\lambda_{j}t_{j}^{i}}+ (1-\sum_{j=1}^{m-1}{\lambda_{j})t_{m}^{i}}\Bigr)_{i=1}^{2m-1}\, .
\end{eqnarray}
It also follows that
each point in $Int(M^{2m-1})$  has a unique upper principal representation of the form
\begin{equation}
\label{rep_principal_odd_u}
 \mu=\lambda_{0}\delta_{0}+\sum_{j=1}^{m-1}{\lambda_{j}\delta_{t_{j}}} +\lambda_{m}\delta_{1},
 \quad \lambda_{j} >0,\, j=0,..,m, \,\, \sum_{j=0}^{m}{\lambda_{j}}=1
\end{equation}
where $0 < t_{1} < t_{2} < \cdots < t_{m-1} < 1$.
Therefore,
 consider the bijection
\[ \phi^{p}_{ou}: \Lambda^{m} \times T^{m-1} \rightarrow Int(M^{2m-1})\]  defined by
\begin{eqnarray}
\label{def_phi_pou}
 \phi^{p}_{ou}(\lambda,t)
&= &\Psi  \Bigl(\lambda_{0}\delta_{0}+ \sum_{j=1}^{m-1}{\lambda_{j}\delta_{t_{j}}} +
(1-\sum_{j=0}^{m-1}{\lambda_{j}})\delta_{1}\Bigr)\notag\\
&=& \Bigl(\sum_{j=1}^{m-1}{\lambda_{j}t_{j}^{i}}+ (1-\sum_{j=0}^{m-1}{\lambda_{j})}\Bigr)_{i=1}^{2m-1}\, .
\end{eqnarray}

For an increasing sequence $t_{j} < t_{j+1}$ let
\begin{equation}
\label{def_vandermonde}
\Delta(t):=\prod_{ j<k }{(t_{k}-t_{j})}
\end{equation}
denote the Vandermonde determinant
 (see e.g.~\cite[Pg.~400]{HornJohnson2})
  of the matrix with entries $[t_{j}^{i}],\, j=1,..,N,\, i=0,..,N-1$.
which we write as $\Delta_{N}$ to emphasize the dimension of $t$.
We will also use the same formula for non-increasing sequences when we eventually take the absolute value.

\begin{prop}
\label{prop_Jpo}
When $n=2m-1$, the Jacobian determinants are
\[ |det(d\phi^{p}_{ol})|(\lambda,t)= \mathcal{J}^{p}_{ol}(t)
\bigl(1-\sum_{j=1}^{m-1}{\lambda_{j}}\bigr)\prod_{j=1}^{m-1}{\lambda_{j}} \]
\[|det(d\phi^{p}_{ou})|(\lambda,t)= \mathcal{J}^{p}_{ou}(t)
\prod_{j=1}^{m-1}{\lambda_{j}} \]
where
\[\mathcal{J}^{p}_{ol}(t):=
 \Delta^{4}_{m}(t)
	\]
\[\mathcal{J}^{p}_{ou}(t):=\prod_{j=1}^{m-1}{t_{j}^{2}(1-t_{j})^{2}}
 \cdot \Delta_{m-1}^{4}(t)\, .\]
\end{prop}
Note that although each term appears to have the same multiplier $\prod_{j=1}^{m-1}{\lambda_{j}}$,
in the lower case
this multiplier is the full product in
 on $\Lambda^{m-1}$ and in the upper case it is only a partial product on $\Lambda^{m}$, that is, it is missing
the $\lambda_{0}$ term.
Finally, let us observe the symmetries under the reflection $t \mapsto 1-t$:
\begin{eqnarray}
\label{Jposymm}
\mathcal{J}^{p}_{ol}(1-t)&=&\mathcal{J}^{p}_{ol}(t)\notag\\
\mathcal{J}^{p}_{ou}(1-t)&=&\mathcal{J}^{p}_{ou}(t)\, .
\end{eqnarray}

Now consider the even case when $n=2m$. Since $\frac{n+1}{2}=m+\frac{1}{2}$ is an integer plus $\frac{1}{2}$
it follows that the support of any lower principal representation
contains the left endpoint but not the right and any upper principal representation  contains the right endpoint
but not the left. Let us first consider the lower representation.
Theorem \ref{thm_principal} implies that
 every point in the interior  $Int(M^{2m})$ has a unique lower principal representation of the form
\[\mu =\sum_{j=1}^{m}{\lambda_{j}\delta_{t_{j}}}+(1-\sum_{j=1}^{m}{\lambda_{j}})\delta_{0},
 \quad \lambda_{j} >0,\, j=1,..,m,\,\, \sum_{j=1}^{m}{\lambda_{j}} <1\]
where $0 < t_{1} < \cdots t_{m} < 1$.
Therefore, consider the bijection
\[ \phi^{p}_{el}:\Lambda^{m}\times T^{m} \rightarrow Int( M^{2m})\] defined by
\begin{eqnarray}
\label{def_phi_pel}
 \phi^{p}_{el}(\lambda,t)
&= &\Psi  \Bigl(\sum_{j=1}^{m}{\lambda_{j}\delta_{t_{j}}}+(1-\sum_{j=1}^{m}{\lambda_{j}})\delta_{0}\Bigr)\\
&=& \Bigl(\sum_{j=1}^{m}{\lambda_{j}t_{j}^{i}}\Bigr)_{i=1}^{2m}\, .
\end{eqnarray}
On the other hand,
 every point in the interior  $Int(M^{2m})$ has a unique upper principal representation of the form
\[\mu =\sum_{j=1}^{m}{\lambda_{j}\delta_{t_{j}}}+(1-\sum_{j=1}^{m}{\lambda_{j}})\delta_{1},
 \quad \lambda_{j} >0,\, j=1,..,m,\,\, \sum_{j=1}^{m}{\lambda_{j}} <1\]
where $0 < t_{1} < \cdots t_{m} < 1$.
Therefore, consider the bijection
\[ \phi^{p}_{eu}:\Lambda^{m}\times T^{m} \rightarrow Int( M^{2m})\] defined by
\begin{eqnarray}
\label{def_phi_peu}
 \phi^{p}_{eu}(\lambda,t)
&= &\Psi  \Bigl(\sum_{j=1}^{m}{\lambda_{j}\delta_{t_{j}}}+(1-\sum_{j=1}^{m}{\lambda_{j}})\delta_{1}\Bigr)\\
&=& \Bigl(\sum_{j=1}^{m}{\lambda_{j}t_{j}^{i}}+ (1-\sum_{j=1}^{m}{\lambda_{j})}\Bigr)_{i=1}^{2m}\, .
\end{eqnarray}
\begin{prop}
\label{prop_Jpe}
When $n=2m$ the Jacobian determinants are
\[ |det(d\phi^{p}_{el})|(\lambda,t)= \mathcal{J}^{p}_{el}(t)\prod_{j=1}^{m}{\lambda_{j}}
 \]
\[ |det(d\phi^{p}_{eu})|(\lambda,t)= \mathcal{J}^{p}_{eu}(t)
\prod_{j=1}^{m}{\lambda_{j}} \]
where
\[
\mathcal{J}^{p}_{el}(t)
:= \prod_{j=1}^{m}{t_{j}^{2}}
\cdot \Delta_{m}^{4}(t)
\]
\[
\mathcal{J}^{p}_{eu}(t)
:= \prod_{j=1}^{m}{(1-t_{j})^{2}}\cdot \Delta_{m}^{4}(t)
\]
\end{prop}
Here, instead of the reflection  $t \mapsto 1-t$ leaving the lower and upper invariant as in the  odd case \eqref{Jposymm},
 reflection swaps lower and upper;
\begin{equation}
\label{Jpesymm}
\mathcal{J}^{p}_{el}(1-t)=\mathcal{J}^{p}_{eu}(t)\, .
\end{equation}

\subsection{Canonical Representations}
\label{sec_canonical}
Theorem \ref{thm_canonical} asserts that, when  $t_{*}\in (0,1)$,
every point in  $Int(M^{n})$ has a unique canonical representation whose support contains
$t_{*}$, and when $t_{*}\in \{0,1\}$, it has a unique principal representation whose support contains
$t_{*}$.   Therefore, every point in $Int(M^{n})$ has
   a unique representing measure $\mu$  such that
\begin{equation}
\label{canonicalrep}
 \mu=\sum_{j=1}^{N}{\lambda_{j}\delta_{t_{j}}}, \quad \lambda_{j} >0,\, j=1,..,N, \quad \sum_{j=1}^{N}{\lambda_{j}}=1
\end{equation}
 such that the sequence $0 \leq t_{1}  < t_{2} < \cdots < t_{N} \leq 1$ contains
$t_{*}$, where for $t_{*}\in (0,1)$,  the sequence  has
index $\frac{n+1}{2}$ or $\frac{n+2}{2}$, and when
 $t_{*}=0$ or $1$,
the index is $\frac{n+1}{2}$.
Now let us remove $t_{*}$ from the list and use the identity $\sum_{j=1}^{N}{\lambda_{j}}=1$ to solve
for the weight $\lambda_{t_{*}}$ corresponding to $t_{*}$. Changing notation from  $N \mapsto N+1$  and relabeling the indices, we
 obtain that
\begin{equation}
\label{canonicalrep5}
 \mu=\sum_{j=1}^{N}{\lambda_{j}\delta_{t_{j}}} +\bigl(1-\sum_{j=1}^{N}{\lambda_{j}}\bigr)\delta_{t_{*}},
 \quad \lambda \in \Lambda^{N}\, ,
\end{equation}
where the resulting sequence
\[0 \leq t_{1}  < t_{2} < \cdots < t_{N} \leq 1\]  {\em does not} contain
$t_{*}$, and when
$t_{*}\in (0,1)$, the removal of this interior point implies that the resulting sequence has
  index $\frac{n-1}{2}$ or $\frac{n}{2}$ and when
$t_{*}=0$ or $1$,  the removal of this boundary point implies that the resulting sequence has
index $\frac{n}{2}$.

Consequently, for $t_{*}\in (0,1)$,  to represent $Int(M^{n})$
 we can split
 into four domains, two corresponding to the two ways
of producing index $\frac{n-1}{2}$ and two corresponding the two ways of producing index $\frac{n}{2}$.
When $n$ is even one of the two index $\frac{n-1}{2}$ configurations corresponds to including
$t=0$ in the sequence and not $t=1$ and the other  corresponds to including
$t=1$ in the sequence and not $t=0$, while one of  the two
 index $\frac{n}{2}$ configurations corresponds to not allowing $t=0$ or $t=1$  and
the other corresponds to including both $t=0$ and $t=1$. When $n$ is odd this relationships is reversed.
Similarly, when $t_{*}\in \{0,1\}$, we can can split
 into two domains corresponding to  the two ways
of producing index $\frac{n}{2}$.

However, we can show that the representations of index $\frac{n-1}{2}$ produce {\em zero} volume and so can be
excluded from the integral analysis.  To that end, we only need to consider the $t_{*}\in (0,1)$ case.
 Then let us
decompose the set of sequences of index $\bar{N}$  by their endpoint configurations. That is, split
such sequences into those which contain $0$ but not $1$,  $1$ but not $0$, $0$ and $1$, and neither
$0$ or $1$. Some of these components will be empty. On any of these endpoint specific subdomains
let \[\mathcal{I} \subset \{1,..,N\}\]  denote the  indices of the interior points,
so that in this notation we have
\begin{eqnarray*}
 \mu&=&\sum_{j\in \mathcal{I}}{\lambda_{j}\delta_{t_{j}}} +\sum_{j\in  \mathcal{I}^{c}}{ \lambda_{j}}\delta_{t_{j}} +
 \bigl(1-\sum_{j=1}^{N}{\lambda_{j}}\bigr)\delta_{t_{*}}\, ,\quad
\lambda \in \Lambda^{N} . \notag
\end{eqnarray*}
Moreover, for a sequence $t$, let $\mathring{t}$ denote the sequence of interior points, and
 define  $\mathring{T}_{*}:=\{\mathring{t}: t \in T^{N}, t_{j} \neq t_{*},\, j=1,..,N \}$  to be the set of interior points which do not cover $t_{*}$ and
  consider the map
\[ \phi:\Lambda^{N} \times \mathring{T}_{*}  \rightarrow Int(M^{n})\]
defined by
\begin{eqnarray}
\label{def_phi}
 \phi(\lambda,\mathring{t})
&= &\Psi\Bigl(\sum_{j\in \mathcal{I}}{\lambda_{j}\delta_{t_{j}}} +\sum_{j\in  \mathcal{I}^{c}}{ \lambda_{j}}\delta_{t_{j}} +
 \bigl(1-\sum_{j=1}^{N}{\lambda_{j}}\bigr)\delta_{t_{*}}\Bigr)\notag\\
&=& \Bigl(\sum_{j\in \mathcal{I}}{\lambda_{j}t_{j}^{i}}+\sum_{j\in  \mathcal{I}^{c}}{ \lambda_{j}}t_{j}^{i} + (1-\sum_{j=1}^{N}{\lambda_{j})t_{*}^{i}}\Bigr)_{i=1}^{n}\, \notag
\end{eqnarray}
where we note that the first sum $\sum_{j\in\mathcal{I}}{\lambda_{j}t_{j}^{i}}$ is over the interior points
and the second $\sum_{j\in \mathcal{I}^{c}}{\lambda_{j}t_{j}^{i}}$ over the endpoints which are fixed.

The dimension of the domain $\Lambda^{N} \times \mathring{T}_{*}$ is clearly
 $N+|\mathcal{I}|$. However, one can easily show that
\[N+|\mathcal{I}|=2\bar{N},\]
 so that in the case $\bar{N}=\frac{n-1}{2}$, it follow that the dimension of this subdomain
is $N+|\mathcal{I}|=2\bar{N}=n-1 < n$. Consequently, the image of this subdomain under the map $\phi$ has zero volume
in $M^{n}$. Since the domain corresponding to index $\frac{n-1}{2}$ is a disjoint union of two such
subdomains, the assertion is proved.
Moreover, the subset consisting of sequences which cover $t_{*}$ also clearly has zero volume, so the constraint that the sequences not cover $t_{*}$ can also be removed.

In conclusion, we can represent the volume $Vol(M^{n})$ using the
the representation
\begin{equation}
\label{canonicalrep7}
 \sum_{j=1}^{N}{\lambda_{j}t^{i}_{j}} +\bigl(1-\sum_{j=1}^{N}{\lambda_{j}}\bigr)t^{i}_{*},
 \quad i=1,..,n,
\end{equation}
defined on two subdomains corresponding to the two ways that
 the sequence
\[0 \leq t_{1}  < t_{2} < \cdots < t_{N} \leq 1\]
 can have index $\frac{n}{2}$.
That is,
when $n$ is even,
 one subdomain corresponds to not allowing $0$ or $1$  and
the other to including both $0$ and $1$.  When $n$ is odd,
 one subdomain corresponds to  including $0$ and not $1$  and
the other to including $1$ and not $0$.

We now compute the Jacobian determinants.
First consider the odd case, $n=2m-1$.
Then,  sequences of index  $\frac{n}{2}=m-\frac{1}{2}$ split into the lower and upper sequences
\begin{equation*}
\begin{cases}
0=t_{1} < t_{2} < \cdots < t_{m}<1,\\
0 < t_{1} < t_{2} < \cdots < t_{m}= 1\, .
\end{cases}
\end{equation*}
Define
the lower representation
\[ \phi^{c}_{ol}:\Lambda^{m}\times T^{m-1} \rightarrow Int(M^{2m-1})\]   by
\begin{eqnarray}
\label{def_phi_col}
 \phi^{c}_{ol}(\lambda,t;t_{*})
&= &\Psi  \Bigl(\lambda_{0}\delta_{0}+  \sum_{j=1}^{m-1}{\lambda_{j}\delta_{t_{j}}}   +(1-\sum_{j=0}^{m-1}{\lambda_{j}})\delta_{t_{*}}\Bigr)\notag\\
&=& \Bigl(\sum_{j=1}^{m-1}{\lambda_{j}t_{j}^{i}}+ (1-\sum_{j=0}^{m-1}{\lambda_{j})t_{*}^{i}}\Bigr)_{i=1}^{2m-1}
\end{eqnarray}
and the upper
representation \[ \phi^{c}_{ou}:\Lambda^{m}\times T^{m-1} \rightarrow Int(M^{2m-1}) \]
  by
\begin{eqnarray}
\label{def_phi_cou}
 \phi^{c}_{ou}(\lambda,t;t_{*})
&= &\Psi  \Bigl(\lambda_{0}\delta_{1}+  \sum_{j=1}^{m-1}{\lambda_{j}\delta_{t_{j}}}   +(1-\sum_{j=0}^{m-1}{\lambda_{j}})\delta_{t_{*}}\Bigr)\notag\\
&=& \Bigl(\lambda_{0}+\sum_{j=1}^{m-1}{\lambda_{j}t_{j}^{i}}+ (1-\sum_{j=0}^{m-1}{\lambda_{j})t_{*}^{i}}\Bigr)_{i=1}^{2m-1}\, .
\end{eqnarray}
\begin{prop}
\label{prop_Jco}
When $n=2m-1$, for $t_{*} \in (0,1)$, the Jacobian determinants are
\[ |det(d\phi^{c}_{ol})|(\lambda,t;t_{*})= \mathcal{J}^{c}_{ol}(t_{*},t)
\prod_{j=1}^{m-1}{\lambda_{j}} \]
\[ |det(d\phi^{c}_{ou})|(\lambda,t;t_{*})= \mathcal{J}^{c}_{ou}(t_{*},t)
\prod_{j=1}^{m-1}{\lambda_{j}} \]
where
\[\mathcal{J}^{c}_{ol}(t_{*},t):=
 t_{*} \prod_{j=1}^{m-1}{(t_{j}-t_{*})^{2}}\prod_{j=1}^{m-1}{t_{j}^{2}}
\cdot \Delta_{m-1}^{4}(t)
\]
\[
\mathcal{J}^{c}_{ou}(t_{*},t):=
 \bigl(1-t_{*}\bigr) \prod_{j=1}^{m-1}{(t_{j}-t_{*})^{2}}
\prod_{j=1}^{m-1}{(1-t_{j})^{2}}
\cdot \Delta_{m-1}^{4}(t)
\]
Moreover,
\begin{equation}
\label{Jcsymm}
\mathcal{J}^{c}_{ou}(t_{*},t)=\mathcal{J}^{c}_{ol}(1-t_{*},1-t)\, .
\end{equation}

\end{prop}

Now consider the even case, $n=2m$.
Then  sequences of index  $\frac{n}{2}=m$ split into the lower and upper  sequences
\begin{equation*}
\begin{cases}
0 < t_{1} < t_{2} < \cdots < t_{m}< 1\\
0=t_{1} < t_{2} < \cdots < t_{m+1}=1,
\end{cases}
\end{equation*}
Therefore, we define the lower representation
\[ \phi^{c}_{el}:\Lambda^{m}\times T^{m} \rightarrow Int(M^{2m})\]  by
\begin{eqnarray}
\label{def_phi_cel}
 \phi^{c}_{el}(\lambda,t;t_{*})
&= &\Psi  \Bigl( \sum_{j=1}^{m}{\lambda_{j}\delta_{t_{j}}}+(1-\sum_{j=1}^{m}{\lambda_{j}})\delta_{t_{*}}\Bigr)\notag\\
&=& \Bigl(\sum_{j=1}^{m}{\lambda_{j}t_{j}^{i}}+ (1-\sum_{j=1}^{m}{\lambda_{j})t_{*}^{i}}\Bigr)_{i=1}^{2m}\, .
\end{eqnarray}
and the upper representation
\[ \phi^{c}_{eu}:\Lambda^{m+1}\times T^{m-1} \rightarrow Int(M^{2m})\]  by
\begin{eqnarray}
\label{def_phi_ceu}
 \phi^{c}_{eu}(\lambda,t;t_{*})
&= &\Psi  \Bigl(\lambda_{0}\d_{0}+ \sum_{j=1}^{m-1}{\lambda_{j}\delta_{t_{j}}}+\lambda_{m}\d_{1} +(1-\sum_{j=0}^{m}{\lambda_{j}})\delta_{t_{*}}\Bigr)\notag\\
&=& \Bigl(\sum_{j=1}^{m-1}{\lambda_{j}t_{j}^{i}}+\lambda_{m}+ (1-\sum_{j=0}^{m}{\lambda_{j})t_{*}^{i}}\Bigr)_{i=1}^{2m}\, .
\end{eqnarray}

\begin{prop}
\label{prop_Jce}
When $n=2m$, for $t_{*} \in (0,1)$, the Jacobian determinants are
 \[|det(d\phi^{c}_{el})|(\lambda,t;t_{*})= \mathcal{J}^{c}_{el}(t_{*},t)
\prod_{j=1}^{m}{\lambda_{j}}
\]
\[
 |det(d\phi^{c}_{eu})|(\lambda,t;t_{*})= \mathcal{J}^{c}_{eu}(t_{*},t)
\prod_{j=1}^{m-1}{\lambda_{j}}\, ,
\]
where
\[
\mathcal{J}^{c}_{el}(t_{*},t)
:= \prod_{j=1}^{m}{(t_{j}-t_{*})^{2}}\cdot \Delta_{m}^{4}(t)\, .
\]
\[
\mathcal{J}^{c}_{eu}(t_{*},t)
:=
 t_{*}\bigl(1-t_{*}\bigr) \prod_{j=1}^{m-1}{(t_{j}-t_{*})^{2}}\prod_{j=1}^{m-1}{t_{j}^{2}(1-t_{j})^{2}}
\cdot \Delta_{m-1}^{4}(t)\, ,
\]
\end{prop}

Finally, observe that if we extend the canonical representations to be defined for $t_{*}=0,1$
by continuity, we obtain the following relations between the canonical representations evaluated at the
 endpoints and the principal representations:
\begin{eqnarray}
\label{principal_canonical0}
\mathcal{J}^{c}_{ol}(0,t)&\equiv&0\notag\\
\mathcal{J}^{c}_{ou}(0,t)&=&\mathcal{J}^{p}_{ou}(t)\notag\\
\mathcal{J}^{c}_{el}(0,t)&=&\mathcal{J}^{p}_{el}(t)\notag\\
\mathcal{J}^{c}_{eu}(0,t)&\equiv&0\, ,
\end{eqnarray}
\begin{eqnarray}
\label{principal_canonical0_2}
|d\phi^{c}_{ol}(0,t)|&\equiv&0\notag\\
|d\phi^{c}_{ou}(0,t)|&=&|d\phi^{p}_{ou}(t)|\notag\\
|d\phi^{c}_{el}(0,t)|&=&|d\phi^{p}_{el}(t)|\notag\\
|d\phi^{c}_{eu}(0,t)|&\equiv&0\, ,
\end{eqnarray}
\begin{eqnarray}
\label{principal_canonical1}
\mathcal{J}^{c}_{ol}(1,t)&=&\mathcal{J}^{p}_{ou}(t)\notag\\
\mathcal{J}^{c}_{ou}(1,t)&\equiv&0\notag\\
\mathcal{J}^{c}_{el}(1,t)&=&\mathcal{J}^{p}_{eu}(t)\notag\\
\mathcal{J}^{c}_{eu}(1,t)&\equiv&0\, ,
\end{eqnarray}
\begin{eqnarray}
\label{principal_canonical1_2}
|d\phi^{c}_{ol}(1,t)|&=&|d\phi^{p}_{ou}(t)|\notag\\
|d\phi^{c}_{ou}(1,t)|&\equiv&0\notag\\
|d\phi^{c}_{el}(1,t)|&=&|d\phi^{p}_{eu}(t)|\notag\\
|d\phi^{c}_{eu}(1,t)|&\equiv&0 \, .
\end{eqnarray}

\subsection{Change of variables integral representations}
\label{sec_cov}
In Karlin and Shapley's
  \cite[Thm.~15.2]{KarlinShapley} proof
of the Hausdorff moment volume formula  \eqref{karlinshapley}, they used the lower principal representation $\phi^{p}_{ol}$ of \eqref{def_phi_pol}
 when $n$ is odd and
$\phi^{p}_{el}$ of \eqref{def_phi_pel} when $n$ is even combined with the change of variables formula.
To develop this method so that it can be used for the canonical
representations, which are not bijections,
it
 is convenient to proceed in  some generality. To begin, consider  a
 representation
\[ \phi:W \rightarrow Int(M^{n}),\]
where  $W\subset \R^{n}$ is open and  $\phi$ is a continuously differentiable bijection.
Then, since $\phi$ is injective, by the change of variables
formula for injective differentiable mappings whose Jacobian determinant may vanish
 (see e.g.~\cite[Thm.~3.13]{Spivak} combined
with
 Sard's Theorem \cite[Thm.~3.14]{Spivak}), we conclude that
\[Vol\bigl(\phi(W)\bigr)=\int_{W}{|d\phi|}\, .\]
Moreover, since $\phi$ is surjective we have
\[\phi(W)=Int(M^{n})\]
and from  \eqref{boundarynull} we have
\[Vol\bigl(M^{n}\bigr)=Vol\bigl(Int(M^{n})\bigr)\]
 so that we  conclude
\begin{equation}
\label{eq_cov}
Vol\bigl(M^{n}\bigr)=\int_{W}{|d\phi|}\, .
\end{equation}
To compute $Vol\bigl(M^{n}\bigr)$, Karlin and Shapley then evaluated the righthand side
by determining the Jacobian determinant and then evaluating the resulting integral using a Selberg integral formula.

However, more can be done along these lines.
Indeed, applying the full change of variables formula  we obtain
\[
\int_{\phi(W)}{f}=\int_{W}{\bigl(f\circ\phi\bigr) |d\phi|}
\]
for any function $f:\phi(W) \rightarrow \R$ that is integrable over $\phi(W)$. In particular, since $M^{n}$ is compact, it follows using the same reasoning
that was applied above to the case $f\equiv 1$, that for any
bounded measurable function $f:M^{n} \rightarrow \R$ we have
\begin{equation}
\label{id_cov}
\int_{M^{n}}{f}=\int_{W}{\bigl(f\circ\phi\bigr) |d\phi|}\, .
\end{equation}
We now apply this to the
component functions $ q\mapsto q_{i},
i=1,..,n$ on $M^{n}$ where we abuse notation and indicate them by the symbol $q_{i}$. It may be profitable to also consider nonlinear functions such as $q \mapsto q_{i}^{2}$ but we will not do that here.
Then, in this notation,   $q_{i}\circ\phi=\phi^{i}$  and \eqref{id_cov} becomes
\begin{equation}
\label{cov}
\int_{M^{n}}{q_{i}}=\int_{W}{\phi^{i} |d\phi|}
\end{equation}
That is, we have an integral representation of the mean Hausdorff moments.

However, to prove Lemma \ref{lem_mass_sup}, instead of a principal representation, we use a {\em family} of canonical
representations from Section \ref{sec_canonical}. In this case, utilizing the conclusion at
\eqref{canonicalrep7},
the major difference with the previous discussion is that, instead of a single bijection,
there are two continuously differentiable injections
\[\phi_{k}:W_{k}\rightarrow Int(M^{n}),\quad k=1,2\]
that are volume filling in the sense that
\[
Vol\bigl(Int(M^{n})\bigr)= Vol\bigl(\phi_{1}(W_{1})\cup \phi_{2}(W_{2})\bigr)
\]
and
\[
\phi_{1}(W_{1})\cap \phi_{2}(W_{2})=\emptyset  \]
and, instead of $W_{k}, k=1,2$ being  open, there exists open sets $V_{k}\subset W_{k}, k=1,2$
such that
\[   Vol(W_{k})=Vol(V_{k})\, .\]
Then
the analysis above can easily be repeated to conclude that
\begin{equation}
\label{id_cov3}
\int_{M^{n}}{f}=\int_{W_{1}}{f\circ{\phi_{1}} |d\phi_{1}|} +\int_{W_{2}}{f\circ{\phi_{2}} |d\phi_{2}|}
\end{equation}
for  any bounded measurable function $f:M^{n} \rightarrow \R$.
 In particular, we conclude
\begin{equation}
\label{id_cov2}
\int_{M^{n}}{q_{i}}=\int_{W_{1}}{\phi_{1}^{i} |d\phi_{1}|}+\int_{W_{2}}{\phi_{2}^{i} |d\phi_{2}|}\, ,
\end{equation}
our primary integration identity for the mean Hausdorff moments to be used in the next section.

\section{Mean Hausdorff Moments using the Markov-Kre\u{\i}n Representations}
\label{sec_mean}

We are now prepared to derive integral representations of the mean truncated Hausdorff moments
 with respect to the
uniform measure on $M^{n}$ and show that the canonical representations
generate  reproducing kernel identities corresponding to reproducing kernel Hilbert spaces
of $n$-th degree polynomials.
These identities
are used in
Section \ref{sec_newselberg} to derive
biorthogonal systems of Selberg integral formulas.
 The mean moments with respect to many other Selberg-type
densities can also be computed but to keep this presentation simple we will not do that here.

We will use
Selberg's result (see e.g.~\cite{ForresterWarnaar})
\begin{equation}
\label{selbergformula}
S_{n}(\alpha,\beta,\gamma)= \prod_{j=0}^{n-1}{\frac{\Gamma(\alpha+ j\gamma)\Gamma(\beta+j\gamma)
\Gamma(1+(j+1)\gamma)}{\Gamma(\alpha +\beta +(n+j-1)\gamma) \Gamma(1+\gamma)}}
\end{equation}
for the integrals
\begin{equation}
\label{selbergintegral}
S_{n}(\alpha,\beta,\gamma):=
\int_{I^{n}}{ \prod_{j=1}^{n}{t_{j}^{\alpha-1}(1-t_{j})^{\beta-1}}
|\Delta(t)|^{2\gamma}dt}\, ,
\end{equation}
where
$  Re(\alpha)>0, Re(\beta)>0,
Re(\gamma) > -\min{\bigl(\frac{1}{n}, Re(\alpha)/(n-1), Re(\beta)/(n-1) \bigr)} \, .$

We begin with the volume calculation and then proceed to higher moments using the result of the volume
calculation.
The main idea of our approach is the following. Recall from Section \ref{sec_principal} that the lower
and upper principal representations are each  bijections with $Int(M^{n})$ so that the volume $Vol(M^{n})$ can be computed
using  the change of variables result \eqref{eq_cov}. For example, when $n=2m-1$,
         the  lower principal representation
       $\phi^{p}_{ol}$ defined in \eqref{def_phi_pol} and the upper principal representation $\phi^{p}_{ol}$
       defined in \eqref{def_phi_pou}, along with the values of their Jacobian determinants from
       Proposition \ref{prop_Jpo}  produce two different integral representations for $Vol(M^{n})$. Specifically,
       in  the notation for Selberg's formulas \eqref{selbergformula} for
       the integrals \eqref{selbergintegral},  using the identity
$ \int_{\Lambda^{m-1}}{\bigl(1-\sum_{j=1}^{m-1}{\lambda_{j}}\bigr)\prod_{j=1}^{m-1}{\lambda_{j}}d\lambda}
                       =\frac{1}{(2m-1)!}$ ,
        the lower representation yields
       \begin{eqnarray}
\label{vol_Jpol}
       Vol(M^{2m-1})&=& \int_{\Lambda^{m-1} \times T^{m}}{|det(d\phi^{p}_{ol})|}\notag\\
 &=& \Bigl(\int_{\Lambda^{m-1}}{\bigl(1-\sum_{j=1}^{m-1}{\lambda_{j}}\bigr)\prod_{j=1}^{m-1}{\lambda_{j}}d\lambda}\Bigr)
               \int_ {T^{m}}{\mathcal{J}^{p}_{ol}}\notag\\
&=&\frac{1}{(2m-1)!}
               \int_ {T^{m}}{\mathcal{J}^{p}_{ol}}\notag\\
&=&\frac{1}{(2m-1)!m!}
               \int_ {I^{m}}{\mathcal{J}^{p}_{ol}}\\
	       &=& \frac{1}{(2m-1)!m!}
	       \int_ {I^{m}}{\Delta_{m}^{4}(t)dt}\notag\\
&=& \frac{1}{(2m-1)!m!} S_{m}(1,1,2)\, .\notag
		       \end{eqnarray}
	On the other hand,  using the identity
$ \int_{\Lambda^{m}}{\prod_{j=1}^{m-1}{\lambda_{j}}d\lambda}=\frac{1}{(2m-1)!}$, the upper representation $\phi^{p}_{ou}$ yields
	\begin{eqnarray}
\label{vol_Jpou}
	Vol(M^{2m-1})&=& \int_{\Lambda^{m} \times T^{m-1}}{|det(d\phi^{p}_{ou})|}\notag\\
 &=& \Bigl(\int_{\Lambda^{m}}{\prod_{j=1}^{m-1}{\lambda_{j}}d\lambda}\Bigr)\int_ {T^{m-1}}{\mathcal{J}^{p}_{ou}}
                        \notag \\
&=& \frac{1}{(2m-1)!}\int_ {T^{m-1}}{\mathcal{J}^{p}_{ou}}
                        \notag \\
&=& \frac{1}{(2m-1)!(m-1)!}\int_ {I^{m-1}}{\mathcal{J}^{p}_{ou}}\\
		&=& \frac{1}{(2m-1)!(m-1)!} \int_ {I^{m-1}}{\prod_{j=1}^{m-1}{t_{j}^{2}(1-t_{j})^{2}}\cdot \Delta_{m}^{4}(t)dt}\notag\\
&=&\frac{1}{(m-1)!(2m-1)!} S_{m-1}(3,3,2) \, .\notag
				\end{eqnarray}
	Combining the two results \eqref{vol_Jpol} and \eqref{vol_Jpou} we conclude
	the identity
	\[\frac{1}{(m-1)!} S_{m-1}(3,3,2)= \frac{1}{m!} S_{m}(1,1,2)\]
which is  confirmed  through direct calculation.

	In the even case, where $n=2m$,
	we use the representation $\phi^{p}_{el}$ defined in \eqref{def_phi_pel}
	and its Jacobian determinant from
	Proposition \ref{prop_Jpe}, along with the identity
$ \int_{\Lambda^{m}}{\prod_{j=1}^{m}{\lambda_{j}}d\lambda}=\frac{1}{(2m)!}$,
  to conclude that
\begin{equation}
\label{vol_Jpel}
 Vol(M^{2m})=\frac{1}{(2m)!m!}\int_{I^{m}}{\mathcal{J}^{p}_{el}}=\frac{1}{(2m)!m!}S_{m}(3,1,2)\, .
\end{equation}
Using same identity, the upper representation  $\phi^{p}_{eu}$ defined in \eqref{def_phi_peu} yields
\begin{equation}
\label{vol_Jpeu}
 Vol(M^{2m})=\frac{1}{(2m)!m!}\int_{I^{m}}{\mathcal{J}^{p}_{eu}}=\frac{1}{(2m)!m!}S_{m}(1,3,2)\, .
\end{equation}
Equating the two we conclude
that
\[S_{m}(1,3,2)=S_{m}(3,1,2)\]
which  is well known from the symmetry of the Selberg formula in its first two arguments, and corresponds to
 the change of variables
$t \mapsto 1-t$.
Consequently, we see how two different integral representations of the volume $Vol(M^{n})$ generate
identities.

However, the canonical representations form a {\em one parameter family} of  representations
 of $Int(M^{n})$ and the value $Vol(M^{n})$ expressed in terms
of the resulting one parameter family of integrals produces more interesting results.  To see this,
consider the odd case $n=2m-1$, and the volume filling pair of representations
$\phi^{c}_{ol}$ and $\phi^{c}_{ou}$ defined in \eqref{def_phi_col} and \eqref{def_phi_cou}  with
Jacobian determinants evaluated in Proposition \ref{prop_Jco}.  Apply the modified change of variables formula
\eqref{id_cov2} in terms of these two representations, and the
 identity
$ \int_{\Lambda^{m}}{\prod_{j=1}^{m-1}{\lambda_{j}}d\lambda}=\frac{1}{(2m-1)!}$,  to obtain
\begin{eqnarray}
\label{eq_odd_sum0}
Vol(M^{2m-1})&=&\int_{\Lambda^{m}\times T^{m-1}}{|d\phi^{c}_{ol}|}+
\int_{\Lambda^{m}\times T^{m-1}}{|d\phi^{c}_{ou}|}\notag\\
&=&\int_{\Lambda^{m}}{\prod_{j=1}^{m-1}{\lambda_{j}}d\lambda}\int_{T^{m-1}}{\mathcal{J}^{c}_{ol}}+\int_{\Lambda^{m}}{\prod_{j=1}^{m-1}{\lambda_{j}}d\lambda}\int_{T^{m-1}}{\mathcal{J}^{c}_{ou}}\notag\\
&=&\frac{1}{(2m-1)!}\int_{T^{m-1}}{\bigl(\mathcal{J}^{c}_{ol}+\mathcal{J}^{c}_{ou}\bigr)}\notag\\
&=&\frac{1}{(2m-1)!(m-1)!}\int_{I^{m-1}}{\bigl(\mathcal{J}^{c}_{ol}+\mathcal{J}^{c}_{ou}\bigr)}\, .\notag
\end{eqnarray}
Therefore, showing the parameters,  we conclude that for $t_{*} \in (0,1)$ we have
\begin{equation}
\label{vol_Jco}
Vol(M^{2m-1}) =\frac{1}{(2m-1)!(m-1)!} \int_{I^{m-1}}{\bigl(\mathcal{J}^{c}_{ol}(t_{*},t)+
\mathcal{J}^{c}_{ou}(t_{*},t)\bigr)dt}\, .
\end{equation}
Since the identity \eqref{vol_Jco} holds for all  $t_{*} \in (0,1)$ it generates integral identities.
For the first, since the integrand is continuous in $t_{*}$ we can set $t_{*}=0$ to obtain
\[Vol(M^{2m-1}) =\frac{1}{(2m-1)!(m-1)!} \int_{I^{m-1}}{\bigl(\mathcal{J}^{c}_{ol}(0,t)+\mathcal{J}^{c}_{ou}(0,t)\bigr)dt}\, ,\]
but from Propositions \ref{prop_Jco} we have
\[\mathcal{J}^{c}_{ol}(0,t)\equiv 0\]
and from \eqref{principal_canonical0}
\[
\mathcal{J}^{c}_{ou}(0,t)=
\mathcal{J}^{p}_{ou}(t)
\]
so that we obtain
\[Vol(M^{2m-1}) =\frac{1}{(2m-1)!(m-1)!} \int_{I^{m-1}}{\mathcal{J}^{p}_{ou}(t)dt}\, ,\]
which we already knew from  the volume calculation using the principal representation
 \eqref{vol_Jpou}.
However, if we compute the first order differential invariant by differentiating \eqref{vol_Jco}
with respect to $t_{*}$ at $t_{*}=0$ we obtain the first integral formula of Theorem \ref{thm_selberg}.

Now consider the even case $n=2m$, and the volume filling pair of representations
$\phi^{c}_{el}$ and $\phi^{c}_{eu}$ defined in \eqref{def_phi_cel} and \eqref{def_phi_ceu}  with
Jacobian determinants evaluated in Proposition \ref{prop_Jce}.  Apply the modified change of variables formula
\eqref{id_cov2} in terms of these two representations, and
 the identities $ \int_{\Lambda^{m}}{\prod_{j=1}^{m}{\lambda_{j}}d\lambda}=\frac{1}{(2m)!}$
 and
$\int_{\Lambda^{m+1}}
{\prod_{j=1}^{m-1}{\lambda_{i}}d\lambda} = \frac{1}{(2m)!}$
 to obtain
\begin{eqnarray}
\label{eq_even_sum0}
Vol(M^{2m})&=&\int_{\Lambda^{m}\times T^{m} }{|d\phi^{c}_{el}|}+
\int_{\Lambda^{m+1}\times T^{m-1}}{|d\phi^{c}_{eu}|}\notag\\
&=&\Bigl(\int_{\Lambda^{m}}{\prod_{j=1}^{m}{\lambda_{j}}d\lambda}\Bigr)\int_{T^{m}}{\mathcal{J}^{c}_{el}}+
\Bigl(\int_{\Lambda^{m+1}}{\prod_{j=1}^{m-1}{\lambda_{j}}d\lambda}\Bigr)\int_{T^{m-1}}{\mathcal{J}^{c}_{eu}}\notag\\
&=&\frac{1}{(2m)!}\int_{T^{m}}{\mathcal{J}^{c}_{el}}+
\frac{1}{(2m)!}\int_{T^{m-1}}{\mathcal{J}^{c}_{eu}}\notag\\
&=&\frac{1}{(2m)!m!}\int_{I^{m}}{\mathcal{J}^{c}_{el}}+
\frac{1}{(2m)!(m-1)!}\int_{I^{m-1}}{\mathcal{J}^{c}_{eu}}\, .\notag
\end{eqnarray}
Showing the parameters,  we conclude that for $t_{*} \in (0,1)$ we have
\begin{equation}
\label{vol_Jce}
Vol(M^{2m})=\frac{1}{(2m)!m!}\int_{I^{m}}{\mathcal{J}^{c}_{el}(t_{*},t)dt}+
\frac{1}{(2m)!(m-1)!}\int_{I^{m-1}}{\mathcal{J}^{c}_{eu}(t_{*},t)dt}\, .
\end{equation}
Setting $t_{*}=0$ and using
\[
\mathcal{J}^{c}_{eu}(0,t)
\equiv 0
\]
from Proposition \ref{prop_Jce}
and
\[
\mathcal{J}^{c}_{el}(0,t)
= \mathcal{J}^{p}_{el}(t)
\]
from \eqref{principal_canonical0},  we obtain
\begin{eqnarray}
Vol(M^{2m})&=&\frac{1}{(2m)!m!}\int_{I^{m}}{\mathcal{J}^{c}_{el}(0,t)dt}+
\frac{1}{(2m)!(m-1)!}\int_{I^{m-1}}{\mathcal{J}^{c}_{eu}(0,t)dt}\notag\\
&=&\frac{1}{(2m)!m!}\int_{I^{m}}{\mathcal{J}^{p}_{el}(t)dt}\notag\\
&=&\frac{1}{(2m)!m!}\int_{I^{m}}{\prod_{j=1}^{m}{t_{j}^{2}}
\cdot \Delta_{m}^{4}(t)dt}\notag\\
&=&\frac{1}{(2m)!m!}S_{m}(3,1,2)\notag
\end{eqnarray}
which we  alread knew from  the volume calculation using the principal representation
\eqref{vol_Jpel}.
However, if we compute the first order differential invariant by differentiating \eqref{vol_Jce}
with respect to $t_{*}$ at $t_{*}=0$ we obtain the second integral formula of Theorem \ref{thm_selberg}.

We can now proceed to compute the mean of the  moments
with respect to the uniform measure on $M^{n}$ using
 the volume identities \eqref{vol_Jpol}, \eqref{vol_Jpou}, \eqref{vol_Jpel}, \eqref{vol_Jpeu}
 from the principal representations and
  \eqref{vol_Jco} and \eqref{vol_Jce} from the canonical representations.
From the identities \eqref{principal_canonical0}, \eqref{principal_canonical0_2},\eqref{principal_canonical1}, \eqref{principal_canonical1_2} connecting the canonical representations at the endpoints and the principal representations, it is clear that we can generate the integral formula for the mean moments corresponding to all the principal representations except $\phi^{p}_{ol}$ by doing so using the canonical representations and then evaluating the result at the endpoints. Therefore, we move directly to the
canonical representations.
Let $\d_{0}$ denote the indicator function defined by $\d_{0}(i)=1, i=0$ and  $\d_{0}(i)=0$ otherwise.
Using the convention that $0^{0}:=1$, the following proposition utilizes the volume equalities
 \eqref{vol_Jco} and \eqref{vol_Jce} to simultaneously expresses themselves and the moment equalities
generated by the canonical representations.
For a function $\phi:I \rightarrow \R$ we
 define the diagonal extension
$\Sigma \phi:I^{N} \rightarrow \R$
by
\[\bigl(\Sigma\phi\bigr)(t):=\sum_{j=1}^{N}{\phi(t_{j})},\quad t\in I^{N}\, .\]
For simple powers, we introduce the notation
\[\Sigma t^{i}:=\sum_{j=1}^{N}{t_{j}^{i}}\] for the power sum and note the important example
\[\Sigma t^{-1}:=\sum_{j=1}^{N}{t_{j}^{-1}}\]
that will be used in the Selberg integral formulas of Theorem \ref{thm_selberg}.
\begin{prop}
\label{prop_mom}
Let $n=2m-1$. Then for all $t_{*}\in I$  and $i=0,1,.., 2m-1$ we have
\begin{eqnarray}
\label{eq_odd_cov5a}
&&\int_{M^{2m-1}}{q_{i}}-\frac{t_{*}^{i}}{2m} Vol(M^{2m-1})\notag\\
&=&\frac{\d_{0}(i)}{(2m)!(m-1)!}\int_{I^{m-1}}
{\mathcal{J}^{c}_{ol}(t_{*},t)dt}+\frac{1}{(2m)!(m-1)!}\int_{I^{m-1}}
{\mathcal{J}^{c}_{ou}(t_{*},t)dt}\notag\\
 &+&\frac{2}{(2m)!(m-1)!}\int_{I^{m-1}}{\Sigma t^{i}
\Bigl(\mathcal{J}^{c}_{ol}(t_{*},t)+\mathcal{J}^{c}_{ou}(t_{*},t)\Bigr)dt}\, .\notag
\end{eqnarray}
  Let $n=2m$. Then for all $t_{*}\in I$  and $i=0,1,.., 2m$ we have
\begin{eqnarray}
&&\int_{M^{2m}}{q_{i}}- \frac{t_{*}^{i}}{2m+1}Vol(M^{2m})\notag\\
&=&\frac{\d_{0}(i)+1}{(2m+1)!(m-1)!}\int_{I^{m-1}}{
\mathcal{J}^{c}_{eu}(t_{*},t)dr}\notag\\
&+&\frac{2}{(2m+1)!m!}\int_{I^{m}}{\Sigma t^{i}\mathcal{J}^{c}_{el}(t_{*},t)dt}
+
\frac{2}{(2m+1)!(m-1)!}\int_{I^{m-1}}{\Sigma t^{i}
\mathcal{J}^{c}_{eu}(t_{*},t)dt}\, .\notag
\end{eqnarray}
\end{prop}

The above technique of comparing two representations of the same volume to generate identities we now
apply to the higher order moments with respect to the uniform measure on the moment space (the moment moments) by simply  subtracting the integral representations  of
Proposition \ref{prop_mom} evaluated at $t_{*}=0$ from that with arbitrary $t_{*} \in I$. We now show
that this  procedure produces a reproducing kernel identity on the space of polynomials.

\section{The Canonical Representations and Reproducing Kernel Hilbert Spaces of Polynomials}
\label{sec_RKHS}
The integral representations  of
Proposition \ref{prop_mom} show clear signs of the existence of reproducing kernel identities  of the form
\[f(x)=\int{K(x,y)f(y)dy},\quad  f \in H, \, x \in X\] since, in the odd case, the integrand on the righthand side
$\Sigma t^{i}$
is integrated against a kernel $\mathcal{J}^{c}_{ol}(t_{*},t)+\mathcal{J}^{c}_{ou}(t_{*},t)$  and produces a multiple of $t_{*}^{i}$ plus some terms. Reproducing kernel Hilbert spaces are Hilbert spaces of functions such that
pointwise evaluation is continuous on the Hilbert space. They have remarkable properties, in particular,
the reproducing kernel identities which can be thought of like an abstract Cauchy integral formula from complex analysis.

Let us present Proposition \ref{prop_mom} in reproducing kernel form. To that end,
define
\begin{equation}
\label{def_H}
\mathcal{H}(t_{*},t):=\mathcal{J}^{c}_{ou}(0,t)-\mathcal{J}^{c}_{ol}(t_{*},t)-\mathcal{J}^{c}_{ou}(t_{*},t),
\end{equation}
and note that from \eqref{principal_canonical0} we have $\mathcal{J}^{c}_{ou}(0,t)=\mathcal{J}^{p}_{ou}(t)$, so that
\[\mathcal{H}(t_{*},t)=\mathcal{J}^{p}_{ou}(t)-\mathcal{J}^{c}_{ol}(t_{*},t)-\mathcal{J}^{c}_{ou}(t_{*},t)\, .\]
Moreover, observe that the symmetries
\begin{eqnarray*}
\mathcal{J}^{p}_{ol}(1-t)&=&\mathcal{J}^{p}_{ol}(t)\\
\mathcal{J}^{p}_{ou}(1-t)&=&\mathcal{J}^{p}_{ou}(t)\\
\mathcal{J}^{c}_{ou}(t_{*},t)&=&\mathcal{J}^{c}_{ol}(1-t_{*},1-t)
\end{eqnarray*}
of \eqref{Jposymm} and \eqref{Jcsymm} combined with $\mathcal{J}^{c}_{ol}(0,t)\equiv 0$ imply that
\begin{eqnarray}
\label{H_dirichlet}
\mathcal{H}(0,t)& \equiv& 0\notag\\
\mathcal{H}(1,t)& \equiv & 0\, .
\end{eqnarray}
and
\begin{equation}
\label{H_symm}
\mathcal{H}(1-t_{*},1-t) =\mathcal{H}(t_{*},t)\, .
\end{equation}

Let  $\Pi^{n}$ denote the space of $n$-th degree polynomials in one variable with real coefficients.
\begin{thm}
\label{thm_RK}
For all $\phi \in \Pi^{2m-1}$ we have
\begin{eqnarray}
\phi(t_{*}) Vol(M^{2m-1})
&=&\frac{2}{(2m-1)!(m-1)!}\int_{I^{m-1}}{(\Sigma \phi)(t)\mathcal{H}(t_{*},t)dt}\notag\\
&+&\frac{\phi(0)}{(2m-1)!(m-1)!}\int_{I^{m-1}}
{\mathcal{J}^{c}_{ou}(t_{*},t)dt}
\\&+&\frac{\phi(1)}{(2m-1)!(m-1)!}\int_{I^{m-1}}
{\mathcal{J}^{c}_{ol}(t_{*},t)dt}\notag
\end{eqnarray}
and for  $\phi \in \Pi^{2m}$ we have
\begin{eqnarray*}
 \phi(t_{*})Vol(M^{2m})
&=&-\frac{2}{(2m)!m!}\int_{I^{m}}{(\Sigma \phi)(t)\bigl(\mathcal{J}^{c}_{el}(t_{*},t)-
\mathcal{J}^{c}_{el}(0,t)\bigr)dt}
\\&-&
\frac{2}{(2m)!(m-1)!}\int_{I^{m-1}}{(\Sigma \phi)(t)
\mathcal{J}^{c}_{eu}(t_{*},t)dt}\notag \\
&+&
\frac{\phi(0)}{(2m)!m!}\int_{I^{m}}{\mathcal{J}^{c}_{el}(t_{*},t)
dt}
- \frac{\phi(1)}{(2m)!(m-1)!}\int_{I^{m-1}}{
\mathcal{J}^{c}_{eu}(t_{*},t)dt}\, .
\notag\\
\end{eqnarray*}

\end{thm}

To integrate out the diagonal extension $\Sigma$,
for any function $ (t_{*},t) \mapsto \mathcal{J}(t_{*},t)$ we let
\[\bar{\mathcal{J}}(t_{*},s):=\int{\mathcal{J}\bigl(t_{*},(s,t_{2},..,t_{N})\bigr)dt_{2}\cdots dt_{N}}\]
denote the marginalization to the first component of $t$.
Now for any such function $\mathcal{J}$, which is invariant under the symmetric group acting on its second variable, we have
\[\int_{I^{N}}{(\Sigma \phi)(t) \mathcal{J}(t_{*},t)dt}=  N \int_{I}{ \phi(s) \bar{\mathcal{J}}(t_{*},s)ds}\]
so that we obtain the following corollary to Theorem \ref{thm_RK}.
Let us define
\[\bar{\mathcal{G}}(t_{*},s):=\bar{\mathcal{J}}^{c}_{el}(0,s)-
\bar{\mathcal{J}}^{c}_{el}(t_{*},s)
-(m-1)
\bar{\mathcal{J}}^{c}_{eu}(t_{*},s) \]
and note that
\[\bar{\mathcal{G}}(0,s)\equiv 0
\]
but
\[ \bar{\mathcal{G}}(1,s)\not \equiv 0\, .\]
Let $\Pi^{n}_{0} \subset \Pi^{n}$ denote the $n$-th degree polynomials $\phi \in  \Pi^{n}$
which vanish on the boundary of $I$, that is, $\phi(0)=\phi(1)=0$.
\begin{cor}
\label{cor_RK}
For all $\phi \in \Pi^{2m-1}$ we have
\begin{eqnarray}
\phi(t_{*}) Vol(M^{2m-1})
&=&\frac{2}{(2m-1)!(m-2)!}\int_{I}{ \phi(s)\bar{\mathcal{H}}(t_{*},s)ds}\notag\\
&+&\frac{\phi(0)}{(2m-1)!(m-2)!}\int_{I}
{\bar{\mathcal{J}}^{c}_{ou}(t_{*},s)ds}
\\&+&\frac{\phi(1)}{(2m-1)!(m-2)!}\int_{I}
{\bar{\mathcal{J}}^{c}_{ol}(t_{*},s)ds}\notag
\end{eqnarray}
and for  $\phi \in \Pi^{2m}$ we have
\begin{eqnarray*}
 \phi(t_{*})Vol(M^{2m})
&=&\frac{2}{(2m)!(m-1)!}\int_{I}{\phi(s)\bar{\mathcal{G}}(t_{*},s)ds}\notag\\
&+&
\frac{\phi(0)}{(2m)!(m-1)!}\int_{I}{\bar{\mathcal{J}}^{c}_{el}(t_{*},s)
ds}
- \frac{\phi(1)}{(2m)!(m-2)!}\int_{I}{
\bar{\mathcal{J}}^{c}_{eu}(t_{*},s)ds}\, .
\notag\\
\end{eqnarray*}
In particular,
for the normalizations
\begin{eqnarray*}
\hat{\bar{\mathcal{H}}}&:=&\frac{1}{Vol(M^{2m-1})}\frac{2}{(2m-1)!(m-2)!}\bar{\mathcal{H}}\notag\\
\hat{\bar{\mathcal{G}}}&:=&\frac{1}{Vol(M^{2m})}\frac{2}{(2m)!(m-1)!}\bar{\mathcal{G}}\, \notag
\end{eqnarray*}
 we have
\begin{eqnarray*}
\label{id_RKHS}
\phi(t_{*})
&=&\int_{I}{ \phi(s)\hat{\bar{\mathcal{H}}}(t_{*},s)ds},\quad \phi \in \Pi^{2m-1}_{0}\notag\\
 \phi(t_{*})
&=&\int_{I}{\phi(s)\hat{\bar{\mathcal{G}}}(t_{*},s)ds}, \quad  \phi \in \Pi^{2m}_{0}\, .
\end{eqnarray*}
\end{cor}

Let us now restrict our attention to the odd case and
let $L^{2}(I)$ denote the usual Lebesgue space corresponding to the uniform Borel measure on $I$. Then,
it is well known, see e.g. Saitoh \cite[Thm.~1, Pg.~21]{Saitoh}, that the integral operator
\[ \phi \mapsto \int_{I}{\phi(s)\hat{\bar{\mathcal{H}}}(t_{*},s)ds},\quad \phi \in L^{2}(I)\]
determines a reproducing kernel Hilbert space structure on its range with reproducing kernel
\[\mathcal{K}(r_{1},r_{2}):=\int_{I}{\hat{\bar{\mathcal{H}}}(r_{1},s)\hat{\bar{\mathcal{H}}}(r_{2},s)ds}\, .\]
From the definition \eqref{def_H}
\[\mathcal{H}(t_{*},t):=\mathcal{J}^{c}_{ou}(0,t)-\mathcal{J}^{c}_{ol}(t_{*},t)-\mathcal{J}^{c}_{ou}(t_{*},t)\]
and the definitions of
$\mathcal{J}^{c}_{ol}$ and $\mathcal{J}^{c}_{ou}$
from Proposition \ref{prop_Jco}, it follows that
$\mathcal{H}(\cdot ,t) \in \Pi^{2m-1},\, t \in I^{m-1},$
and therefore  it follows from  \eqref{H_dirichlet} that
$\mathcal{H}(\cdot ,t) \in \Pi^{2m-1}_{0}, \, t \in I^{m-1}.$
 Consequently, by marginalization to $\bar{\mathcal{H}}$ and scalar normalization,
  we have
\begin{equation}
\label{hi_pi}
 \hat{\bar{\mathcal{H}}}(\cdot,s)\in \Pi^{2m-1}_{0}, \quad s \in I\, .
\end{equation}
Therefore the range of this integral operator is contained in $\Pi^{2m-1}_{0}$. However,
it follows from Corollary \ref{cor_RK} that the range is identically $\Pi^{2m-1}_{0}$. Therefore we
conclude that $\Pi^{2m-1}_{0}$ is a reproducing kernel Hilbert space with kernel
$\mathcal{K}$.

Because of Corollary \ref{cor_RK}, one might be tempted to think that this reproducing kernel Hilbert
space structure corresponds to that which $\Pi^{2m-1}_{0}$ inherits as the subspace
$\Pi^{2m-1}_{0} \subset L^{2}(I)$, but this is not the case.  Indeed, let $P_{2m-1}$ denote
the $L^{2}(I)$ orthogonal projection $P_{2m-1}:L^{2}(I) \rightarrow \Pi^{2m-1}_{0}$ and consider
the kernel
\begin{equation}
\label{def_K2m-1}
\mathcal{K}_{2m-1}(r_{1},r_{2}):=\int_{I}{P_{2m-1}\hat{\bar{\mathcal{H}}}(r_{1},s)\cdot P_{2m-1}
\hat{\bar{\mathcal{H}}}(r_{2},s)ds}\,
\end{equation}
where the projections are acting on the kernels in the second component. Then, since this projection makes no
difference in the reproducing identities in  Corollary \ref{cor_RK},  one can show
that $\mathcal{K}_{2m-1}$ is the reproducing kernel associated with $\Pi^{2m-1}_{0} \subset L^{2}(I)$
and since the latter  can  be computed
in terms of  the Legendre polynomials of order $2$  (see e.g.~\cite[Sec.~12.5]{Arfken}) using the
Christoffel-Darboux formula  \cite{Christoffel} (see e.g.~Simon \cite{Simon}
for a more current reference), we conclude an identification of $\mathcal{K}_{2m-1}$ with the
Christoffel-Darboux formula  for the kernel of the  Legendre polynomials of order $2$. That, is
\begin{equation}
\label{def_CD}
\mathcal{K}_{2m-1}(r_{1},r_{2})=\frac{(r_{1}-r^{2}_{1})(r_{2}-r^{2}_{2})}{2(2m-1)(2m)(2m+1)}
\frac{P''_{2m}(r_{1})P''_{2m-1}(r_{2})-P''_{2m-1}(r_{1})
P''_{2m}(r_{2})}{r_{1}-r_{2}}
\end{equation}
for $ (r_{1},r_{2}) \in I^{2}$,
where $P_{k}$ are the Legendre polynomials  shifted to the interval
\begin{equation}
\label{def_Legendre}
P_{k}(r)=\frac{1}{k!}\frac{d^{k}(r^{2}-r)^{k}}{dr^{k}}, \quad r \in I
\end{equation}
and
\begin{equation}
\label{def_Legendre2}
	Q_{k}(r):= (r-r^{2}) P''_{k}(r), \quad r \in I
\end{equation}
are the  associated  Legendre polynomials of order $2$ (see e.g.~\cite[Sec.~12.5]{Arfken}).

Moreover, since
\[\bar{\mathcal{H}}(t_{*},s):=\bar{\mathcal{J}}^{c}_{ou}(0,s)-\bar{\mathcal{J}}^{c}_{ol}(t_{*},s)-
\bar{\mathcal{J}}^{c}_{ou}(t_{*},s),\]
and from Proposition \ref{prop_Jco} we have
 \begin{eqnarray*}
\bar{\mathcal{J}}^{c}_{ou}(0,0)& \equiv &0\\
\bar{\mathcal{J}}^{c}_{ol}(t_{*},0)&\equiv &0\\
\bar{\mathcal{J}}^{c}_{ou}(t_{*},0)& >& 0,\quad t_{*}\in (0,1)\, ,
\end{eqnarray*}
we find that
\[\bar{\mathcal{H}}(t_{*},0)=-\bar{\mathcal{J}}^{c}_{ou}(t_{*},0) <  0, \quad  t_{*}\in (0,1)\, .\]
Consequently, for $t_{*} \in  (0,1)$, it follows  that $\hat{\bar{\mathcal{H}}}(t_{*},\cdot) \notin \Pi^{2m-1}_{0}$ and therefore
\[\mathcal{K} \neq \mathcal{K}_{2m-1}\, .\]
Moreover,
 from the orthogonal decomposition
\begin{eqnarray*}
\label{dK}
\mathcal{K}(r_{1},r_{2})&=&\int_{I}{\hat{\bar{\mathcal{H}}}(r_{1},s)\hat{\bar{\mathcal{H}}}(r_{2},s)ds}\notag\\
&=&\int_{I}{P_{2m-1}\hat{\bar{\mathcal{H}}}(r_{1},s)\cdot P_{2m-1}\hat{\bar{\mathcal{H}}}(r_{2},s)ds}
+\int_{I}{P^{\perp}_{2m-1}\hat{\bar{\mathcal{H}}}(r_{1},s)\cdot P^{\perp}_{2m-1}\hat{\bar{\mathcal{H}}}(r_{2},s)ds}\notag \\
&=&\mathcal{K}_{2m-1}(r_{1},r_{2})
+\int_{I}{P^{\perp}_{2m-1}\hat{\bar{\mathcal{H}}}\cdot P^{\perp}_{2m-1}\hat{\bar{\mathcal{H}}}(r_{2},s)ds}
\end{eqnarray*}
we conclude
\begin{thm}
\label{thm_RKHS}
Let $\mathcal{K}_{2m-1}$ denote the reproducing kernel for the polynomials
$\Pi^{2m-1}_{0}$ as a subset of $L^{2}(I)$. Then $\mathcal{K}_{2m-1}$ can be expressed by
both \eqref{def_K2m-1} and \eqref{def_CD}.
Moreover, $\Pi^{2m-1}_{0}$ is also a reproducing kernel Hilbert space with kernel
$\mathcal{K}$, and
\[\mathcal{K} -\mathcal{K}_{2m-1}\]
is a reproducing kernel.
\end{thm}
\begin{rmk}
Besides the fact that the kernel $\mathcal{K}$ defining the Hilbert space structure for the polynomials
$\Pi^{2m-1}_{0}$ is not that of the Legendre polynomials, we do not know if this kernel is known, nor do we have
an explicit formula for it. However,
what this section shows is that this kernel and its associated Hilbert space $\Pi^{2m-1}_{0}$ are intimately
connected with the canonical representations of truncated Hausdorff moments, and therefore might be called the
Markov-Kre\u{\i}n kernel. Moreover, if instead of the uniform
measure on the moments, a Selberg type density is used, more such reproducing kernels may be revealed.
\end{rmk}
\section{New Selberg Integral Formulas}
\label{sec_newselberg}
The integral representations of the mean Hausdorff moments of Proposition \ref{prop_mom} provide
  new integral identities of Selberg type.
In the following theorem, we provide the first in a sequence
 corresponding to when $n$ is odd and even. We then show how to use the reproducing kernel identities
of Theorem \ref{thm_RK} to generate biorthogonal systems of Selberg integral formulas.
\begin{thm}
It holds true that
\label{thm_selberg}
\begin{equation}
\int_{I^{m}}{\Sigma t^{-1}\cdot \prod_{j=1}^{m}{t_{j}^{2}(1-t_{j})^{2}}
\Delta_{m}^{4}(t)
dt}= \frac{S_{m}(5,1,2)-S_{m}(3,3,2)}{2}\, .\end{equation}
and
\begin{equation}
 \int_{I^{m}}{\Sigma t^{-1} \cdot \prod_{j=1}^{m}{t_{j}^{2}}\cdot \Delta_{m}^{4}(t)dt}=
\frac{m}{2}S_{m-1}(5,3,2)\, .\end{equation}
\end{thm}

The identities of Theorem \ref{thm_selberg} follow only from the volume equalities, that is,  the $i=0$ case of Theorem \ref{thm_RK}.
The following theorem demonstrates how to use all the moment equalities of  Theorem \ref{thm_RK}  to
 generate biorthogonal systems of
  Selberg integral formulas. Let us recall definition \eqref{def_H}
\[\mathcal{H}(t_{*},t)=\mathcal{J}^{p}_{ou}(t)-\mathcal{J}^{c}_{ol}(t_{*},t)-\mathcal{J}^{c}_{ou}(t_{*},t)\, .\]
\begin{thm}
\label{thm_selberg2}
Let $n=2m-1$ and
consider the scaled kernel
 \[\hat{\mathcal{H}}:=\frac{1}{Vol(M^{2m-1})}\frac{2}{(2m-1)!(m-1)!}\mathcal{H}\, .\]
Then,
\[\hat{\mathcal{H}}(\cdot ,t) \in \Pi^{2m-1}_{0}, \quad t \in I^{m-1}\, .\]
Moreover, consider  a basis $\{p_{j},j=1,..,2m-2\}$ for $\Pi^{2m-1}_{0}$  and the resulting  expansion
of $\hat{\mathcal{H}}(\cdot ,t)$ in this basis for each $t \in I^{m-1}$;
\[\hat{\mathcal{H}}(t_{*} ,t)=\sum_{j=1}^{2m-2}{h_{j}(t)p_{j}(t_{*})},\quad  (t_{*},t) \in I \times I^{m-1}\, .\]
Then,  $\{\Sigma p_{j},j=1,..,2m-2\},$ $\{h_{j},j=1,..,2m-2\}$ form an $L^{2}(I^{m-1})$ biorthogonal system.
That is,
\[
\int_{I^{m-1}}{ h_{j}\Sigma p_{k}} =\delta_{jk}, \quad  j,k=1,..,2m-2\, \,
\]
\end{thm}
As an immediate corollary, we have
\begin{cor}
\label{cor_selberg2}
Let $n=2m-1$ and
consider the scaled marginal kernel
 \[\hat{\bar{\mathcal{H}}}:=\frac{1}{Vol(M^{2m-1})}\frac{2}{(2m-1)!(m-2)!}\bar{\mathcal{H}}\, \]
(note  the different scaling than Theorem
\ref{thm_selberg2}).
Then,
\[\hat{\bar{\mathcal{H}}}(\cdot ,s) \in \Pi^{2m-1}_{0}, \quad s \in I\, .\]
Moreover, consider  a basis $\{p_{j},j=1,..,2m-2\}$ for $\Pi^{2m-1}_{0}$  and the resulting  expansion
of $\hat{\bar{\mathcal{H}}}(\cdot ,s)$ in this basis for each $s \in I$;
\[\hat{\bar{\mathcal{H}}}(t_{*} ,s)=\sum_{j=1}^{2m-2}{\bar{h}_{j}(s)p_{j}(t_{*})},\quad  (t_{*},s) \in I \times I\, .\]
Then,  $\{p_{j},j=1,..,2m-2\},$ $\{\bar{h}_{j},j=1,..,2m-2\}$ form an $L^{2}(I)$ biorthogonal system.
That is,
\[
\int_{I}{ \bar{h}_{j}p_{k}} =\delta_{jk}, \quad  j,k=1,..,2m-2\,  .
\]
\end{cor}

The choice of basis for $\Pi^{2m-1}_{0}$ determines
the corresponding component functions $h_{j},j=1,..,2m-2$  and the integrands
$ h_j\Sigma p_{k}$
in Theorem \ref{thm_selberg2}.
 Therefore, the task remaining is to select
a basis for which the component functions $h_{j}$ can be determined and such that the
resulting integrals are of interest. When the chosen basis is orthonormal with respect to some
inner product $\langle \cdot, \cdot\rangle$, then the coefficients $h_{j}$ in the representation
\[\hat{\mathcal{H}}(t_{*} ,t)=\sum_{j=1}^{2m-2}{h_{j}(t)p_{j}(t_{*})},\quad  (t_{*},t) \in I \times I^{m-1}\, .\]
of Theorem  \ref{thm_selberg} are
\[h_{j}(t)=\langle\hat{\mathcal{H}}(\cdot,t),  p_{j}\rangle \, .\]

As an example, we now compute these component functions, and therefore determine
explicit forms for these Selberg integrals, when the basis consists of the associated Legendre polynomials of order $2$.
To that end, recall the definitions \eqref{def_Legendre} and  \eqref{def_Legendre2} of the Legendre polynomials
and the associated Legendre polynomials of order $2$
translated to the unit interval $I$.
 In addition, recall
 the $j$-th symmetric function $e_{j}$  defined as
\[ e_{j}(t):=\sum_{i_{1}< \cdots < i_{j}}{t_{i_{1}}\cdots t_{i_{j}}}\]
with $e_{0}:=1$ and the symmetric functions $e_{j}(t,z)$ restricted to the diagonal $t=z$
\begin{equation}
\label{def_e}
e_{j}(t,t):=\sum_{j_{1}+j_{2}=j}{e_{j_{1}}(t)e_{j_{2}}(t)},\quad  j=0,..,2m-2\, .
\end{equation}
\begin{thm}
\label{thm_selberg_explicit}
Consider the basis of $\Pi^{2m-1}_{0}$ consisting of the   associated Legendre polynomials $Q_{j},j=2,..,2m-1$ of order $2$
translated to the unit interval $I$.  For $k=2,..,2m-1$ define
\[a_{jk}:=
 \frac{(j+k+k^{2})\Gamma(j+2)\Gamma(j)}{\Gamma(j+k+2)\Gamma(j-k+1)},\quad k \leq j \leq 2m-1\]
\[\tilde{h}_{k}(t):=
\sum_{j=k}^{2m-1}{(-1)^{j+1}a_{jk}e_{2m-1-j}(t,t)}\, .
        \]
Then for $j=k\bmod 2$, $j,k=2,..,2m-1$,  we have
\[\int_{I^{m-1}}{\tilde{h}_{k}(t)\Sigma Q_{j}(t)\prod_{j'=1}^{m-1}{t_{j'}^{2}}\cdot
        \Delta_{m-1}^{4}(t)dt}=Vol(M^{2m-1})(2m-1)!(m-1)!\frac{(k+2)!}{(8k+4)(k-2)!} \delta_{jk}\, .\]
\end{thm}

\section{Proofs}

\subsection{Proof of Theorem \ref{thm_shiva} }
We seek to apply the nested reduction theorem \cite[Thm.~4.11]{BayesOUQ}.
The assertion is trivially true when  $\mathcal{L}(\mathcal{A})=\mathcal{U}(\mathcal{A})$ so we can assume
$\mathcal{L}(\mathcal{A})<\mathcal{U}(\mathcal{A})$.
Let us first establish that the assumptions of the theorem are well defined. To that end, note that
\cite[Lem.~3.10]{BayesOUQ} (which follows from Castaing and Valadier
 \cite[Lemma III.39\, p.~86]{CastaingValadier:1977}, which in turn follows
from  Saint-Beuve's
 \cite{SainteBeuve1974} extension of  Aumann's Selection Theorem to Suslin spaces) implies that $q\rightarrow \inf_{(f,\mu)\in \Psi^{-1}(q)} \Dmap(f,\mu)[B]$ is
 universally measurable and hence the conditions of the theorem are well defined if we extend
 the definitions in the usual way when operating on universally measurable sets and functions.
Similarly, since for any $\lambda$ the function $(f,\mu)\mapsto (\Phi(f,\mu)-\lambda)\Dmap(f,\mu)[B]$ is measurable, the function
$\theta:\mathcal{Q}\rightarrow
\R$ defined by
\[\theta(q):=
 \sup_{(f,\mu)\in \Psi^{-1}(q)} (\Phi(f,\mu)-\lambda)\Dmap(f,\mu)[B_{\d}]
\] is universally
measurable.

For the proof of the theorem,  let $\mathbb{Q}\in \mathfrak{Q},\,\Dmap\in \Dmaps$ satisfy the assumptions, and    define $\lambda:=\mathcal{U}(\mathcal{A})-2\delta '$.
Consider the events
\[\mathcal{Q}_{\d}:=\bigl\{q: \inf_{(f,\mu)\in  \Psi^{-1}(q)} \Dmap(f,\mu)[B_{\d}]\leq \tau\big\}\]
\[U_{\e,\d}= \Bigl\{ q: \sup_{(f,\mu)\in \Psi^{-1}(q),\, \Dmap(f,\mu)[B_{\d}]> \e}\Phi(f,\mu) >\sup_{(f,\mu)\in \mathcal{A}}\Phi(f,\mu)    -\delta' \Bigr\}\]
where the assumptions \eqref{eq:dto0app} and \eqref{eq:djkdjehjehj33app} become
\[ \mathbb{Q}\bigl(\mathcal{Q}_{\d}\bigr) \geq 1-h(\d),\quad  \d >0\]
\[ \mathbb{Q}\bigl(U_{\e,\d}\bigr) \geq \e',  \quad \d >0\]
Let us denote $\tau':=\tau\bigl(\mathcal{U}(\mathcal{A})-\mathcal{L}(\mathcal{A})\bigr) $.
It is easy to see that
\[ \bigl\{\theta \geq -\tau' \bigr\}\supset \mathcal{Q}_{\d} \]
and
\[ \{\theta > \e\d'\} \supset  U_{\e,\d} \]
and therefore
\begin{eqnarray*}
 \mathbb{Q}\bigl(\{\theta \geq -\tau' \}\bigr)& \geq &\mathbb{Q}\bigl(\mathcal{Q}_{\d}\bigr)\\
  &\geq& 1-h(\d)
\end{eqnarray*}
and
\begin{eqnarray*}
\mathbb{Q}\bigl(\{\theta > \e\d'\} \bigr)& \geq &\mathbb{Q}\bigl(U_{\e,\d}\bigr)\\
  &\geq& \e'\, .
\end{eqnarray*}
Since $\Phi(f,\mu) \geq \mathcal{L}(\mathcal{A}), (f,\mu) \in \mathcal{A}$, it follows
that $|\theta|\leq \mathcal{U}(\mathcal{A})-\mathcal{L}(\mathcal{A})$, and so we obtain
\begin{eqnarray*}
\int{\theta d\mathbb{Q}} &=&
\int_{\{\theta > \e\d' \}}{\theta d\mathbb{Q}}+\int_{\{\theta \leq \e\d' \}}
{\theta d\mathbb{Q}}\\
&>&\e\d' \mathbb{Q}\bigl(\{\theta > \e\d' \}\bigr)+\int_{\{\theta \leq e\d'
\}}{\theta d\mathbb{Q}}\\
&\geq &\e\d' \mathbb{Q}\bigl(\{\theta > \e\d' \}\bigr)+\int_{\{\theta \leq 0
\}}{\theta d\mathbb{Q}}\\
&\geq &\e\d' \mathbb{Q}\bigl(\{\theta > \e\d' \}\bigr)+\int_{\{\theta < -\tau'
\}}{\theta d\mathbb{Q}}+\int_{\{-\tau' \leq \theta \leq 0
\}}{\theta d\mathbb{Q}}\\
&\geq &\e\d' \mathbb{Q}\bigl(\{\theta > \e\d' \}\bigr)
-\bigl(\mathcal{U}(\mathcal{A})-\mathcal{L}(\mathcal{A})\bigr)
\mathbb{Q}\bigl(\{\theta  < -\tau'
\}\bigr)- \tau'\mathbb{Q}\bigl(\{-\tau' \leq \theta  \leq 0
\}\bigr)
\\
&\geq &\e\d'\e'-\bigl(\mathcal{U}(\mathcal{A})-\mathcal{L}(\mathcal{A})\bigr)h(\d) -\tau'\,.
\end{eqnarray*}
Therefore, for  any strictly positive solution  $\d >0$ to
\[h(\d)+\tau \leq  \frac{\e\d'\e'}{\mathcal{U}(\mathcal{A})-\mathcal{L}(\mathcal{A})} \]
we have
\[ \E_{\mathbb{Q}}[\theta] =\int{\theta d\mathbb{Q}} > 0,\]
where we recall that the function  $\theta$ depends on $\d$,
and therefore trivially
\[ \sup_{\mathbb{Q}\in \mathfrak{Q},\, \Dmap\in \Dmaps} \mathbb{E}_{q \sim \mathbb{Q}} \left[ \sup_{(f,\mu)\in  \Psi^{-1}(q)} \bigl(\Phi(f,\mu)-\lambda\bigr)\Dmap(f,\mu)[B_{\d}] \right]
 > 0.\]
The assertion  then follows from  \cite[Thm.~4.11]{BayesOUQ}.

\subsection{Proof of Theorem \ref{thm_shiva_singlesample}}
We will apply
the Mass Supremum Equality \ref{lem_mass_sup}, the Mass Infimum Inequality \ref{lem_mass_inf},
and the Mass of First Moment Inequality \ref{massfirstmoment}. To that end,
 define the  events
\begin{equation*}
\label{mass_sup_event}
   S_{\e,\d}=\Bigl\{q\in M^{n}: \exists \mu \in \Psi^{-1}q: \mu(B_{\d})> \e\Bigr\}
\end{equation*}
\begin{equation*}
\label{mass_inf_event}
   I_{\d}=\Bigl\{q\in M^{n}: \exists \mu \in \Psi^{-1}q: \mu(B_{\d})=0\Bigr\}
\end{equation*}
\begin{equation*}
\label{massfirstmoment_event}
   FM_{\d'}=\Bigl\{q\in M^{n}: q_{1} \in (1-\d',1] \Bigr\}
\end{equation*}
First observe that some endpoint conditions have zero mass. For example,
\[\mathbb{Q}(S_{\e,\d}):=  \mathbb{Q}\bigl(\{q: \exists \mu \in \Psi^{-1}q: \mu(B_{\d})> \e\}\bigr)
= \mathbb{Q}\bigl(\{q: \exists \mu \in \Psi^{-1}q: \mu(B_{\d})\geq \e\}\bigr)\]
and\[\mathbb{Q}(FM_{\d'}):=\mathbb{Q}\bigl(\{q\in M^{n}: q_{1} \in (1-\d',1] \}\bigr)=\mathbb{Q}\bigl(\{q\in M^{n}: q_{1} \in [1-\d',1] \}\bigr) \, .\]
Consequently,
the Mass Supremum  Equality \ref{lem_mass_sup} asserts that
\begin{equation*}
\label{mass_sup1}
\mathbb{Q}(S_{\e,\d}) = \mathbb{Q}\bigl(\{q: \exists \mu \in \Psi^{-1}q: \mu(B_{\d})\geq \e\}\bigr)\geq
(1- \e)^{n}
\end{equation*}
where the right-hand side is independent of $\d$,
the Mass Infimum Inequality \ref{lem_mass_inf} asserts that
\begin{equation*}
\label{mass_inf1}
\mathbb{Q}(I_{\d})\geq 1- \d (2e)^{2n} \, ,
\end{equation*}
and  the Mass of First Moment Inequality \ref{massfirstmoment} asserts that
\begin{equation*}
\label{massfirstmoment1}
\mathbb{Q}(FM_{\d'})=\mathbb{Q}\bigl(\{q\in M^{n}: q_{1} \in [1-\d',1] \}\bigr) \geq (\d')^{n}\, .
\end{equation*}

Define the events
\[\mathcal{Q}_{\d}:=\bigl\{q: \inf_{\mu\in  \Psi^{-1}(q)} \mu[B_{\d}]=0\big\}\]
\begin{equation*}
U_{\e,\d}:= \Bigl\{ q: \sup_{\mu\in \Psi^{-1}(q),\, \mu[B_{\d}]> \e}\E_{\mu}[X] >1  -\delta' \Bigr\}
\end{equation*}
Then since
\[ \mathcal{Q}_{\d} \supset I_{\d}\]
we have
\[\mathbb{Q}\bigl( \mathcal{Q}_{\d}\bigr) \geq \mathbb{Q}\bigl( I_{\d}\bigr) \geq  1- \d (2e)^{2n}\]
and since
\begin{eqnarray*}
U_{\e,\d}&=& \Bigl\{ q: \sup_{\mu\in \Psi^{-1}(q),\, \mu[B_{\d}]> \e}\E_{\mu}[X] >1  -\delta' \Bigr\}\\
&=&  \Bigl\{ q: q_{1} \in (1-\d',1], \exists \mu\in \Psi^{-1}q: \mu[B_{\d}]>\e \Bigr\}\\
&=& S_{\e,\d} \cap FM_{\d'}
\end{eqnarray*}
we have
\begin{eqnarray*}
\mathbb{Q}\bigl(U_{\e,\d}\bigr)& = & \mathbb{Q}\bigl( S_{\e,\d} \cap FM_{\d'}\bigr)\\
&=& 1-\mathbb{Q}\bigl((S_{\e,\d} \cap FM_{\d'})^{c}\bigr)\\
&=& 1-\mathbb{Q}\bigl(S^{c}_{\e,\d} \cup FM_{\d'}^{c}\bigr)\\
&=& 1-\mathbb{Q}\bigl(S^{c}_{\e,\d}\bigr) - \mathbb{Q}\bigl(FM_{\d'}^{c}\bigr)\\
&=& \mathbb{Q}\bigl(S_{\e,\d}\bigr) + \mathbb{Q}\bigl(FM_{\d'}\bigr)-1\\
&=& \bigl(1-\e\bigr)^{n}-1 +(\d')^{n}\\
&=&  (\d')^{n}-n\e
\end{eqnarray*}
Consequently, if we choose $\e:=\frac{(\d')^{n}}{2n}$, then
\[\mathbb{Q}\bigl(U_{\e,\d}\bigr) \geq \frac{(\d')^{n}}{2},\]
so that
the assumptions \eqref{eq:dto0app} and \eqref{eq:djkdjehjehj33app}, expressed as
\[ \mathbb{Q}\bigl(\mathcal{Q}_{\d}\bigr) \geq 1-h(\d),\quad  \d >0\]
\[ \mathbb{Q}\bigl(U_{\e,\d}\bigr) \geq \e',  \quad \d >0\, ,\]
are satisfied with
$\e':= \frac{(\d')^{n}}{2}$, $\e:=\frac{(\d')^{n}}{2n}$, and
$h(\d):= \d (2e)^{2n}\, .$
We can solve
\begin{eqnarray*}
h(\d)& \leq&  \frac{\e\d'\e'}{\mathcal{U}(\mathcal{A})-\mathcal{L}(\mathcal{A})}\\
&=& \e\d'\e'\\
&=&\frac{(\d')^{n}}{2n} \d' \frac{(\d')^{n}}{2}\\
&=& \frac{(\d')^{2n+1}}{4n}
\end{eqnarray*}
by choosing $\d \leq \frac{1}{4n} \bigl(\d'\bigr)^{2n+1}\bigl(2e\bigr)^{-2n}$.

\subsection{Proof of Lemma \ref{lem_mass_sup} }
\label{sec_mass_sup_proof}
For
$ M^{n}_{\e}:=\bigl\{q \in  M^{n}: \exists \mu \in\Psi^{-1}q:  \mu(\{t_{*}\})\geq \e\bigr\}$, it follows that
$\Psi \mu \in  M^{n}_{\e}$ if and only if $\mu =\epsilon \delta_{t_{*}}+(1-\epsilon)\mu_{*}$ with $\mu_{*} \in \mathcal{M}(I)$.
For such a $\mu$ it follows that $\Psi \mu = \epsilon \Psi \delta_{t_{*}}+(1-\epsilon)\Psi\mu_{*}$ and therefore
\[ M^{n}_{\e}=\epsilon \Psi \delta_{t_{*}}+(1-\epsilon)M^{n}\] from which we conclude that
\[Vol( M^{n}_{\e})=(1-\epsilon)^{n}Vol(M^{n}),\]
establishing the assertion.

\subsection{Proof of Lemma \ref{lem_mass_inf}}
First consider the odd case, $n=2m-1$.
We utilize the bijective principal representation $\phi^{p}_{ol}: \Lambda^{m-1} \times T^{m} \rightarrow Int(M^{2m-1})$ defined in \eqref{def_phi_pol} and
\[ |det(d\phi^{p}_{ol})|(\lambda,t)= \mathcal{J}^{p}_{ol}(t)
\bigl(1-\sum_{j=1}^{m-1}{\lambda_{j}}\bigr)\prod_{j=1}^{m-1}{\lambda_{j}} \]
where
\[\mathcal{J}^{p}_{ol}(t):=
 \Delta^{4}_{m}(t)
        \]
from Proposition \ref{prop_Jpo} along with the change of variables formula
\eqref{id_cov}.

Fix $t_{*}\in (0,1)$ and let
\[ T_{\d}^{m}:= \bigl\{ (t_{1},..,t_{m})\in   T^{m}:
t_{j} \notin B_{\d}(t_{*}) , j=1,..,m\, \bigr\}  \]
denote those sequences which have no point a distance less than $\d$ from $t_{*}$.
It follows that
\[M^{2m-1}_{\d} \supset \phi\bigl(\Lambda^{m-1}\times T_{\d}^{m}\bigr)\]
and therefore
\begin{equation}
\label{e11111}
Vol\Bigl(M^{2m-1}_{\d}\Bigr)  \geq  Vol\Bigl(\phi^{p}_{ol}\bigl(\Lambda^{m-1}\times T_{\d}^{m} \bigr)\Bigr)\, .
\end{equation}
We bound the righthand side from below using the change of variables formula
\eqref{id_cov} as
\begin{eqnarray*}
 Vol\Bigl(\phi^{p}_{ol}\bigl(\Lambda^{m-1}\times T_{\d}^{m} \bigr)\Bigr)
&=&
\int_{\Lambda^{m-1} \times  T_{\d}^{m}}{|det(d\phi^{p}_{ol})|
 }\\
&=&\int_{\Lambda^{m-1}}{\bigl(1-\sum_{j=1}^{m-1}{\lambda_{j}}\bigr)\bigl(\prod_{j=1}^{m-1}{\lambda_{j}}\bigr)
d\lambda} \int_{T_{\d}^{m}}{\mathcal{J}^{p}_{ol}}\, ,
\end{eqnarray*}
and then bounding
\begin{eqnarray*}
\int_{T_{\d}^{m}}{\mathcal{J}^{p}_{ol}dt}
&=& \int_{T_{\d}^{m}}{\Delta_{m}^{4}(t)dt}\\
&=& \frac{1}{m!} \int_{I_{\d}^{m}}{\Delta_{m}^{4}(t)dt}
\end{eqnarray*}
where
\[ I_{\d}^{m}:= \bigl\{ (t_{1},..,t_{m})\in   I^{m}:
t_{j} \notin B_{\d}(t_{*}) , j=1,..,m\, \bigr\} . \]
To bound this from below we bound the integral over
$ \bigl(I_{\d}^{m}\bigr)^{c}$ from above.
To that end, let
\[ I_{\d,j}^{m}:= \bigl\{ (t_{1},..,t_{m})\in   I^{m}:
t_{j} \in B_{\d}(t_{*})\bigr \},\quad  j=1,..,m\,, \]
so that
\[ \bigl(I_{\d}^{m}\bigr)^{c}=  \cup_{j} I_{\d,j}^{m}\, .\]
Therefore, using a union bound and the symmetry of $\Delta$ we have
\begin{eqnarray*}
 \int_{\bigl(I_{\d}^{m}\bigr)^{c}}{\Delta_{m}^{4}(t)dt}&=&
 \int_{\cup_{j'=1}^{m}{I^{m}_{\d,j'}}}{\Delta_{m}^{4}(t)dt}\\
&\leq &\sum_{j'=1}^{m}{\int_{I^{m}_{\d,j'}}{\Delta_{m}^{4}(t)dt}}\\
& =&m\int_{I^{m}_{\d,1}}{\Delta_{m}^{4}(t)dt}\\
& = & m\int_{I^{m}_{\d,1}}{\prod_{1 \leq j < k \leq m}{(t_{k}-t_{j})^{4}}dt_{1}\cdots dt_{m}}\\
& \leq & m\int_{I^{m}_{\d,1}}{\prod_{2 \leq j < k \leq m}{(t_{k}-t_{j})^{4}}dt_{1}\cdots dt_{m}}\\
& = & mVol(B_{\d})\int_{I^{m-1}}{\prod_{2 \leq j < k \leq m}{(t_{k}-t_{j})^{4}}dt_{2}\cdots dt_{m}}\\
& = &mVol(B_{\d})\int_{I^{m-1}}{\Delta_{m-1}^{4}(t)dt}\\
&\leq & 2m\d S_{m-1}(1,1,2)\\
\end{eqnarray*}
and so obtain
\begin{eqnarray*}
\frac{Vol\Bigl(\phi(\Lambda^{m-1}\times T_{\d}^{m})\Bigr)}{Vol\Bigl(\phi(\Lambda^{m-1}\times T^{m})\Bigr)}
&=& \frac{\int_{I^{\d}_{m}}{\Delta_{m}^{4}(t)dt}}{\int_{I^{m}}{\Delta_{m}^{4}(t)dt}}\\
&\geq& 1- 2m\d \frac{S_{m-1}(1,1,2)}{S_{m}(1,1,2)}\, .
\end{eqnarray*}
Using Selberg's formulas \eqref{selbergformula} we compute
\begin{eqnarray*}
\frac{S_{m-1}(1,1,2)}{S_{m}(1,1,2)}&=&
\frac{ \prod_{j=0}^{m-2}{\frac{\Gamma(1+ 2j)^{2}
\Gamma(3+2j)}{2\Gamma(2(m+j)-2) }}}{ \prod_{j=0}^{m-1}{\frac{\Gamma(1+ 2j)^{2}
\Gamma(3+2j)}{2\Gamma(2(m+j)) }}}\\
&=&\frac{2\Gamma(4m-4)}{\Gamma(2m-1)^{2}\Gamma(2m+1)}
\frac{ \prod_{j=0}^{m-1}{\frac{\Gamma(1+ 2j)^{2}\Gamma(3+2j)}{2\Gamma(2(m+j)-2) }}}{ \prod_{j=0}^{m-1}{\frac{\Gamma(1+ 2j)^{2}
\Gamma(3+2j)}{2\Gamma(2(m+j)) }}}\\
&=&\frac{2\Gamma(4m-4)}{\Gamma(2m-1)^{2}\Gamma(2m+1)}
 \prod_{j=0}^{m-1}{\frac{\Gamma(2(m+j)) }{\Gamma(2(m+j)-2)}} \\
&=&\frac{2\Gamma(4m-4)}{\Gamma(2m-1)^{2}\Gamma(2m+1)}
 \frac{\Gamma(4m-2) }{\Gamma(2m-2)} \, .
\end{eqnarray*}
To bound
$\frac{2\Gamma(4m-4)\Gamma(4m-2)}{\Gamma(2m-1)^{2}\Gamma(2m)\Gamma(2m-2)}$ from above we use
the binomial relation (see e.g.~\cite[Eq.~6.1.21]{AbramowitsStegun}) for the Gamma function
\begin{equation}
\label{binomial_gamma}
\binom{z}{w}=\frac{\Gamma(z+1)}{\Gamma(w+1)\Gamma(z-w+1)}
\end{equation}
and the inequality (see e.g.~\cite[Eq.~C.5]{CormenLeisersonRivestStein})
\begin{equation}
\label{binomial_upperbound}
\bigl(\frac{z}{w}\bigr)^{w} \leq \binom{z}{w}\leq \bigl(\frac{ez}{w}\bigr)^{w}
\end{equation}
to obtain
\begin{eqnarray*}
\frac{\Gamma(4m-4)\Gamma(4m-2)}{\Gamma(2m-1)^{2}\Gamma(2m)\Gamma(2m-2)}&=&
\frac{\Gamma(4m-4)}{\Gamma(2m-1)\Gamma(2m-2)}\cdot \frac{\Gamma(4m-2)}{\Gamma(2m-1)\Gamma(2m)}\\
&=& \binom{4m-5}{2m-2}\binom{4m-3}{2m-2}\\
&=& \binom{4m-5}{2m-2}\binom{4m-3}{2m-1}\\
&\leq& \bigl(e\frac{4m-5}{2m-2}\bigr)^{2m-2}\bigl(e\frac{4m-3}{2m-1}\bigr)^{2m-1}\\
&\leq& \bigl(2e)^{2m-2}\bigl(2e\bigr)^{2m-1}\\
&=& \bigl(2e)^{4m-3}\\
&\leq &\frac{1}{2} \bigl(2e\bigr)^{4m-2}\, .
\end{eqnarray*}
Recalling \eqref{e11111} establishes the assertion for $n=2m-1$.

Now consider the even case $n-2m$.
We utilize the bijective principal representation $\phi^{p}_{el}: \Lambda^{m} \times T^{m} \rightarrow Int(M^{2m})$ defined in \eqref{def_phi_pel}  and, proceeding as in the odd case, we obtain
\begin{eqnarray*}
\frac{Vol\Bigl(M^{2m}_{\d}\Bigr)}{Vol\Bigl(M^{2m}\Bigr)}
&\geq& 1- 2m\d \frac{S_{m-1}(3,1,2)}{S_{m}(3,1,2)}\\
\end{eqnarray*}
Using Selberg's formulas \eqref{selbergformula} we compute
\begin{eqnarray*}
\frac{S_{m-1}(3,1,2)}{S_{m}(3,1,2)}&=&
\frac{ \prod_{j=0}^{m-2}{\frac{\Gamma(1+ 2j)
\Gamma(3+2j)^{2}}{2\Gamma(2(m+j)) }}}{ \prod_{j=0}^{m-1}{\frac{\Gamma(1+ 2j)
\Gamma(3+2j)^{2}}{2\Gamma(2(m+j)+2) }}}\\
&=&\frac{2\Gamma(4m-2)}{\Gamma(2m-1)^{2}\Gamma(2m+1)}
\frac{ \prod_{j=0}^{m-1}{\frac{\Gamma(1+ 2j)
\Gamma(3+2j)^{2}}{2\Gamma(2(m+j)) }}}{ \prod_{j=0}^{m-1}{\frac{\Gamma(1+ 2j)
\Gamma(3+2j)^{2}}{2\Gamma(2(m+j)+2) }}}\\
&=&\frac{2\Gamma(4m-2)}{\Gamma(2m-1)^{2}\Gamma(2m+1)}
 \prod_{j=0}^{m-1}{\frac{\Gamma(2(m+j)+2) }{\Gamma(2(m+j))}} \\
&=&\frac{2\Gamma(4m-2)}{\Gamma(2m-1)^{2}\Gamma(2m+1)}
  \frac{\Gamma(4m)}{\Gamma(2m)}\, .
\end{eqnarray*}

We will now use the  Beta function
           \begin{equation}\label{eq:beta1}
            B(a,b):=\int_{0}^{1}{t^{a-1}(1-t)^{b-1}dt},   \quad a >0, b>0\,,
            \end{equation}
          and the identity
           \begin{equation}\label{eq:beta2} B(a,b) =\frac{\Gamma(a)\Gamma(b)}{\Gamma(a+b)}\, ,\end{equation}
           see e.g.~\cite[Pg.~258]{AbramowitsStegun},
where $\Gamma$ is the gamma function.

To bound
$\frac{2\Gamma(4m-2)\Gamma(4m)}{\Gamma(2m-1)^{2}\Gamma(2m)^{2}}$ from above we use
the inequality
\[ B(a,a) =\frac{\Gamma(a)^{2}}{\Gamma(2a)} \geq \frac{4}{a}2^{-2a}\]
from Proposition \ref{beta_lowerbound}
to obtain
\begin{eqnarray*}
\frac{\Gamma(4m-2)\Gamma(4m)}{\Gamma(2m-1)^{2}\Gamma(2m)^{2}}&=&
\frac{1}{B(2m-1,2m-1)}\frac{1}{B(2m,2m)}\\
&\leq & \frac{2m-1}{4}2^{4m-2}\frac{2m}{4}2^{4m}\\
&\leq & \frac{2m}{4}2^{4m-2}\frac{2m}{4}2^{4m}\\
&\leq&\frac{1}{16} m^{2}2^{8m}
\end{eqnarray*}
Finally, we apply the inequality
\[ m^{2} \leq 8 (\frac{e}{2})^{4m}\]
from Proposition \ref{prop_tech} to conclude that
\begin{eqnarray*}
\frac{2\Gamma(4m-2)\Gamma(4m)}{\Gamma(2m-1)^{2}\Gamma(2m)^{2}}
&\leq&\frac{1}{8} m^{2}2^{8m}\\
&\leq&(\frac{e}{2})^{4m}2^{8m}\\
&=&\bigl(2e\bigr)^{4m}
\end{eqnarray*}
thus establishing the assertion for $n=2m$.

\subsection{Proof of Lemma \ref{massfirstmoment}}

According to Chang, Kemperman, Studden \cite[Thm.~1.3]{ChangKempermanStudden} one can show, using
Skibinsky's canonical coordinates for the moment
problem \cite{Skibinsky}, that the uniform distribution on $M^{n}$ marginalizes to  a  Beta distribution corresponding to
$B(n,n)$ (see \eqref{eq:beta1} and \eqref{eq:beta2}) on the first moment. Consequently,
\begin{eqnarray*}
\frac{Vol\bigl( q \in M^{n}: q_{1} \in [1-\d,1]\bigr)}{Vol(M^{n})}&=&
\frac{1}{B(n,n)}\int_{1-\d}^{1}{t^{n-1}(1-t)^{n-1}dt}\\
&=& I_{\d}(n,n)
\end{eqnarray*}
where $I_{\d}(n,n)$ is the  Incomplete Beta function (see  e.g.~\cite[Pg.~258]{AbramowitsStegun}).
 Using the binomial relations \eqref{binomial_gamma} and
and \eqref{binomial_upperbound}, for the upper bound  we obtain
\begin{eqnarray*}
I_{\d}(n,n)&:=&\frac{1}{B(n,n)}\int_{0}^{\d}{t^{n-1}(1-t)^{n-1}dt}\\
& \leq &
\frac{1}{B(n,n)}\int_{0}^{\d}{t^{n-1}dt}\\
&=& \frac{\d^{n}}{nB(n,n)}\\
&=&  \d^{n} \frac{\Gamma(2n)}{n\Gamma(n)^{2}}\\
&=&  \d^{n} \binom{2n-1}{n}\\
&\leq &  \d^{n} \Bigl(e \frac{2n-1}{n}\Bigr)^{n}\\
&\leq &  \d^{n} \bigl(2e\bigr)^{n}
\end{eqnarray*}
and for the lower bound
\begin{eqnarray*}
I_{\d}(n,n)&:=&\frac{1}{B(n,n)}\int_{0}^{\d}{t^{n-1}(1-t)^{n-1}dt}\\
& \geq & \frac{1}{B(n,n)}\bigl(1-\d)^{n-1}\int_{0}^{\d}{t^{n-1}dt}\\
& = & \frac{1}{nB(n,n)}\bigl(1-\d)^{n-1}\d^{n}\\
& = & \binom{2n-1}{n-1} \bigl(1-\d)^{n-1}\d^{n}\\
& \geq & \Bigl(\frac{2n-1}{n-1}\Bigr)^{n-1} \bigl(1-\d)^{n-1}\d^{n}\\
& \geq & 2^{n-1} \bigl(1-\d)^{n-1}\d^{n}\\
& \geq & \d^{n}
\end{eqnarray*}
where the assumption $\d \leq \frac{1}{2}$  was used in the last step.

\subsection{Proof of Proposition \ref{prop_Jco}}

The following identity of
Karlin and Shapley \cite[Proof of Thm.~6.2]{KarlinStudden:1966} will be useful in all the Jacobian determinant calculations of this paper:
For $t_{1} < s_{1} < \cdots < t_{m} < s_{m}$, we have
\begin{equation}
\label{karlinstudden}
 \frac{\partial^{m}}{\partial s_{1}\cdots \partial s_{m}}
\Delta(t_{1},s_{1},..,
t_{m},s_{m})|_{(s_{1},..,s_{m}) =
(t_{1},..,t_{m})}=
\Delta^{4}_{m}(t)\, .
\end{equation}

We can develop the upper and lower configurations simultaneously, by introducing
a point $t_{0} \in \{0,1\}$ and representations $\phi_{t_{0}}$ where when $t_{0}=0$
we have
$\phi_{0}=\phi^{c}_{ol}$ defined in \eqref{def_phi_col} and when $t_{0}=1$ we have $\phi_{1}=\phi^{c}_{ou}$
defined in \eqref{def_phi_cou}.
So, let us use this notation and  a change of indices,  and consider the two   maps
\[ \phi_{t_{0}}:\Lambda^{m}\times T^{m-1} \rightarrow Int(M^{2m-1}),\quad  t_{0}=0,1 \] defined by
\begin{eqnarray}
\label{phi_t0}
 \phi_{t_{0}}(\lambda,t)
&= &\Psi  \Bigl(  \sum_{j=0}^{m-1}{\lambda_{j}\delta_{t_{j}}}   +(1-\sum_{j=0}^{m-1}{\lambda_{j}})\delta_{t_{*}}\Bigr)\notag\\
&=& \Bigl(\sum_{j=0}^{m-1}{\lambda_{j}t_{j}^{i}}+ (1-\sum_{j=0}^{m-1}{\lambda_{j})t_{*}^{i}}\Bigr)_{i=1}^{2m-1}\, .
\end{eqnarray}
In this notation,  Proposition \ref{prop_Jco}  becomes
\begin{prop}
\label{prop_Jco_2}
For $t_{0}=0,1$ we have
\[ |det(d\phi_{t_{0}})|(\lambda,t)= \mathcal{J}_{t_{0}}(t)
\prod_{j=1}^{m-1}{\lambda_{j}} \]
where
\begin{equation*}
\label{Jac_imp0}
\mathcal{J}_{t_{0}}(t)=
 \bigl|t_{0}-t_{*}\bigr| \prod_{j=1}^{m-1}{(t_{j}-t_{*})^{2}}\prod_{j=1}^{m-1}{(t_{j}-t_{0})^{2}}
\cdot \Delta_{m-1}^{4}(t)
\end{equation*}
\end{prop}

The differential of $\phi_{t_{0}}$ is determined by
\[\frac{ \partial \phi^{i}_{t_{0}}}{\partial \lambda_{j}}=  t^{i}_{j}-t^{i}_{*}, \quad j=0,..,m-1
\]
and
\[ \frac{ \partial \phi^{i}_{t_{0}}}{\partial t_{j}}= i\lambda_{j}t^{i-1}_{j}  \quad j=1,..,m-1
\]
for $i=1,..,2m-1,$
from which we conclude that
\[ |det(d\phi_{t_{0}})|= |\mathcal{J}_{t_{0}}|
\prod_{j=1}^{m-1}{\lambda_{j}} \]
where
\begin{eqnarray*}
\mathcal{J}_{t_{0}}&=&
\begin{vmatrix}  t_{0}-t_{*} & t_{1}-t_{*} & 1 &  \cdots & t_{m-1}-t_{*} & 1   \\
t^{2}_{0}-t^{2}_{*}& t^{2}_{1}-t^{2}_{*} & 2t_{1}  & \cdots &   t^{2}_{m-1}-t^{2}_{*} & 2t_{m-1}  \\
\vdots    &  \vdots & \vdots    &  \cdots & \vdots    &  \vdots    \\
t_{0}^{2m-1}-t_{*}^{2m-1}& t^{2m-1}_{1}-t^{2m-1}_{*} & (2m-1)t^{2m-2}_{1}&\cdots  & t^{2m-1}_{m-1}-t^{2m-1}_{*}
  &   (2m-1)t^{2m-2}_{m-1}
\end{vmatrix}
\\
& &\\
&=&
\begin{vmatrix}
1&    1&    1&     0&    \cdots &1     & 0\\
t_{*}&  t_{0} & t_{1} & 1 &  \cdots & t_{m-1} & 1   \\
t_{*}^{2} & t^{2}_{0}& t^{2}_{1} & 2t_{1}  & \cdots &   t^{2}_{m-1} & 2t_{m-1}  \\
\vdots& \vdots&\vdots    &  \vdots & \cdots    &  \vdots & \vdots    \\
t_{*}^{2m-1} & t_{0}^{2m-1}& t^{2m-1}_{1} & (2m-1)t^{2m-2}_{1}&\cdots  & t^{2m-1}_{m-1}
  &   (2m-1)t^{2m-2}_{m-1}
\end{vmatrix}
\end{eqnarray*}

To
 evaluate $\mathcal{J}_{t_{0}}$ for $t_{0}=0,1$,
 let
 $s_{1},..,s_{m-1}$ satisfy
$t_{j} < s_{j} < t_{j+1}, j=1,..,m-1$ and define the Vandermonde determinant
\begin{equation*}
\label{Jaux}
\mathcal{J}(s_{1},..,s_{m-1}):=
\begin{vmatrix}
1&    1&    1&     s_{1}^{0}&    \cdots &1     &  s_{m-1}^{0}\\
t_{*}&  t_{0} & t_{1} & s_{1} &  \cdots & t_{m-1} & s_{m-1}   \\
t_{*}^{2} & t^{2}_{0}& t^{2}_{1} & s^{2}_{1}  & \cdots &   t^{2}_{m-1} & s^{2}_{m-1}  \\
\vdots& \vdots&\vdots    &  \vdots & \cdots    &  \vdots & \vdots    \\
t_{*}^{2m-1} & t_{0}^{2m-1}& t^{2m-1}_{1} &s^{2m-1}_{1} &\cdots  & t^{2m-1}_{m-1}
  &  s^{2m-1}_{m-1}
\end{vmatrix}
\end{equation*}
and observe that the multilinearity of the determinant shows that
\[
 \mathcal{J}_{t_{0}}=\frac{\partial^{m-1}}{\partial s_{1}\cdots \partial s_{m-1}}\mathcal{J}(s_{1},..,s_{m-1}) |_{(s_{1},..,s_{m-1})=
(t_{1},..,t_{m-1})}
\]
To evaluate this differentiation,
observe that
\begin{eqnarray*}
\mathcal{J}(s_{1},..,s_{m-1})&=&\Delta(t_{*}, t_{0}, t_{1},s_{1},...,t_{m-1},s_{m-1})\\
&=& \bigl(t_{0}-t_{*}\bigr)\Bigl(\prod_{j=1}^{m-1}{(t_{j}-t_{*})(s_{j}-t_{*})}\Bigr)
\Bigl(\prod_{j=1}^{m-1}{(t_{j}-t_{0})(s_{j}-t_{0})}\Bigr) \\
&&\Delta(t_{1},s_{1},...,t_{m-1},s_{m-1})\, ,
\end{eqnarray*}
from which we conclude that
\begin{eqnarray*}
\mathcal{J}_{t_{0}}
&=& \frac{\partial^{m-1}}{\partial s_{1}\cdots \partial s_{m-1}}
\mathcal{J}(s_{1},..,
s_{m-1})|_{(s_{1},..,s_{m-1})=(t_{1},..,t_{m-1})}\\
&=&  \bigl(t_{0}-t_{*}\bigr)\cdot \prod_{j=1}^{m-1}{(t_{j}-t_{*})^{2}}
\prod_{j=1}^{m-1}{(t_{j}-t_{0})^{2}}\cdot
\frac{\partial^{m-1}}{\partial s_{1}\cdots \partial s_{m-1}}\\&& \Delta(t_{1},s_{1},...,t_{m-1},s_{m-1})|_{(s_{1},..,s_{m-1})=(t_{1},..,t_{m-1})}
\end{eqnarray*}
Using the  identity
\eqref{karlinstudden}
we conclude that
\[\mathcal{J}_{t_{0}}=
 \bigl(t_{0}-t_{*}\bigr) \prod_{j=1}^{m-1}{(t_{j}-t_{*})^{2}}\prod_{j=1}^{m-1}{(t_{j}-t_{0})^{2}}
\Delta_{m-1}^{4}(t)
\]
thereby proving Proposition \ref{prop_Jco_2} and therefore Proposition \ref{prop_Jco}.

\subsection{Proof of Proposition \ref{prop_Jce}}

To simplify notation, let $\phi_{l}:=\phi^{c}_{el}$ defined in \eqref{def_phi_cel}
 and $\phi_{u}:=\phi^{c}_{eu}$ defined in \eqref{def_phi_ceu}.
We begin with the lower representation $\phi_{l}$.
The differential of $\phi_{l}$ is determined by
\[\frac{ \partial \phi^{i}_{l}}{\partial \lambda_{j}}=  t_{j}^{i}-t^{i}_{*}, \quad j=1,..,m\, ,
\]
and
\[ \frac{ \partial \phi^{i}_{l}}{\partial t_{j}}= i\lambda_{j}t^{i-1}_{j}  \quad j=1,..,m
\]
for $i=1,..,2m,$
from which we conclude that
\[
 |det(d\phi_{l})|= |\mathcal{J}_{l}|
\prod_{j=1}^{m}{\lambda_{j}}
\]
where
\begin{eqnarray*}
\mathcal{J}_{l}&=&
\begin{vmatrix}   t_{1}-t_{*} & 1 &  \cdots & t_{m}-t_{*} & 1    \\
 t^{2}_{1}-t^{2}_{*} & 2t_{1}  & \cdots &   t^{2}_{m}-t^{2}_{*} & 2t_{m}  \\
\vdots    &  \vdots & \vdots    &  \cdots & \vdots        \\
 t^{2m}_{1}-t^{2m}_{*} & 2mt^{2m-1}_{1}&\cdots  & t^{2m}_{m}-t^{2m}_{*}
  &   2mt^{2m-1}_{m}
\end{vmatrix}
\\
& &\\
&=&
\begin{vmatrix}
1  & 1& 0 & \cdots & 1 & 0 \\
  t_{*} & t_{1} & 1 &  \cdots & t_{m} & 1    \\
t^{2}_{*}& t^{2}_{1} & 2t_{1}  & \cdots &   t^{2}_{m} & 2t_{m} \\
\vdots    &  \vdots & \vdots    &  \cdots & \vdots    &  \vdots    \\
t_{*}^{2m}& t^{2m}_{1} & 2mt^{2m-1}_{1}&\cdots  & t^{2m}_{m}
  &   2mt^{2m-1}_{m}
\end{vmatrix}
\\
&
\end{eqnarray*}

To evaluate
$\mathcal{J}_{l}$,
 let
 $s_{1},..,s_{m}$ satisfy
$t_{j} < s_{j} < t_{j+1}, j=1,..,m$ and define
\[
\mathcal{J}(s_{1},..,s_{m}):=
\begin{vmatrix}
 1 & 1 & 1& \cdots & 1 & 1\\
  t_{*} & t_{1} & s_{1} &  \cdots & t_{m} & s_{m}    \\
t^{2}_{*}& t^{2}_{1} & s^{2}_{1}  & \cdots &   t^{2}_{m} & s^{2}_{m}  \\
\vdots    &  \vdots & \vdots    &  \cdots & \vdots        \\
t_{*}^{2m}& t^{2m}_{1} & s^{2m}_{1}&\cdots  & t^{2m}_{m}
  &   s^{2m}_{m}
\end{vmatrix}
\]
and observe that the multilinearity of the determinant shows that
\[
 \mathcal{J}_{l}=\frac{\partial^{m}}{\partial s_{1}\cdots \partial s_{m}}\mathcal{J}(s_{1},..,s_{m}) |_{(s_{1},..,s_{m})=
(t_{1},..,t_{m})}
\]
To evaluate this differentiation, observe
that
\[\mathcal{J}(s_{1},..,s_{m})=\Delta(t_{*},t_{1},s_{1},..,t_{m},s_{m})\]
and, using
 the recursion relation of the Vandermonde determinant, we obtain
\begin{eqnarray*}
 \mathcal{J}(s_{1},..,s_{m})&=&\Delta(t_{*},t_{1},s_{1},..,t_{m},s_{m})\\
&=&\prod_{j=1}^{m}{(t_{j}-t_{*})(s_{j}-t_{*})} \cdot \Delta(t_{1},s_{1},..,t_{m},
s_{m})
\end{eqnarray*}
from which we conclude that
\begin{eqnarray*}
&& \frac{\partial^{m}}{\partial s_{1}\cdots \partial s_{m}}
\mathcal{J}(s_{1},..,s_{m}) |_{(s_{1},..,s_{m}) =
(t_{1},..,t_{m})}\\
& =&
 \prod_{j=1}^{m}{(t_{j}-t_{*})^{2}} \cdot \frac{\partial^{m}}{\partial s_{1}\cdots \partial s_{m}}
\Delta(t_{1},s_{1},..,t_{m},
s_{m})|_{(s_{1},..,s_{m}) =
(t_{1},..,t_{m})}\\
&=& \prod_{j=1}^{m}{(t_{j}-t_{*})^{2}}\cdot \Delta_{m}^{4}(t)
\end{eqnarray*}
and therefore
\[
\mathcal{J}_{l}
= \prod_{j=1}^{m}{(t_{j}-t_{*})^{2}}\cdot \Delta_{m}^{4}(t)
\] thus establishing the lower identity.

Now, for the upper representation $\phi_{u}:=\phi^{c}_{eu}$,
the differential of $\phi_{u}$ is determined by
\[\frac{ \partial \phi^{i}_{u}}{\partial \lambda_{0}}=  -t^{i}_{*}, \, ,
\]
\[\frac{ \partial \phi^{i}_{u}}{\partial \lambda_{j}}=  t_{j}^{i}-t^{i}_{*}, \quad j=1,..,m-1\, ,
\]
\[\frac{ \partial \phi^{i}_{u}}{\partial \lambda_{m}}=  1-t^{i}_{*}\, ,
\]
and
\[ \frac{ \partial \phi^{i}_{u}}{\partial t_{j}}= i\lambda_{j}t^{i-1}_{j}  \quad j=1,..,m-1
\]
for $i=1,..,2m,$
from which we conclude that
\[
 |det(d\phi_{u})|= |\mathcal{J}_{u}|
\prod_{j=1}^{m-1}{\lambda_{j}}
\]
where
\begin{eqnarray*}
\mathcal{J}_{u}&=&
\begin{vmatrix}  -t_{*} & t_{1}-t_{*} & 1 &  \cdots & t_{m-1}-t_{*} & 1 & 1-t_{*}   \\
-t^{2}_{*}& t^{2}_{1}-t^{2}_{*} & 2t_{1}  & \cdots &   t^{2}_{m-1}-t^{2}_{*} & 2t_{m-1}& 1-t^{2}_{*}  \\
\vdots    &  \vdots & \vdots    &  \cdots & \vdots    &  \vdots    \\
-t_{*}^{2m}& t^{2m}_{1}-t^{2m}_{*} & 2mt^{2m-1}_{1}&\cdots  & t^{2m}_{m-1}-t^{2m}_{*}
  &   2mt^{2m-1}_{m-1}& 1-t_{*}^{2m}
\end{vmatrix}
\\
& &\\
&=&
\begin{vmatrix}  t_{*} & t_{1} & 1 &  \cdots & t_{m-1} & 1 & 1   \\
t^{2}_{*}& t^{2}_{1} & 2t_{1}  & \cdots &   t^{2}_{m-1} & 2t_{m-1}& 1  \\
\vdots    &  \vdots & \vdots    &  \cdots & \vdots    &  \vdots    \\
t_{*}^{2m}& t^{2m}_{1} & 2mt^{2m-1}_{1}&\cdots  & t^{2m}_{m-1}
  &   2mt^{2m-1}_{m-1}& 1
\end{vmatrix}
\\
&
\end{eqnarray*}

To evaluate  $\mathcal{J}_{u}$,
 let
 $s_{1},..,s_{m-1}$ satisfy
$t_{j} < s_{j} < t_{j+1}, j=1,..,m-1$ and define
\[
\mathcal{J}(s_{1},..,s_{m-1}):=
\begin{vmatrix}  t_{*} & t_{1} & s_{1} &  \cdots & t_{m-1} & s_{m-1} & 1   \\
t^{2}_{*}& t^{2}_{1} & s^{2}_{1}  & \cdots &   t^{2}_{m-1} & s^{2}_{m-1}& 1  \\
\vdots    &  \vdots & \vdots    &  \cdots & \vdots    &  \vdots    \\
t_{*}^{2m}& t^{2m}_{1} & s^{2m}_{1}&\cdots  & t^{2m}_{m-1}
  &   s^{2m}_{m-1}& 1
\end{vmatrix}
\]
and observe that the multilinearity of the determinant shows that
\[
 \mathcal{J}_{u}=\frac{\partial^{m-1}}{\partial s_{1}\cdots \partial s_{m-1}}\mathcal{J}(s_{1},..,s_{m-1}) |_{(s_{1},..,s_{m-1})=
(t_{1},..,t_{m-1})}
\]
To evaluate this differentiation,  observe that
\[
 \mathcal{J}(s_{1},..,s_{m-1})= t_{*} \prod_{j=1}^{m-1}{t_{j}}\prod_{j=1}^{m-1}{s_{j}}
\begin{vmatrix} 1& 1 & 1 & 1  & 1 & \cdots & 1 & 1 &  1  \\
t_{*}& t_{1} &s_{1} & t_{2}  &  s_{2} & \cdots & t_{m-1} & s_{m-1} &  1  \\
\vdots    &  \vdots & \vdots    &  \vdots & \cdots    &  \vdots   & \vdots & \vdots \\
t^{2m-1}_{*} &t^{2m-1}_{1} & s^{2m-1}_{1}  & t^{2m-1}_{2}  &   s^{2m-1}_{2} & \cdots & t^{2m-1}_{m-1} &
 s^{2m-1}_{m-1} &
1
\end{vmatrix}
\]
That is, we have
\[
 \mathcal{J}(s_{1},..,s_{m-1})= t_{*} \prod_{j=1}^{m-1}{t_{j}}\prod_{j=1}^{m-1}{s_{j}} \cdot \Delta(t_{*},t_{1},s_{1},..,t_{m-1},
s_{m-1},1) \, .
\]
We use the recursion relations
\[ \Delta(t_{*},t_{1},s_{1},..,t_{m-1},
s_{m-1},1)=(1-t_{*})\prod_{j=1}^{m-1}{(1-t_{j})(1-s_{j})} \cdot \Delta(t_{*},t_{1},s_{1},..,t_{m-1},
s_{m-1})\]
and
 \[ \Delta(t_{*},t_{1},s_{1},..,t_{m-1},
s_{m-1})=\prod_{j=1}^{m-1}{(t_{j}-t^{*})(s_{j}-t^{*})} \cdot \Delta(t_{1},s_{1},..,t_{m-1},
s_{m-1})\] to obtain
\[
 \mathcal{J}(s_{1},..,s_{m-1})= t_{*}(1-t_{*}) \prod_{j=1}^{m-1}{t_{j}(1-t_{j})s_{j}(1-s_{j})}\cdot \prod_{j=1}^{m-1}
{(t_{j}-t_{*})(s_{j}-t_{*})} \cdot \Delta(t_{1},s_{1},..,t_{m-1},
s_{m-1})\, .\]
Consequently, the identity \eqref{karlinstudden} implies
\begin{eqnarray*}
\mathcal{J}_{u}
&=& \frac{\partial^{m-1}}{\partial s_{1}\cdots \partial s_{m-1}}\mathcal{J}(s_{1},..,s_{m-1}) |_{(s_{1},..,s_{m-1}) =
(t_{1},..,t_{m-1})}\\
& =&
 t_{*}(1-t_{*}) \prod_{j=1}^{m-1}{t^{2}_{j}(1-t_{j})^{2}} \prod_{j=1}^{m-1}{(t_{j}-t_{*})^{2}} \cdot \frac{\partial^{m-1}}{\partial s_{1}\cdots \partial s_{m-1}}
\\&&\Delta(t_{*},t_{1},s_{1},..,t_{m-1},
s_{m-1})|_{(s_{1},..,s_{m-1}) =
(t_{1},..,t_{m-1})}\\
& =&
 t_{*}\bigl(1-t_{*}\bigr) \prod_{j=1}^{m-1}{(t_{j}-t_{*})^{2}}\prod_{j=1}^{m-1}{t_{j}^{2}(1-t_{j})^{2}}
\cdot \Delta_{m-1}^{4}(t)
\end{eqnarray*}
establishing the upper identity and thus completing the proof.

\subsection{Proof of Proposition \ref{prop_mom}}
For the first assertion, let $n=2m-1$.
As in the proof of Proposition \ref{prop_Jco} we find it convenient to analyze the upper and lower configurations simultaneously, by introducing
a point $t_{0} \in \{0,1\}$ and the volume filling representations $\phi_{t_{0}},t_{0} \in \{0,1\}$ where when $t_{0}=0$
we have
$\phi_{0}=\phi^{c}_{ol}$ defined in \eqref{def_phi_col} and when $t_{0}=1$ we have $\phi_{1}=\phi^{c}_{ou}$
defined in \eqref{def_phi_cou}.
In this notation,  from \eqref{phi_t0} we have
\begin{eqnarray}
 \phi^{i}_{t_{0}}(\lambda_{0},..,\lambda_{m-1};t_{1},..,t_{m-1})
&= &\sum_{j=0}^{m-1}{\lambda_{j}t_{j}^{i}}+ (1-\sum_{j=0}^{m-1}{\lambda_{j})t_{*}^{i}}\notag \\
&= &\sum_{j=0}^{m-1}{\lambda_{j}\bigl(t_{j}^{i}-t_{*}^{i}\bigr)}+ t_{*}^{i}\notag \\
&= &\lambda_{0}\bigl(t_{0}^{i}-t_{*}^{i}\bigr)+\sum_{j=1}^{m-1}{\lambda_{j}\bigl(t_{j}^{i}-t_{*}^{i}\bigr)}+ t_{*}^{i}
\end{eqnarray}
for $t_{0}=0,1$ and
 Proposition \ref{prop_Jco_2} expresses  the Jacobian determinants as
\[ |det(d\phi_{t_{0}})|(\lambda,t)= \mathcal{J}_{t_{0}}(t)
\prod_{j=1}^{m-1}{\lambda_{j}} \]
where
\begin{equation}
\label{Jac_imp0a}
\mathcal{J}_{t_{0}}(t)=
 \bigl|t_{0}-t_{*}\bigr| \prod_{j=1}^{m-1}{(t_{j}-t_{*})^{2}}\prod_{j=1}^{m-1}{(t_{j}-t_{0})^{2}}
\cdot \Delta_{m-1}^{4}(t)\, .
\end{equation}
In this notation, the modified change of variables formula \eqref{id_cov2} becomes
\begin{eqnarray}
\label{eq_odd_cov1}
\int_{M^{2m-1}}{q_{i}}
&=&\sum_{t_{0}=1,2}\int_{\Lambda^{m}\times T^{m-1}}{\phi_{t_{0}}^{i}|d\phi_{t_{0}}|}\notag\\
\end{eqnarray}
for $i \geq 0$.
Therefore, we conclude that
\begin{eqnarray}
\label{eq_odd_cov2}
\int_{M^{2m-1}}{q_{i}}
&=&\sum_{t_{0}=1,2}\int_{\Lambda^{m}\times T^{m-1}}{\phi_{t_{0}}^{i}|d\phi_{t_{0}}|}\notag\\
&=&\sum_{t_{0}=1,2}\int_{\Lambda^{m}\times T^{m-1}}{\phi_{t_{0}}^{i}\bigl(\prod_{j=1}^{m-1}{\lambda_{j}}\bigr)
\mathcal{J}_{t_{0}}}\notag\\
&=&\sum_{t_{0}=1,2}\int_{T^{m-1}}{\Bigl(\int_{\Lambda^{m}}\phi_{t_{0}}^{i}\prod_{j=1}^{m-1}{\lambda_{j}d\lambda}\Bigr)\mathcal{J}_{t_{0}}
}\notag\\
\end{eqnarray}
Performing the $\Lambda^{m}$ integration, using the identities
$\int_{\Lambda^{m}}{\lambda_{m-1}^{2}\prod_{i=1}^{m-2}{\lambda_{i}}d\lambda} = \frac{2}{(2m)!}$,
$\int_{\Lambda^{m}}{\prod_{i=1}^{m}{\lambda_{i}}d\lambda} = \frac{1}{(2m)!}$, and $\int_{\Lambda^{m}}{\prod_{i=1}^{m-1}{\lambda_{i}}d\lambda} = \frac{1}{(2m-1)!}$,
we obtain
\begin{eqnarray}
\int_{\Lambda^{m}}{\phi_{t_{0}}^{i}\prod_{j=1}^{m-1}{\lambda_{j}}}
&=& \int_{\Lambda^{m}}{\Bigl(\lambda_{0}\bigl(t_{0}^{i}-t_{*}^{i}\bigr)+\sum_{j=1}^{m-1}{\lambda_{j}\bigl(t_{j}^{i}-t_{*}^{i}\bigr)}+ t_{*}^{i}\Bigr)\prod_{j=1}^{m-1}{\lambda_{j}}d\lambda}\notag\\
&=& \frac{1}{(2m)!}\bigl(t_{0}^{i}-t_{*}^{i}\bigr)+\frac{2}{(2m)!}\sum_{j=1}^{m-1}{\bigl(t_{j}^{i}-t_{*}^{i}\bigr)}+
\frac{1}{(2m-1)!} t_{*}^{i}\notag\\
&=& \frac{1}{(2m)!}t_{0}^{i}+\frac{2}{(2m)!}\sum_{j=1}^{m-1}{t_{j}^{i}}+\frac{1}{(2m)!} t_{*}^{i}\notag\, .
\end{eqnarray}

Consequently, for $i\geq 1$, we have
\begin{eqnarray}
\label{eq_odd_cov3}
\int_{M^{2m-1}}{q_{i}}
&=&\sum_{t_{0}=1,2}\int_{T^{m-1}}{\Bigl(\int_{\Lambda^{m}}\phi_{t_{0}}^{i}\prod_{j=1}^{m-1}{\lambda_{j}}d\lambda\Bigr)\mathcal{J}_{t_{0}}
}\notag\\
&=&\sum_{t_{0}=1,2}\int_{T^{m-1}}{\Bigl(\frac{1}{(2m)!}t_{0}^{i}+\frac{2}{(2m)!}\sum_{j=1}^{m-1}{t_{j}^{i}}+
\frac{1}{(2m)!} t_{*}^{i}\Bigr)\mathcal{J}_{t_{0}}
}\notag\\
&=&\frac{1}{(2m)!}\int_{T^{m-1}}
{\mathcal{J}_{1}}+\frac{1}{(2m)!} t_{*}^{i}\int_{T^{m-1}}{\Bigl(\mathcal{J}_{0}+
\mathcal{J}_{1}\Bigr)}+\frac{2}{(2m)!}
 \int_{T^{m-1}}{\Bigl(\sum_{j=1}^{m-1}{t_{j}^{i}}\Bigr)\Bigl(\mathcal{J}_{0}+
\mathcal{J}_{1}\Bigr)
}\notag
\end{eqnarray}
and, for $i=0$
\begin{eqnarray}
\label{eq_odd_sum1}
\int_{M^{2m-1}}{q_{0}}
&=&\frac{1}{(2m-1)!}\int_{T^{m-1}}
{\Bigl(\mathcal{J}_{0}+\mathcal{J}_{1}\Bigr)}
\end{eqnarray}
which  we already knew from \eqref{vol_Jco}.
Combining the two, we  obtain for $i\geq 0$
\begin{eqnarray}
\label{eq_odd_cov4}
\int_{M^{2m-1}}{q_{i}}
&=&\frac{\d_{0}(i)}{(2m)!}\int_{T^{m-1}}
{\mathcal{J}_{0}}+\frac{1}{(2m)!}\int_{T^{m-1}}
{\mathcal{J}_{1}}
+\frac{1}{(2m)!} t_{*}^{i}\int_{T^{m-1}}{\Bigl(\mathcal{J}_{0}+
\mathcal{J}_{1}\Bigr)}\notag\\
&+&\frac{2}{(2m)!}
 \int_{T^{m-1}}{\Bigl(\sum_{j=1}^{m-1}{t_{j}^{i}}\Bigr)\Bigl(\mathcal{J}_{0}+
\mathcal{J}_{1}\Bigr)
}
\end{eqnarray}
and the substitution  of the volume equality
 \eqref{eq_odd_sum1} (that is, \eqref{vol_Jco}) yields the assertion in the odd case.

For the even case, let $n=2m$, and let us simplify notation by denoting the volume filling representations by
$\phi_{1}:=\phi^{c}_{el}$  and $\phi_{2}:=\phi^{c}_{eu}$ defined in  \eqref{def_phi_cel} and
\eqref{def_phi_ceu} so that, in this notation,
\[ \phi_{1}:\Lambda^{m}\times T^{m} \rightarrow Int(M^{2m})\]  is defined by
\begin{eqnarray*}
 \phi_{1}(\lambda_{1},..,\lambda_{m};t_{1},..,t_{m})
&=& \Bigl(\sum_{j=1}^{m}{\lambda_{j}t_{j}^{i}}+ (1-\sum_{j=1}^{m}{\lambda_{j})t_{*}^{i}}\Bigr)_{i=1}^{2m}
\end{eqnarray*}
and
 \[ \phi_{2}:\Lambda^{m+1}\times T^{m-1} \rightarrow Int(M^{2m})\]  by
\begin{eqnarray*}
 \phi_{2}(\lambda_{0},..,\lambda_{m};t_{1},..,t_{m-1})
&=& \Bigl(\sum_{j=1}^{m-1}{\lambda_{j}t_{j}^{i}}+\lambda_{m}+ (1-\sum_{j=0}^{m}{\lambda_{j})t_{*}^{i}}\Bigr)_{i=1}^{2m}\, .
\end{eqnarray*}
From
Proposition \ref{prop_Jce} we have
 \begin{eqnarray*}
 |det(d\phi_{1})(\lambda,t)|&=& \mathcal{J}_{1}(t)
\prod_{j=1}^{m}{\lambda_{j}}\\
 |det(d\phi_{2})(\lambda,t)|&=& \mathcal{J}_{2}(t)
\prod_{j=1}^{m-1}{\lambda_{j}}
\end{eqnarray*}
where
\begin{eqnarray*}
\mathcal{J}_{1}(t)
&=& \prod_{j=1}^{m}{(t_{j}-t_{*})^{2}}\cdot\Delta_{m}^{4}(t)\\
\mathcal{J}_{2}(t)
&=&
 t_{*}\bigl(1-t_{*}\bigr) \prod_{j=1}^{m-1}{(t_{j}-t_{*})^{2}}\prod_{j=1}^{m-1}{t_{j}^{2}(1-t_{j})^{2}}
\cdot\Delta_{m-1}^{4}(t)\, .
\end{eqnarray*}
 In this notation, the modified change of variable formula \eqref{id_cov2}
 becomes
\begin{equation}
\label{eq_even_sum1}
\int_{M^{2m}}{q_{i}}
=\int_{\Lambda^{m}\times T^{m}}{\phi_{1}^{i}|d\phi_{1}|}+\int_{\Lambda^{m+1}\times T^{m-1}}{\phi_{2}^{i}|d\phi_{2}|}\, .
\end{equation}
We evaluate the two integrals in  \eqref{eq_even_sum1}  by
\begin{eqnarray*}
\int_{\Lambda^{m}\times T^{m}}{\phi_{1}^{i}|d\phi_{1}|}&=&
\int_{\Lambda^{m}\times T^{m}}{\phi_{1}^{i}\bigl(\prod_{j=1}^{m}{\lambda_{j}}\bigr)\mathcal{J}_{1}}
\notag \\
&=&\int_{T^{m}}{\Bigl(\int_{\Lambda^{m}}\phi_{1}^{i}\prod_{j=1}^{m}{\lambda_{j}d\lambda}\Bigr)\mathcal{J}_{1}
}
\end{eqnarray*}
and
\begin{eqnarray*}
\int_{\Lambda^{m+1}\times T^{m-1}}{\phi_{2}^{i}|d\phi_{1}|}&=&
\int_{\Lambda^{m+1}\times T^{m-1}}{\phi_{2}^{i}\bigl(\prod_{j=1}^{m-1}{\lambda_{j}}\bigr)\mathcal{J}_{2}}
\notag \\
&=&\int_{T^{m-1}}{\Bigl(\int_{\Lambda^{m+1}}\phi_{1}^{i}\prod_{j=1}^{m-1}{\lambda_{j}d\lambda}\Bigr)\mathcal{J}_{2}
}
\end{eqnarray*}

Performing the $\Lambda^{m+1}$ and $\Lambda^{m}$  integrations, using the identities
$\int_{\Lambda^{m}}{\lambda_{1}\prod_{j=1}^{m}{\lambda_{j}}
d\lambda}
= \frac{2}{(2m+1)!}$,
$\int_{\Lambda^{m}}
{\lambda_{m-2}^{2}\prod_{i=1}^{m-3}{\lambda_{i}}d\lambda} = \frac{2}{(2m-1)!}$,
$\int_{\Lambda^{m}}
{\prod_{j=1}^{m-2}{\lambda_{i}}d\lambda} = \frac{1}{(2m-2)!}
$,
$\int_{\Lambda^{m}}{\lambda_{m-1}^{2}\prod_{i=1}^{m-2}{\lambda_{i}}d\lambda} = \frac{2}{(2m)!}$,
$\int_{\Lambda^{m}}{\prod_{i=1}^{m}{\lambda_{i}}d\lambda} = \frac{1}{(2m)!}$, and $\int_{\Lambda^{m}}{\prod_{i=1}^{m-1}{\lambda_{i}}d\lambda} = \frac{1}{(2m-1)!}$,
 we obtain for $i\geq 1$
\begin{eqnarray*}
\int_{\Lambda^{m}}\phi_{1}^{i}\prod_{j=1}^{m}{\lambda_{j}d\lambda}&=&
\int_{\Lambda^{m}}\Bigl(\sum_{j=1}^{m}{\lambda_{j}t_{j}^{i}}+ (1-\sum_{j=1}^{m}{\lambda_{j})t_{*}^{i}}\Bigr)\prod_{j=1}^{m}{\lambda_{j}d\lambda}\notag \\
&=&
\int_{\Lambda^{m}}{\Bigl(\sum_{j=1}^{m}{\lambda_{j}(t_{j}^{i}-t_{*}^{i})}+ t_{*}^{i}\Bigr)
\prod_{j=1}^{m}{\lambda_{j}}d\lambda}\notag \\
&=&
\frac{2}{(2m+1)!}\sum_{j=1}^{m}{(t_{j}^{i}-t_{*}^{i})}+\frac{1}{(2m)!} t_{*}^{i}
\notag \\
&=&
\frac{2}{(2m+1)!}\sum_{j=1}^{m}{t_{j}^{i}}+\frac{1}{(2m+1)!} t_{*}^{i}\, .
\end{eqnarray*}
and
\begin{eqnarray*}
\int_{\Lambda^{m+1}}{\phi_{2}^{i}\prod_{j=1}^{m-1}{\lambda_{j}}d\lambda}&=&
\int_{\Lambda^{m+1}}{\Bigl(\sum_{j=1}^{m-1}{\lambda_{j}t_{j}^{i}}+\lambda_{m}+ (1-\sum_{j=0}^{m}{\lambda_{j})t_{*}^{i}}\Bigr)\prod_{j=1}^{m-1}{\lambda_{j}}d\lambda}\notag \\
&=&
\int_{\Lambda^{m+1}}{\Bigl(\sum_{j=1}^{m-1}{\lambda_{j}(t_{j}^{i}-t_{*}^{i})}+\lambda_{m}(1-t_{*}^{i})+ (1-\lambda_{0})
t_{*}^{i}
\Bigr)\prod_{j=1}^{m-1}{\lambda_{j}}d\lambda}\notag \\
&=&
\frac{2}{(2m+1)!}\sum_{j=1}^{m-1}{(t_{j}^{i}-t_{*}^{i})}+ \frac{1}{(2m+1)!}(1-t_{*}^{i})+ (\frac{1}{(2m)!}-\frac{1}{(2m+1)!})
t_{*}^{i}
\notag \\
&=&
\frac{2}{(2m+1)!}\sum_{j=1}^{m-1}{t_{j}^{i}}+ \frac{1}{(2m+1)!}+ \frac{1}{(2m+1)!}
t_{*}^{i}\, .
\end{eqnarray*}

For $i=0$, \eqref{eq_even_sum1} implies
\begin{equation}
\label{vol_eveniii}
Vol(M^{2m})=\frac{1}{(2m)!}\int_{T^{m}}{
\mathcal{J}_{1}}+
\frac{1}{(2m)!}\int_{T^{m-1}}{
\mathcal{J}_{2}}
\notag
\end{equation}
so that for $i\geq 1$ we have
\begin{eqnarray*}
&&\int_{M^{2m}}{q_{i}}\notag\\
&=&\int_{\Lambda^{m}\times T^{m}}{\phi_{1}^{i}|d\phi_{1}|}+\int_{\Lambda^{m+1}\times T^{m-1}}{\phi_{2}^{i}|d\phi_{2}|}\notag \\
&=&\int_{T^{m}}{\Bigl(\frac{2}{(2m+1)!}\sum_{j=1}^{m}{t_{j}^{i}}+\frac{1}{(2m+1)!} t_{*}^{i}\Bigr)\mathcal{J}_{1}}\notag\\
&+&
\int_{T^{m-1}}{\Bigl(\frac{2}{(2m+1)!}\sum_{j=1}^{m-1}{t_{j}^{i}}+ \frac{1}{(2m+1)!}+ \frac{1}{(2m+1)!}
t_{*}^{i}
\Bigr)\mathcal{J}_{2}}\notag\\
&=&\int_{T^{m}}{\Bigl(\frac{2}{(2m+1)!}\sum_{j=1}^{m}{t_{j}^{i}}\Bigr)\mathcal{J}_{1}}\notag \\
&+&
\int_{T^{m-1}}{\Bigl(\frac{2}{(2m+1)!}\sum_{j=1}^{m-1}{t_{j}^{i}}+ \frac{1}{(2m+1)!}
\Bigr)\mathcal{J}_{2}}\notag\\
&+&\frac{t_{*}^{i}}{2m+1}Vol(M^{2m})
\end{eqnarray*}
so that for $i\geq 1$ we conclude
\begin{eqnarray*}
&&\int_{M^{2m}}{q_{i}}- \frac{t_{*}^{i}}{2m+1}Vol(M^{2m})\notag\\
&=&\frac{1}{(2m+1)!}\int_{T^{m}}{\Bigl(2\sum_{j=1}^{m}{t_{j}^{i}}\Bigr)\mathcal{J}_{1}}
+
\frac{1}{(2m+1)!}\int_{T^{m-1}}{\Bigl(2\sum_{j=1}^{m-1}{t_{j}^{i}}+1
\Bigr)\mathcal{J}_{2}}\notag\, .
\end{eqnarray*}
Combining with the result \eqref{vol_eveniii} for $i=0$
we conclude
\begin{eqnarray*}
&&\int_{M^{2m}}{q_{i}}- \frac{t_{*}^{i}}{2m+1}Vol(M^{2m})\notag\\
&=&\frac{2}{(2m+1)!}\int_{T^{m}}{\sum_{j=1}^{m}{t_{j}^{i}}\mathcal{J}_{1}}
+
\frac{2}{(2m+1)!}\int_{T^{m-1}}{\sum_{j=1}^{m-1}{t_{j}^{i}}
\mathcal{J}_{2}}+\frac{1+\d_{0}(i)}{(2m+1)!}\int_{T^{m-1}}{
\mathcal{J}_{2}}\notag
\end{eqnarray*}
establishing the assertion in the even case.

\subsection{Proof of Theorem \ref{thm_RK}}
Recall the identity  $\mathcal{J}^{c}_{ol}(0,t)\equiv 0$. Then
subtracting the volume identity \eqref{vol_Jco}
\begin{equation}
\label{vol_Jco2}
Vol(M^{2m-1}) =\frac{1}{(2m-1)!(m-1)!} \int_{I^{m-1}}{\bigl(\mathcal{J}^{c}_{ol}(t_{*},t)+
\mathcal{J}^{c}_{ou}(t_{*},t)\bigr)dt}\,
\end{equation}
from itself evaluated at $t_{*}=0$,
we conclude that
\begin{equation}
\label{vol_subtract_odd}
\int_{I^{m-1}}{\bigl(\mathcal{J}^{c}_{ol}(t_{*},t)+
\mathcal{J}^{c}_{ou}(t_{*},t)-\mathcal{J}^{c}_{ou}(0,t)\bigr)dt}\, \equiv 0\, ,
\end{equation}
that is,
\begin{equation*}
\label{vol_Jco3}
 \int_{I^{m-1}}{\mathcal{H}(t_{*},t)
dt}\, \equiv 0\, .
\end{equation*}
We now do the same subtraction for all the moments. To that end, recall the convention
$0^{0}=1$, and
 observe that, for $i\geq 0$,   the identity
\begin{eqnarray}
\label{eq_odd_cov5b}
&&\int_{M^{2m-1}}{q_{i}}-\frac{t_{*}^{i}}{2m} Vol(M^{2m-1})\notag\\
&=&\frac{\d_{0}(i)}{(2m)!(m-1)!}\int_{I^{m-1}}
{\mathcal{J}^{c}_{ol}(t_{*},t)dt}+\frac{1}{(2m)!(m-1)!}\int_{I^{m-1}}
{\mathcal{J}^{c}_{ou}(t_{*},t)dt}\notag\\
 &+&\frac{2}{(2m)!(m-1)!}\int_{I^{m-1}}{\Sigma t^{i}
\Bigl(\mathcal{J}^{c}_{ol}(t_{*},t)+\mathcal{J}^{c}_{ou}(t_{*},t)\Bigr)dt}\,
\end{eqnarray}
from Proposition \ref{prop_mom},
evaluated at $t_{*}=0$ becomes
\begin{eqnarray*}
&&\int_{M^{2m-1}}{q_{i}}-\frac{\d_{0}(i)}{2m} Vol(M^{2m-1})\notag\\
&=&\frac{\d_{0}(i)}{(2m)!(m-1)!}\int_{I^{m-1}}
{\mathcal{J}^{c}_{ol}(0,t)dt}+\frac{1}{(2m)!(m-1)!}\int_{I^{m-1}}
{\mathcal{J}^{c}_{ou}(0,t)dt}\notag\\
 &+&\frac{2}{(2m)!(m-1)!}\int_{I^{m-1}}{\Sigma t^{i}
\Bigl(\mathcal{J}^{c}_{ol}(0,t)+\mathcal{J}^{c}_{ou}(0,t)\Bigr)dt}\notag \, .
\end{eqnarray*}
Subtracting from \eqref{eq_odd_cov5b},
using the identity $\mathcal{J}^{c}_{ol}(0,t)\equiv 0$, we obtain
\begin{eqnarray*}
&&-\frac{t_{*}^{i}}{2m} Vol(M^{2m-1})
+ \frac{\d_{0}(i)}{2m} Vol(M^{2m-1})\notag\\
&=&\frac{\d_{0}(i)}{(2m)!(m-1)!}\int_{I^{m-1}}
{\mathcal{J}^{c}_{ol}(t_{*},t)dt}+\frac{1}{(2m)!(m-1)!}\int_{I^{m-1}}
{\bigl(\mathcal{J}^{c}_{ou}(t_{*},t)-\mathcal{J}^{c}_{ou}(0,t)\bigr)dt}\notag\\
 &-&\frac{2}{(2m)!(m-1)!}\int_{I^{m-1}}{\Sigma t^{i}
\mathcal{H}(t_{*},t)dt}\notag
\end{eqnarray*}
and applying the volume identity \eqref{vol_Jco2} we obtain
\begin{eqnarray*}
&&-\frac{t_{*}^{i}}{2m} Vol(M^{2m-1})\notag\\
&=&-\frac{\d_{0}(i)}{(2m)!(m-1)!}\int_{I^{m-1}}
{\mathcal{J}^{c}_{ou}(t_{*},t)dt}+\frac{1}{(2m)!(m-1)!}\int_{I^{m-1}}
{\bigl(\mathcal{J}^{c}_{ou}(t_{*},t)-\mathcal{J}^{c}_{ou}(0,t)\bigr)dt}\notag\\
 &-&\frac{2}{(2m)!(m-1)!}\int_{I^{m-1}}{\Sigma t^{i}
\mathcal{H}(t_{*},t)dt}\notag
\end{eqnarray*}
and the subtracted volume identity \eqref{vol_subtract_odd}  we obtain with a change of sign
\begin{eqnarray}
\label{eq_odd_cov5c2}
&&\frac{t_{*}^{i}}{2m} Vol(M^{2m-1})\\
&=&\frac{\d_{0}(i)}{(2m)!(m-1)!}\int_{I^{m-1}}
{\mathcal{J}^{c}_{ou}(t_{*},t)dt}+\frac{1}{(2m)!(m-1)!}\int_{I^{m-1}}
{\mathcal{J}^{c}_{ol}(t_{*},t)dt}\notag\\
 &+&\frac{2}{(2m)!(m-1)!}\int_{I^{m-1}}{\Sigma t^{i}
\mathcal{H}(t_{*},t)dt}\notag\, .
\end{eqnarray}

Then, if we let  $\phi(s):=\sum_{i=0}^{2m-1}{\phi_{i}s^{i}}$ be a polynomial of degree
$n=2m-1$, summing over each identity in \eqref{eq_odd_cov5c2}, we conclude  that
\begin{eqnarray}
\label{eq_odd_cov5d}
&&\frac{\phi(t_{*})}{2m} Vol(M^{2m-1})\notag\\
&=&\phi_{0}\frac{1}{(2m)!(m-1)!}\int_{I^{m-1}}
{\bigl(\mathcal{J}^{c}_{ol}(t_{*},t)+\mathcal{J}^{c}_{ou}(t_{*},t)\bigr)dt}
+\sum_{i=1}^{2m-1}{\phi_{i}}\frac{1}{(2m)!(m-1)!}\int_{I^{m-1}}
{\mathcal{J}^{c}_{ol}(t_{*},t)dt}\notag\\
 &+&\frac{2}{(2m)!(m-1)!}\int_{I^{m-1}}{(\Sigma \phi)(t)
\mathcal{H}(t_{*},t))dt}\notag\\
&=&\phi_{0}\frac{1}{(2m)!(m-1)!}\int_{I^{m-1}}
{\mathcal{J}^{c}_{ou}(t_{*},t)dt}
+\sum_{i=0}^{2m-1}{\phi_{i}}\frac{1}{(2m)!(m-1)!}\int_{I^{m-1}}
{\mathcal{J}^{c}_{ol}(t_{*},t)dt}\notag\\
 &+&\frac{2}{(2m)!(m-1)!}\int_{I^{m-1}}{(\Sigma \phi)(t)
\mathcal{H}(t_{*},t)dt}\notag\, .
\end{eqnarray}
Since $\phi(0)=\phi_{0}$ and $\phi(1)=\sum_{i=0}^{2m-1}{\phi_{i}}$ the assertion follows by multiplication by
$2m$.
The even case proceeds in the same way, but since it is a little different we have included it in Section \ref{sec_RKproof} in the Appendix.

\subsection{Proof of Theorem \ref{thm_selberg}}
\label{sec_selberg_proof}
For the first assertion, let $n=2m-1$ and consider the integral formula \eqref{vol_Jco2}
\[
Vol(M^{2m-1}) =\frac{1}{(2m-1)!(m-1)!} \int_{I^{m-1}}{\bigl(\mathcal{J}^{c}_{ol}(t_{*},t)+
\mathcal{J}^{c}_{ou}(t_{*},t)\bigr)dt}\,
\]
 for the volume in terms of the canonical representations.
From the definitions
\begin{eqnarray*}
\mathcal{J}^{c}_{ol}(t_{*},t)&=&
 t_{*} \prod_{j=1}^{m-1}{(t_{j}-t_{*})^{2}}\prod_{j=1}^{m-1}{t_{j}^{2}}
\cdot \Delta_{m-1}^{4}(t)
\\
\mathcal{J}^{c}_{ou}(t_{*},t)&=&
 \bigl(1-t_{*}\bigr) \prod_{j=1}^{m-1}{(t_{j}-t_{*})^{2}}
\prod_{j=1}^{m-1}{(1-t_{j})^{2}}\cdot
\Delta_{m-1}^{4}(t)\,
\end{eqnarray*}
of Proposition \ref{prop_Jco},
we obtain
\begin{eqnarray*}
\frac{\partial}{\partial t_{*}}\mathcal{J}^{c}_{ol}(t_{*},t)|_{t_{*}=0}&=&
\prod_{j=1}^{m-1}{t_{j}^{4}}\cdot \Delta_{m-1}^{4}(t)\\
\frac{\partial}{\partial t_{*}}\mathcal{J}^{c}_{ou}(t_{*},t)|_{t_{*}=0}&=&
-\bigl(1+ 2\Sigma t^{-1}\bigr)\prod_{j=1}^{m-1}{t_{j}^{2}(1-t_{j})^{2}}
\Delta_{m-1}^{4}(t)\, .
\end{eqnarray*}
Differentiating the volume formula with respect to $t_{*}$ at $t_{*}=0$, we  obtain
\begin{eqnarray*}
0&=& \frac{1}{(2m-1)!(m-1)!} \int_{I^{m-1}}{\Bigl(\frac{\partial}{\partial t_{*}}\mathcal{J}^{c}_{ol}(t_{*},t)|_{t_{*}=0}+
\frac{\partial}{\partial t_{*}}\mathcal{J}^{c}_{ou}(t_{*},t)|_{t_{*}=0}\Bigr)dt}\notag\\
\end{eqnarray*}
and therefore
\begin{eqnarray*}
 \int_{I^{m-1}}{\prod_{j=1}^{m-1}{t_{j}^{4}}\cdot \Delta_{m-1}^{4}(t)dt}&=&
 \int_{I^{m-1}}{\bigl(1+ 2\Sigma t^{-1}\bigr)\prod_{j=1}^{m-1}{t_{j}^{2}(1-t_{j})^{2}}
\Delta_{m-1}^{4}(t)
dt}\notag\\
&=&
 \int_{I^{m-1}}{\prod_{j=1}^{m-1}{t_{j}^{2}(1-t_{j})^{2}}
\Delta_{m-1}^{4}(t)
dt}\\
&+&2\int_{I^{m-1}}{\Sigma t^{-1}\cdot \prod_{j=1}^{m-1}{t_{j}^{2}(1-t_{j})^{2}}
\Delta_{m-1}^{4}(t)
dt}
\end{eqnarray*}
from which we conclude that
\[2\int_{I^{m-1}}{\Sigma t^{-1}\cdot \prod_{j=1}^{m-1}{t_{j}^{2}(1-t_{j})^{2}}
\Delta_{m-1}^{4}(t)
dt}= S_{m-1}(5,1,2)-S_{m-1}(3,3,2)\, .\]
Changing $m \mapsto m+1$ finishes the proof of the first assertion.

For the second assertion, let $n=2m$ and consider the integral formula \eqref{vol_Jce}
\[
Vol(M^{2m})=\frac{1}{(2m)!m!}\int_{I^{m}}{\mathcal{J}^{c}_{el}(t_{*},t)dt}+
\frac{1}{(2m)!(m-1)!}\int_{I^{m-1}}{\mathcal{J}^{c}_{eu}(t_{*},t)dt}\, .
\]
From the definitions
\begin{eqnarray*}
\mathcal{J}^{c}_{el}(t_{*},t)
&=& \prod_{j=1}^{m}{(t_{j}-t_{*})^{2}}\cdot \Delta_{m}^{4}(t)\\
\mathcal{J}^{c}_{eu}(t_{*},t)
&=&
 t_{*}\bigl(1-t_{*}\bigr) \prod_{j=1}^{m-1}{(t_{j}-t_{*})^{2}}\prod_{j=1}^{m-1}{t_{j}^{2}(1-t_{j})^{2}}
\cdot \Delta_{m-1}^{4}(t)\,
\end{eqnarray*}
of  Proposition \ref{prop_Jce}, we obtain
\begin{eqnarray*}
\frac{\partial}{\partial t_{*}}\mathcal{J}^{c}_{el}(t_{*},t)|_{t_{*}=0}&=&
-2\Sigma t^{-1}\prod_{j=1}^{m}{t_{j}^{2}}\cdot \Delta_{m}^{4}(t)\\
\frac{\partial}{\partial t_{*}}\mathcal{J}^{c}_{eu}(t_{*},t)|_{t_{*}=0}&=&
\prod_{j=1}^{m-1}{t_{j}^{4}(1-t_{j})^{2}}
\cdot \Delta_{m-1}^{4}(t)\, .
\end{eqnarray*}
Differentiating the volume formula with respect to $t_{*}$ at $t_{*}=0$, we obtain
\begin{eqnarray*}
0&=&\frac{1}{(2m)!m!}\int_{I^{m}}{\frac{\partial}{\partial t_{*}}\mathcal{J}^{c}_{el}(t_{*},t)|_{t_{*}=0}dt}+
\frac{1}{(2m)!(m-1)!}\int_{I^{m-1}}{\frac{\partial}{\partial t_{*}}\mathcal{J}^{c}_{eu}(t_{*},t)|_{t_{*}=0}dt}
\end{eqnarray*}
and therefore we conclude
\begin{eqnarray*}
2 \int_{I^{m}}{\Sigma t^{-1} \cdot \prod_{j=1}^{m}{t_{j}^{2}}\cdot \Delta_{m}^{4}(t)dt}&=&
 m\int_{I^{m-1}}{\prod_{j=1}^{m-1}{t_{j}^{4}(1-t_{j})^{2}}
\cdot \Delta_{m-1}^{4}(t)
dt}\\
&=& mS_{m-1}(5,3,2)\, ,
\end{eqnarray*}
finishing the proof of the second assertion.

\subsection{Proof of Theorem \ref{thm_selberg2}}
First note that the definition \eqref{def_H}
\[\mathcal{H}(t_{*},t):=\mathcal{J}^{c}_{ou}(0,t)-\mathcal{J}^{c}_{ol}(t_{*},t)-\mathcal{J}^{c}_{ou}(t_{*},t)\]
and the definitions of
$\mathcal{J}^{c}_{ol}$ and $\mathcal{J}^{c}_{ou}$
from Proposition \ref{prop_Jco}  imply that
$\hat{\mathcal{H}}(\cdot ,t) \in \Pi^{2m-1},\, t \in I^{m-1}.$
 Therefore,  it follows from  \eqref{H_dirichlet} that
$\hat{\mathcal{H}}(\cdot ,t) \in \Pi^{2m-1}_{0}, \, t \in I^{m-1}.$
Now, it follows from Theorem \ref{thm_RK}  that
\begin{eqnarray*}
\phi(t_{*})
&=&\int_{I^{m-1}}{(\Sigma \phi)(t)\hat{\mathcal{H}}(t_{*},t)dt}, \quad \phi \in \Pi^{2m-1}_{0}
\end{eqnarray*}
which expanded becomes
\begin{eqnarray*}
\phi(t_{*})
&=&\sum_{j=1}^{2m-2}{p_{j}(t_{*})\int_{I^{m-1}}{(\Sigma \phi)(t)h_{j}(t)dt}}, \quad \phi \in \Pi^{2m-1}_{0}\, ,
\end{eqnarray*}
in particular, by choosing $\phi:=p_{k}, k=1,..,2m-2$,
\begin{eqnarray*}
p_{k}(t_{*})&=&\sum_{j=1}^{2m-2}{p_{j}(t_{*})\int_{I^{m-1}}{(\Sigma p_{k})(t)h_{j}(t)dt}}, \quad  k=1,..,2m-2\, ,
\end{eqnarray*}
from which we conclude
\begin{equation*}
\label{selbergfamily}
\int_{I^{m-1}}{\Sigma p_{j}\cdot h_{k}} =\delta_{jk}, \quad  j,k=1,..,2m-2\, \,
\end{equation*}
establishing the assertion. Furthermore, from this and the symmetry of $h_{k},,k=1,..,2m-2$ with respect to the
action of the symmetric group, we also conclude
\begin{equation*}
\label{selbergfamilymarg}
(m-1)\int_{I}{ p_{j}\cdot \bar{h}_{k}} =\delta_{jk}, \quad  j,k=1,..,2m-2\, \,
\end{equation*}
establishing Corollary \ref{cor_selberg2}.

\subsection{Proof of Theorem \ref{thm_selberg_explicit}}
From the orthogonality relation (see e.g.~\cite[Eq.~12.110]{Arfken})
\[\int_{I}{Q_{j}Q_{k}}=\frac{1}{2k+1}\frac{(k+2)!}{(k-2)!} \d_{jk}\]
and the definition
\[\hat{\mathcal{H}}:=\frac{1}{Vol(M^{2m-1})}\frac{2}{(2m-1)!(m-1)!}\mathcal{H}\, \]
of the scaling of the kernel
 \eqref{def_H}
\begin{equation}
\label{def_H2}
\mathcal{H}(t_{*},t)=\mathcal{J}^{p}_{ou}(t)-\mathcal{J}^{c}_{ol}(t_{*},t)-\mathcal{J}^{c}_{ou}(t_{*},t)\, ,
\end{equation}
we can compute
the coefficients $h_{k}$ in the expansion
\[\hat{\mathcal{H}}(t_{*} ,t)=\sum_{j=2}^{2m-1}{h_{j}(t)Q_{j}(t_{*})},\quad  (t_{*},t) \in I \times I^{m-1}\, \]
as
\begin{equation}
\label{def_h}
h_{k}(t)= \frac{1}{Vol(M^{2m-1})}\frac{2}{(2m-1)!(m-1)!}\frac{(2k+1)(k-2)!}{(k+2)!}\int_{I}{\mathcal{H}(\cdot,t)Q_{k}(\cdot)},\quad  k=2,..,2m-1\,
\end{equation}
and then apply
Theorem \ref{thm_selberg2} to obtain the assertion.
To that end, for $k=2,..,2m-1$,
to compute
\[\int_{I}{\mathcal{H}(\cdot,t)Q_{k}(\cdot)}\]
we
use the decomposition of \eqref{def_H2} of $\mathcal{H}$ and compute the values
$\int_{I}{\mathcal{J}^{p}_{ou}(t)Q_{k}(\cdot)},
\int_{I}{\mathcal{J}^{c}_{ol}(\cdot,t)Q_{k}(\cdot)}, $ and
$\int_{I}{\mathcal{J}^{c}_{ou}(\cdot,t)Q_{k}(\cdot)} $ separately.
For the first term,  observe that
$
\int_{I}{Q_{k}}=1+(-1)^{k},\, k=2,...,2m-1$ from Proposition \ref{legendre_integral},
so that
\begin{equation}
\label{e111}
\int_{I}{\mathcal{J}^{p}_{ou}(t)Q_{k}(r)dr}=\mathcal{J}^{p}_{ou}(t)\int_{I}{Q_{k}(r)dr}=(1+(-1)^{k})\mathcal{J}^{p}_{ou}(t)\, .
\end{equation}
For the second,  $\int_{I}{\mathcal{J}^{c}_{ol}(\cdot,t)Q_{k}(\cdot)}$,
we expand $\mathcal{J}^{c}_{ol}(\cdot,t)$, defined in \eqref{prop_Jco}, as a polynomial for fixed $t$
and then utilize the values of the integrals
$\int_{I}{r^{j}Q_{k}(r)dr}$
for the monomials $r^{j}, j=1,..,2m-1$.
To that end, define
\begin{equation}
\label{def_tildeJ}
 \tilde{\mathcal{J}}^{c}_{ol}(\cdot,t):=t{*}\prod_{j=1}^{m-1}{(t_{j}-t_{*})^{2}}
\end{equation}
so that
\[\mathcal{J}^{c}_{ol}(\cdot,t)=\tilde{\mathcal{J}}^{c}_{ol}(\cdot,t)\prod_{j=1}^{m-1}{t_{j}^{2}}\cdot
        \Delta_{m-1}^{4}(t)\, .\]
Then from the definition of $\mathcal{J}^{p}_{el}$ of  \eqref{prop_Jpe}
we have $\mathcal{J}^{p}_{el}(t)=\prod_{j=1}^{m-1}{t_{j}^{2}}\cdot
        \Delta_{m-1}^{4}(t)$ so that
\begin{equation}
\label{Jdecomp}
\mathcal{J}^{c}_{ol}(\cdot,t)=\tilde{\mathcal{J}}^{c}_{ol}(\cdot,t)
        \mathcal{J}^{p}_{el}(t)\, .
\end{equation}

        The generating function identity for the elementary symmetric functions $e_{j}$ is
        \[ \prod_{i=1}^{m-1}(1+s t_{i})=\sum_{j=0}^{m-1}{e_{j}(t)s^{j}}\]
        and squaring it we obtain
        \begin{eqnarray}
        \prod_{i=1}^{m-1}{(1+s t_{i})^{2}}& =&\Bigl(\sum_{j=0}^{m-1}{e_{j}(t)s^{j}}\Bigr)^{2}\notag \\
                &=& \sum_{j=0}^{2m-2}{e_{j}(t,t)s^{j}}\, .
                \end{eqnarray}
                Therefore, by changing $s \mapsto -t_{*}^{-1}$ we conclude that
                \begin{eqnarray}
                \prod_{j=1}^{m-1}{(t_{j}-t_{*})^{2}}& =&
                \sum_{j=0}^{2m-2}{e_{2m-2-j}(t,t)(-1)^{j}t_{*}^{j}}
                \end{eqnarray}
                and therefore
                \[t{*}\prod_{j=1}^{m-1}{(t_{j}-t_{*})^{2}}=
                                \sum_{j=0}^{2m-2}{e_{2m-2-j}(t,t)(-1)^{j}t_{*}^{j+1}}
        \]
and, relabeling, we conclude
\begin{equation}
\label{J_exp}
\tilde{\mathcal{J}}^{c}_{ol}(t_{*},t)= - \sum_{j=0}^{2m-1}{d_{j}(t)(-1)^{j}t_{*}^{j}}
 \end{equation}
where
\begin{equation}
\label{def_d}
d_{j}(t):=e_{2m-1-j}(t,t)=\sum_{j_{1}+j_{2}=2m-1-j}{e_{j_{1}}(t)e_{j_{2}}(t)},\quad  j=1,..,2m-1\,
\end{equation}
and $d_{0}(t):=0$. Note that $d_{j}$ is a symmetric polynomial of degree $2m-1-j$.

In particular,
 from the definition \eqref{def_tildeJ} and its resulting polynomial expansion  \eqref{J_exp}  we have
\begin{equation}
\label{J1}
\tilde{\mathcal{J}}^{c}_{ol}(1,t)=-\sum_{j=0}^{2m-1}{d_{j}(t)(-1)^{j}}= \prod_{j=1}^{m-1}{(1-t_{j})^{2}}\, .
\end{equation}
The following proposition  computes the values of the integrals of the Legendre polynomials
 against the monomials,
and we observe  that $\int_{I}{r^{j}Q_{k}(r)dr}=1$ for $1\leq  j <k$
and $\int_{I}{r^{j}Q_{k}(r)dr}=1$ plus a term when $j \geq k$.
\begin{prop}
\label{legendre_integral}
For  $k=2,..,2m-1$ we have
\[\int_{I}{Q_{k}(r)dr}=1+(-1)^{k}\]
\[\int_{I}{r^{j}Q_{k}(r)dr}=1,\quad 1\leq  j <k\]
\[\int_{I}{r^{j}Q_{k}(r)dr}=1-\frac{(j+k+k^{2})\Gamma(j+2)\Gamma(j)}{\Gamma(j+k+2)\Gamma(j-k+1)},
\quad   j\geq k\, .\]
\end{prop}
 Therefore,
 since $d_{0}=0$, we obtain from the polynomial expansion  \eqref{J_exp},
 Proposition \ref{legendre_integral}, and \eqref{J1}  that
\begin{eqnarray}
\label{JQol}
 \int_{I}{ \tilde{\mathcal{J}}^{c}_{ol}(r,t)Q_{k}(r)dr}&=&
\sum_{j=k}^{2m-1}{(-1)^{j}a_{jk}d_{j}(t)}
-\sum_{j=0}^{2m-1}{(-1)^{j}d_{j}(t)}\notag\\
&=&
\sum_{j=k}^{2m-1}{(-1)^{j}a_{jk}d_{j}(t)}
+\tilde{\mathcal{J}}^{c}_{ol}(1,t)
\end{eqnarray}
where we recall the definition
\[a_{jk}:=\frac{(j+k+k^{2})\Gamma(j+2)\Gamma(j)}{\Gamma(j+k+2)\Gamma(j-k+1)}\, ,\]
and note
 that the lower limit in the summation is $k$ and consequently, the first term is a  symmetric
polynomial of degree $2m-1 -k$.
Therefore, multiplying by  $\mathcal{J}^{p}_{el}(t)=\prod_{j=1}^{m-1}{t_{j}^{2}}\cdot
        \Delta_{m-1}^{4}(t)$ and using the definition
\eqref{Jdecomp} and the identity
\begin{eqnarray}
 \tilde{\mathcal{J}}^{c}_{ol}(1,t) \mathcal{J}^{p}_{el}(t)&=&
\prod_{j=1}^{m-1}{(1-t_{j})^{2}}\prod_{j=1}^{m-1}{t_{j}^{2}}\cdot
        \Delta_{m-1}^{4}(t)\notag\\
&=&\mathcal{J}^{p}_{ou}(t)
\end{eqnarray}
from the definition
$\mathcal{J}^{p}_{ou}(t)=\prod_{j=1}^{m-1}{t_{j}^{2}(1-t_{j})^{2}}\cdot
        \Delta_{m-1}^{4}(t)$ from Proposition \eqref{prop_Jpo},
we conclude that
\begin{eqnarray}
\label{JQolfinal}
 \int_{I}{ \mathcal{J}^{c}_{ol}(r,t)Q_{k}(r)dr}
&=&
\mathcal{J}^{p}_{el}(t)\sum_{j=k}^{2m-1}{(-1)^{j}a_{jk}d_{j}(t)}
+\mathcal{J}^{p}_{ou}(t)\notag\, .
\end{eqnarray}
Let us designate the negative of the first  term
\begin{equation}
\label{def_ah}
\acute{h}_{k}(t):=\mathcal{J}^{p}_{el}(t)\sum_{j=k}^{2m-1}{(-1)^{j+1}a_{jk}d_{j}(t)}
\end{equation}
so that
\begin{equation}
\label{e222}
\int_{I}{ \mathcal{J}^{c}_{ol}(r,t)Q_{k}(r)dr}=-\acute{h}_{k}(t)+\mathcal{J}^{p}_{ou}(t) .
\end{equation}

For the third term $\int_{I}{\mathcal{J}^{c}_{ou}(\cdot,t)Q_{j}(\cdot)} $, we utilize the reflection symmetry
\[\mathcal{J}^{c}_{ou}(t_{*},t)=\mathcal{J}^{c}_{ol}(1-t_{*},1-t)\] of
\eqref{Jcsymm} and the reflection symmetry
(see e.g.~\cite[Eq.~12.97]{Arfken})
\[Q_{k}(1-r)=(-1)^{k}Q_{k}(r)\]
of the associate Legendre polynomials to compute the integral $\int_{I}{\mathcal{J}^{c}_{ou}(\cdot,t)Q_{j}(\cdot)} $ in terms of $\int_{I}{\mathcal{J}^{c}_{ol}(\cdot,t)Q_{j}(\cdot)} $.
That is,
\begin{eqnarray*}
 \int_{I}{ \mathcal{J}^{c}_{ou}(r,t)Q_{k}(r)dr}&=&
 \int_{I}{ \mathcal{J}^{c}_{ou}(1-r,t)Q_{k}(1-r)dr}\\
&=& \int_{I}{ \mathcal{J}^{c}_{ol}(r,1-t)Q_{k}(1-r)dr}\\
&=& (-1)^{k}\int_{I}{ \mathcal{J}^{c}_{ol}(r,1-t)Q_{k}(r)dr}\\
\end{eqnarray*}
and therefore
\begin{eqnarray*}
 \int_{I}{ \mathcal{J}^{c}_{ou}(r,t)Q_{k}(r)dr}
&=& (-1)^{k}\int_{I}{ \mathcal{J}^{c}_{ol}(r,1-t)Q_{k}(r)dr}\\
&=& -(-1)^{k}\acute{h}_{k}(1-t) +(-1)^{k}\mathcal{J}^{p}_{ou}(1-t)\, ,
\end{eqnarray*}
so that we conclude
\begin{equation}
\label{e333}
 \int_{I}{ \mathcal{J}^{c}_{ou}(r,t)Q_{k}(r)dr}
= -(-1)^{k}\acute{h}_{k}(1-t) +(-1)^{k}\mathcal{J}^{p}_{ou}(1-t)\, .
\end{equation}
Putting all three terms together using \eqref{def_H2}  and the identities
\eqref{e111}, \eqref{e222}, and \eqref{e333}, along with the symmetry
\[\mathcal{J}^{p}_{ou}(1-t)=\mathcal{J}^{p}_{ou}(t)\, ,\]
 we conclude that
\begin{eqnarray*}
\label{eqieiei}
 \int_{I}{ \mathcal{H}(r,t)Q_{k}(r)dr}
&=&  \int_{I}{\mathcal{J}^{p}_{ou}(t)Q_{k}(r)dr}- \int_{I}{\mathcal{J}^{c}_{ol}(r,t)Q_{k}(r)dr}-
 \int_{I}{\mathcal{J}^{c}_{ou}(r,t)Q_{k}(r)dr}\notag\\
&=&(1+(-1)^{k})\mathcal{J}^{p}_{ou}(t)+ \acute{h}_{k}(t)-\mathcal{J}^{p}_{ou}(t)+
 (-1)^{k}\acute{h}_{k}(1-t) -(-1)^{k}\mathcal{J}^{p}_{ou}(1-t)\\
&=&\acute{h}_{k}(t)+ (-1)^{k}\acute{h}_{k}(1-t)\, .
\end{eqnarray*}

To finish, consider the functions
$\breve{h}_{k}(t):=\acute{h}_{k}(t)+ (-1)^{k}\acute{h}_{k}(1-t)$.
It follows from \eqref{def_h} that the basis coefficients $h_{k}$  satisfy
\[h_{k}=\frac{1}{Vol(M^{2m-1})}\frac{2}{(2m-1)!(m-1)!}\frac{(2k+1)(k-2)!}{(k+2)!}\breve{h}_{k},\quad
k=2,..,2m-1\, . \]
Moreover,
Theorem \ref{thm_selberg2} implies that $\{\Sigma Q_{j},j=2,..,2m-1\}$ and $\{h_{k}, k=2,..,2m-1\}$
are  an $L^{2}(I^{m-1})$ biorthogonal system.  It therefore follows that
\[\int_{I^{m-1}}{\breve{h}_{k}\Sigma Q_{j}}=Vol(M^{2m-1})(2m-1)!(m-1)!\frac{(k+2)!}{(4k+2)(k-2)!}\d_{jk},\quad j,k=2,..,2m-1\, .\]
 Moreover, from the symmetry
$Q_{k}(1-r)=(-1)^{k}Q_{k}(r)$ and a change of variables we obtain
\begin{eqnarray*}
 \int_{I^{m-1}}{\breve{h}_{k}(t) \Sigma Q_{j}(t)dt}&=&
 \int_{I^{m-1}}{\acute{h}_{k}(t) \Sigma Q_{j}(t)dt}+(-1)^{k} \int_{I^{m-1}}{\acute{h}_{k}(1-t) \Sigma Q_{j}(t)dt}
\\
&=&
 \int_{I^{m-1}}{\acute{h}_{k}(t) \Sigma Q_{j}(t)dt}+(-1)^{k} \int_{I^{m-1}}{\acute{h}_{k}(t) \Sigma Q_{j}(1-t)dt}
\\
&=& \int_{I^{m-1}}{\acute{h}_{k}(t) \Sigma Q_{j}(t)dt}+(-1)^{j+k} \int_{I^{m-1}}{\acute{h}_{k}(t) \Sigma Q_{j}(t)dt}\\
&=&(1+(-1)^{j+k}) \int_{I^{m-1}}{\acute{h}_{k}(t) \Sigma Q_{j}(t)dt}\,.
\end{eqnarray*}
Since $(j+k)\bmod 2 =(j-k)\bmod 2$,
the assertion then follows from the definition \eqref{def_ah} of $\acute{h}_{k}$,  the identity
$\mathcal{J}^{p}_{el}(t)=\prod_{j=1}^{m-1}{t_{j}^{2}}\cdot
        \Delta_{m-1}^{4}(t)$, and the definition \eqref{def_d} of $d_{k}$.
Moreover, we see that for $(j-k)\bmod 2=1$ the vanishing of this integral does not depend on the function
$\acute{h}_{k}$ but is instead a consequence only of the relative parity between $Q_{j}$ and
$\acute{h}_{k}$ with respect to the operation of reflection.

\subsection{Proof of Proposition \ref{legendre_integral}}
We abuse notation by letting $P_{k}$ and $Q_{k}$ denote the Legendre polynomials on the standard
set $[-1,1]$. At the end we will change  back to the interval $I$.
We use \cite[Eq.~7.127, pg.~771]{GradshteynRyzhik}
\[\int_{-1}^{1}{(1+x)^{j}P_{k}(x)dx}=\frac{2^{j+1}\Gamma^{2}(j+1)}{\Gamma(j+k+2)\Gamma(j-k+1)}, \quad j\geq 0\]
and the definition
$Q_{k}(x):=(1-x^{2})P''_{k}(x)$ and integration by parts. Because of the poles of the Gamma function at $0$ and the negative integers, we conclude that
\[\int_{-1}^{1}{(1+x)^{j}P_{k}(x)dx}=0,\quad j<k\, .\]

Consider the function $\phi(x):= (1+x)^{j}(1-x^{2})=2(1+x)^{j+1}-(1+x)^{j+2}$,
which has the derivatives $\phi'(x)=2(j+1)(1+x)^{j}-(j+2)(1+x)^{j+1}$ and
$\phi''(x)=2(j+1)j(1+x)^{j-1}-(j+2)(j+1)(1+x)^{j}$. Since, for $j\geq 1$, we have
$\phi'(-1)=0$ and $\phi'(1)=-2^{j+1}$, and $P_{k}(1)=1$, we obtain
\begin{eqnarray*}
\int_{-1}^{1}{(1+x)^{j}Q_{k}(x)dx}&=&\int_{-1}^{1}{(1+x)^{j}(1-x^{2})P''_{k}(x)dx}\\
&=&\int_{-1}^{1}{\phi(x)P''_{k}(x)dx}\\
&=&\phi P'_{k}\Big|^{1}_{-1}-\int_{-1}^{1}{\phi'(x)P'_{k}(x)dx}\\
&=&-\int_{-1}^{1}{\phi'(x)P'_{k}(x)dx}\\
&=& -\phi'P_{k}\Big|^{1}_{-1} +\int_{-1}^{1}{\phi''(x)P_{k}(x)dx}\\
&=& 2^{j+1} +\int_{-1}^{1}{\phi''(x)P_{k}(x)dx}\\
&=& 2^{j+1} +2(j+1)j\int_{-1}^{1}{(1+x)^{j-1}P_{k}(x)dx}\\&&-(j+2)(j+1)\int_{-1}^{1}{(1+x)^{j}P_{k}(x)dx}\\
\end{eqnarray*}
from which we conclude that
\[\int_{-1}^{1}{(1+x)^{j}Q_{k}(x)dx}= 2^{j+1},\quad 1 \leq j < k\, .\]
For the case $j=0$, defining $\phi(x):=1-x^{2}$, we instead have  $\phi'(-1)=2$, $\phi'(1)=-2$, $P_{k}(1)=1$, and
$P_{k}(-1)=(-1)^{k}$. Using $\int_{-1}^{1}{P_{k}}=0, k\geq 1$,
  we obtain
\begin{eqnarray*}
\int_{-1}^{1}{Q_{k}(x)dx}&=&\int_{-1}^{1}{(1-x^{2})P''_{k}(x)dx}\\
&=&\int_{-1}^{1}{\phi(x)P''_{k}(x)dx}\\
&=&-\int_{-1}^{1}{\phi'(x)P'_{k}(x)dx}\\
&=& -\phi'P_{k}\Big|^{1}_{-1} +\int_{-1}^{1}{\phi''(x)P_{k}(x)dx}\\
&=& -\phi'P_{k}\Big|^{1}_{-1}\\
&=& -\Bigl(-2 P_{k}(1) -2P_{k}(-1)\Bigr)\\
&=& 2(1+(-1)^{k}) \, .
\end{eqnarray*}
On the other hand,  for $j \geq  k$ we have
\begin{eqnarray*}
&&\int_{-1}^{1}{(1+x)^{j}Q_{k}(x)dx}-2^{j+1}\\
 &=&2(j+1)j\int_{-1}^{1}{(1+x)^{j-1}P_{k}(x)dx}-(j+2)(j+1)\int_{-1}^{1}{(1+x)^{j}P_{k}(x)dx}\\
&=&2^{j+1}(j+1)j\frac{\Gamma^{2}(j)}{\Gamma(j+k+1)\Gamma(j-k)}-2^{j+1}(j+2)(j+1)\frac{\Gamma^{2}(j+1)}{\Gamma(j+k+2)\Gamma(j-k+1)}\Bigr)\\
&=& 2^{j+1}j(j+1)\Gamma^{2}(j)\Bigl(\frac{1}{\Gamma(j+k+1)\Gamma(j-k)}-j(j+2)
\frac{1}{\Gamma(j+k+2)\Gamma(j-k+1)}\Bigr)\\
&=& 2^{j+1}\frac{j(j+1)\Gamma^{2}(j)}{\Gamma(j+k+2)\Gamma(j-k+1)}\Bigl((j+k+1)(j-k)-j(j+2)\Bigr)\\
&=& -2^{j+1}\frac{j(j+1)(j+k+k^{2})\Gamma^{2}(j)}{\Gamma(j+k+2)\Gamma(j-k+1)}\\
&=& -2^{j+1}\frac{(j+k+k^{2})\Gamma(j+2)\Gamma(j)}{\Gamma(j+k+2)\Gamma(j-k+1)}
\end{eqnarray*}
and therefore
\begin{eqnarray*}
\int_{-1}^{1}{(1+x)^{j}Q_{k}(x)dx}=
2^{j+1}\Bigl(1-\frac{(j+k+k^{2})\Gamma(j+2)\Gamma(j)}{\Gamma(j+k+2)\Gamma(j-k+1)}\Bigr)\, .
\end{eqnarray*}
Translating to the unit interval with the map $I \mapsto [-1,1]$ defined by
$x=2r-1$ we obtain the assertion.

\section{Appendix}

\subsection{Proof of even case of  Lemma \ref{lem_mass_inf}}

We utilize the bijective principal representation $\phi^{p}_{el}: \Lambda^{m} \times T^{m} \rightarrow Int(M^{2m})$ defined in \eqref{def_phi_pel} and
\[|det(d\phi^{p}_{eu})|(\lambda,t)= \mathcal{J}^{p}_{eu}(t)
\prod_{j=1}^{m}{\lambda_{j}} \]
where
\[
\mathcal{J}^{p}_{el}(t)
:= \prod_{j=1}^{m}{t_{j}^{2}}
\cdot \Delta_{m}^{4}(t)
\]
from Proposition \ref{prop_Jpe}.

Fix $t_{*}\in (0,1)$ and let
\[ T_{\d}^{m}:= \bigl\{ (t_{1},..,t_{m})\in   T^{m}:
t_{j} \notin B_{\d}(t_{*}) , j=1,..,m\, \bigr\} . \]
It follows that
\[M^{2m}_{\d} \supset \phi^{p}_{el}\bigl(\Lambda^{m}\times T_{\d}^{m}\bigr)\]
and therefore
\begin{equation}
\label{e22222}
Vol\Bigl(M^{2m}_{\d}\Bigr)  \geq  Vol\Bigl(\phi^{p}_{el}\bigl(\Lambda^{m}\times T_{\d}^{m} \bigr)\Bigr)\, .
\end{equation}

Using the identity
$\int_{\Lambda^{m}}{\prod_{i=1}^{m}{\lambda_{i}}d\lambda} = \frac{1}{(2m)!}$,
we compute the righthand side using the change of variables formula as
\begin{eqnarray*}
 Vol\Bigl(\phi^{p}_{el}\bigl(\Lambda^{m}\times T_{\d}^{m} \bigr)\Bigr)
&=&
\int_{\Lambda^{m} \times  T_{\d}^{m}}{|det(d\phi^{p}_{el})|
 }\\
&=&\int_{\Lambda^{m}}{\bigl(\prod_{j=1}^{m}{\lambda_{j}}\bigr)
d\lambda} \int_{T_{\d}^{m}}{\mathcal{J}^{p}_{el}}\\
&=&\frac{1}{(2m)!}
\int_{T_{\d}^{m}}{\mathcal{J}^{p}_{el}}\\
&=&\frac{1}{(2m)!m!}
\int_{I_{\d}^{m}}{\mathcal{J}^{p}_{el}}
\end{eqnarray*}
where
\[ I_{\d}^{m}:= \bigl\{ (t_{1},..,t_{m})\in   I^{m}:
t_{j} \notin B_{\d} , j=1,..,m\, \bigr\} . \]
To bound this from below we bound the integral over
$ \bigl(I_{\d}^{m}\bigr)^{c}$ from above.
To that end, let
\[ I_{\d,j}^{m}:= \bigl\{ (t_{1},..,t_{m})\in   I^{m}:
t_{j} \in B_{\d}\bigr \},\quad  j=1,..,m\,, \]
so that
\[ \bigl(I_{\d}^{m}\bigr)^{c}=  \cup_{j} I_{\d,j}^{m}\, .\]
Therefore, using a union bound, we have
\begin{eqnarray*}
 \int_{\bigl(I_{\d}^{m}\bigr)^{c}}{\mathcal{J}^{p}_{el}}&=&
 \int_{\bigl(I_{\d}^{m}\bigr)^{c}}{ \prod_{j=1}^{m}{t_{j}^{2}}\cdot \Delta_{m}^{4}(t)dt}\\
 &=&\int_{\cup_{j'=1}^{m}{I^{m}_{\d,j'}}}{ \prod_{j=1}^{m}{t_{j}^{2}}\cdot \Delta_{m}^{4}(t)dt}\\
&\leq &\sum_{j'=1}^{m}{\int_{I^{m}_{\d,j'}}{ \prod_{j=1}^{m}{t_{j}^{2}}\cdot \Delta_{m}^{4}(t)dt}}\\
& =&m\int_{I^{m}_{\d,1}}{ \prod_{j=1}^{m}{t_{j}^{2}}\cdot \Delta_{m}^{4}(t)dt}\\
& = &m\int_{I^{m}_{\d,1}}{ \prod_{j=1}^{m}{t_{j}^{2}}\prod_{1 \leq j<k \leq m}{(t_{k}-t_{j})^{4}}dt_{1}\cdots dt_{m}}\\
& \leq &m\int_{I^{m}_{\d,1}}{ \prod_{j=2}^{m}{t_{j}^{2}}\prod_{2 \leq j<k \leq m}{(t_{k}-t_{j})^{4}}dt_{1}\cdots dt_{m}}\\
& = &mVol(B_{\d})\int_{I^{m-1}}{ \prod_{j=2}^{m}{t_{j}^{2}}\cdot \Delta_{m-1}^{4}(t)dt}\\
& = &mVol(B_{\d}) S_{m-1}(3,1,2)\\
&\leq & 2m\d S_{m-1}(3,1,2)\\
\end{eqnarray*}
and so obtain
\begin{eqnarray*}
\frac{Vol\Bigl(\phi^{p}_{el}(\Lambda^{m}\times T_{\d}^{m})\Bigr)}{Vol\Bigl(\phi^{p}_{el}(\Lambda^{m}\times T^{m})\Bigr)}
&=& \frac{\int_{I^{m}_{\d}}{ \prod_{j=1}^{m}{t_{j}^{2}}\cdot \Delta_{m}^{4}(t)dt}}{\int_{I^{m}}{ \prod_{j=1}^{m}
{t_{j}^{2}}\cdot \Delta_{m}^{4}(t)dt}}\\
&\geq& 1- 2m\d \frac{S_{m-1}(3,1,2)}{S_{m}(3,1,2)}\\
\end{eqnarray*}

Using Selberg's formulas \eqref{selbergformula} we compute
\begin{eqnarray*}
\frac{S_{m-1}(3,1,2)}{S_{m}(3,1,2)}&=&
\frac{ \prod_{j=0}^{m-2}{\frac{\Gamma(1+ 2j)
\Gamma(3+2j)^{2}}{2\Gamma(2(m+j)) }}}{ \prod_{j=0}^{m-1}{\frac{\Gamma(1+ 2j)
\Gamma(3+2j)^{2}}{2\Gamma(2(m+j)+2) }}}\\
&=&\frac{2\Gamma(4m-2)}{\Gamma(2m-1)^{2}\Gamma(2m+1)}
\frac{ \prod_{j=0}^{m-1}{\frac{\Gamma(1+ 2j)
\Gamma(3+2j)^{2}}{2\Gamma(2(m+j)) }}}{ \prod_{j=0}^{m-1}{\frac{\Gamma(1+ 2j)
\Gamma(3+2j)^{2}}{2\Gamma(2(m+j)+2) }}}\\
&=&\frac{2\Gamma(4m-2)}{\Gamma(2m-1)^{2}\Gamma(2m+1)}
 \prod_{j=0}^{m-1}{\frac{\Gamma(2(m+j)+2) }{\Gamma(2(m+j))}} \\
&=&\frac{2\Gamma(4m-2)}{\Gamma(2m-1)^{2}\Gamma(2m+1)}
  \frac{\Gamma(4m)}{\Gamma(2m)}\, .
\end{eqnarray*}

\subsection{Even case  of Proof of Theorem \ref{thm_RK}}
\label{sec_RKproof}
For the even case, $n=2m$,
recall the identity  $\mathcal{J}^{c}_{eu}(0,t)\equiv 0$, and the
 volume identity \eqref{vol_Jce}
\begin{equation}
\label{vol_Jce2}
Vol(M^{2m})=\frac{1}{(2m)!m!}\int_{I^{m}}{\mathcal{J}^{c}_{el}(t_{*},t)dt}+
\frac{1}{(2m)!(m-1)!}\int_{I^{m-1}}{\mathcal{J}^{c}_{eu}(t_{*},t)dt}\,
\end{equation}
 Now observe that, for $i\geq 0$,   the identity
\begin{eqnarray}
&&\int_{M^{2m}}{q_{i}}- \frac{t_{*}^{i}}{2m+1}Vol(M^{2m})\notag\\
&=&\frac{\d_{0}(i)+1}{(2m+1)!(m-1)!}\int_{I^{m-1}}{
\mathcal{J}^{c}_{eu}(t_{*},t)dt}\notag\\
&+&\frac{2}{(2m+1)!m!}\int_{I^{m}}{\Sigma t^{i}\mathcal{J}^{c}_{el}(t_{*},t)dt}
+
\frac{2}{(2m+1)!(m-1)!}\int_{I^{m-1}}{\Sigma t^{i}
\mathcal{J}^{c}_{eu}(t_{*},t)dt}\notag
\end{eqnarray}
from Proposition \ref{prop_mom},
evaluated at $t_{*}=0$ becomes
\begin{eqnarray}
&&\int_{M^{2m}}{q_{i}}- \frac{\d_{0}(i)}{2m+1}Vol(M^{2m})\notag\\
&=&\frac{\d_{0}(i)+1}{(2m+1)!(m-1)!}\int_{I^{m-1}}{
\mathcal{J}^{c}_{eu}(0,t)dt}\notag\\
&+&\frac{2}{(2m+1)!m!}\int_{I^{m}}{\Sigma t^{i}\mathcal{J}^{c}_{el}(0,t)dt}
+
\frac{2}{(2m+1)!(m-1)!}\int_{I^{m-1}}{\Sigma t^{i}
\mathcal{J}^{c}_{eu}(0,t)dt}\notag
\end{eqnarray}
where we note that $\Sigma t^{i}=\sum_{j=1}^{m}{t_{j}^{i}}$ in the  integral over
$I^{m}$ and $\Sigma t^{i}=\sum_{j=1}^{m-1}{t_{j}^{i}}$ in the  integral over
$I^{m-1}$.

Subtracting the two,
using the identity $\mathcal{J}^{c}_{eu}(0,t)\equiv 0$, we obtain
\begin{eqnarray}
\label{eq_even_rk1}
 &&-\frac{t_{*}^{i}}{2m+1}Vol(M^{2m})+ \frac{\delta_{0}(i)}{2m+1}Vol(M^{2m})\notag\\
&=&\frac{\d_{0}(i)+1}{(2m+1)!(m-1)!}\int_{I^{m-1}}{
\mathcal{J}^{c}_{eu}(t_{*},t)dt}\notag\\
&+&\frac{2}{(2m+1)!m!}\int_{I^{m}}{\Sigma t^{i}\bigl(\mathcal{J}^{c}_{el}(t_{*},t)-
\mathcal{J}^{c}_{el}(0,t)\bigr)dt}
+
\frac{2}{(2m+1)!(m-1)!}\int_{I^{m-1}}{\Sigma t^{i}
\mathcal{J}^{c}_{eu}(t_{*},t)dt}\notag
\end{eqnarray}
and applying the volume identity \eqref{vol_Jce2} we obtain
\begin{eqnarray}
\label{eq_even_rk2}
 &&-\frac{t_{*}^{i}}{2m+1}Vol(M^{2m})\notag\\
&=& -\frac{\d_{0}(i)}{(2m+1)!m!}\int_{I^{m}}{\mathcal{J}^{c}_{el}(t_{*},t)
dt}+ \frac{1}{(2m+1)!(m-1)!}\int_{I^{m-1}}{
\mathcal{J}^{c}_{eu}(t_{*},t)dt}\notag\\
&+&\frac{2}{(2m+1)!m!}\int_{I^{m}}{\Sigma t^{i}\bigl(\mathcal{J}^{c}_{el}(t_{*},t)-
\mathcal{J}^{c}_{el}(0,t)\bigr)dt}
+
\frac{2}{(2m+1)!(m-1)!}\int_{I^{m-1}}{\Sigma t^{i}
\mathcal{J}^{c}_{eu}(t_{*},t)dt}\notag \, .
\end{eqnarray}

Then, if we let  $\phi(s):=\sum_{i=0}^{2m}{\phi_{i}s^{i}}$ be a polynomial of degree
$n=2m$, summing over each identity in \eqref{eq_even_rk2}, we conclude  that
\begin{eqnarray*}
 &&-\frac{\phi(t_{*})}{2m+1}Vol(M^{2m})\notag\\
&=&\phi_{0}\Bigl(\frac{1}{(2m+1)!(m-1)!}\int_{I^{m-1}}{
\mathcal{J}^{c}_{eu}(t_{*},t)dt} -\frac{1}{(2m+1)!m!}\int_{I^{m}}{\mathcal{J}^{c}_{el}(t_{*},t)
dt}\Bigr)
\notag\\
&+& \sum_{j=1}^{2m}{\phi_{j}}\frac{1}{(2m+1)!(m-1)!}\int_{I^{m-1}}{
\mathcal{J}^{c}_{eu}(t_{*},t)dt}
\notag\\
&+&\frac{2}{(2m+1)!m!}\int_{I^{m}}{(\Sigma \phi)(t)\bigl(\mathcal{J}^{c}_{el}(t_{*},t)-
\mathcal{J}^{c}_{el}(0,t)\bigr)dt}
+
\frac{2}{(2m+1)!(m-1)!}\int_{I^{m-1}}{(\Sigma \phi)(t)
\mathcal{J}^{c}_{eu}(t_{*},t)dt}\notag \\
&=&-\phi_{0}
\frac{1}{(2m+1)!m!}\int_{I^{m}}{\mathcal{J}^{c}_{el}(t_{*},t)
dt}
\notag\\
&+& \sum_{j=0}^{2m}{\phi_{j}}\frac{1}{(2m+1)!(m-1)!}\int_{I^{m-1}}{
\mathcal{J}^{c}_{eu}(t_{*},t)dt}
\notag\\
&+&\frac{2}{(2m+1)!m!}\int_{I^{m}}{(\Sigma \phi)(t)\bigl(\mathcal{J}^{c}_{el}(t_{*},t)-
\mathcal{J}^{c}_{el}(0,t)\bigr)dt}
+
\frac{2}{(2m+1)!(m-1)!}\int_{I^{m-1}}{(\Sigma \phi)(t)
\mathcal{J}^{c}_{eu}(t_{*},t)dt}\notag \\
\end{eqnarray*}

Since $\phi(0)=\phi_{0}$ and $\phi(1)=\sum_{i=0}^{2m}{\phi_{i}}$ the assertion follows by multiplication by
$2m+1$.

\subsection{Assorted Technical Results}
\begin{prop}
\label{prop_tech}
We have
\[
m^{2} \leq 8 (\frac{e}{2})^{4m}, \quad m \geq 1\, .\]
\end{prop}
\begin{proof}
We proceed by induction. The inequality is clearly true for $m=1,2.$  Therefore,
   suppose that it is true for some $m\geq 2$. Then, since
   \[\frac{m+1}{m} \leq \frac{3}{2}\leq \bigl(\frac{e}{2}\bigr)^{2},\quad  m\geq 2\]
   follows from $6 \leq e^{2}$, we conclude that
   \begin{eqnarray*}
   (m+1)^{2}&=& \frac{(m+1)^{2}}{m^{2}}m^{2} \\
	   &\leq& \frac{(m+1)^{2}}{m^{2}} 8 (\frac{e}{2})^{4m} \\
	   &\leq& \bigl(\frac{e}{2}\bigr)^{4} 8 (\frac{e}{2})^{4m}\\
	   &\leq&  8 (\frac{e}{2})^{4m+4}
	   \end{eqnarray*}
	   thus establishing the inequality for $m+1$ and finishing the proof.
	   \end{proof}

	   \begin{prop}
	   \label{beta_lowerbound}
	   We have
	   \[B(a,a) \geq \frac{4}{a}2^{-2a}, \quad a >1\]
	   \end{prop}
	   \begin{proof}
	   We have from the integral formula
	   \begin{eqnarray*}
	   B(a,b)&=&\int_{0}^{1}{t^{a-1}(1-t)^{a-1}dt}\\
		   &=&2\int_{0}^{\frac{1}{2}}{t^{a-1}(1-t)^{a-1}dt}\\
		   &\geq &2\bigl(\frac{1}{2}\bigr)^{a-1}\int_{0}^{\frac{1}{2}}{t^{a-1}dt}\\
		   &=&2\bigl(\frac{1}{2}\bigr)^{a-1}\frac{1}{a}\bigl(\frac{1}{2}\bigr)^{a}\\
		   &=&\frac{4}{a}2^{-2a}
		   \end{eqnarray*}
		   \end{proof}

\section*{Acknowledgements}
\addcontentsline{toc}{section}{Acknowledgements}

We would like to thank G\'{e}rard Letac for his helpful comments, in particular for his substantial simplification,
 included here, of our previous
 proof of Lemma 4.1. We would also like to thank one of the referees for many helpful comments which we also feel improved
the manuscript.

The authors gratefully acknowledge this work supported by the Air Force Office of Scientific Research under Award Number
FA9550-12-1-0389 (Scientific Computation of Optimal Statistical Estimators).

\addcontentsline{toc}{section}{References}
\bibliographystyle{plain}
\bibliography{./refs}

\def\polhk#1{\setbox0=\hbox{#1}{\ooalign{\hidewidth
  \lower1.5ex\hbox{`}\hidewidth\crcr\unhbox0}}} \def\cprime{$'$}
\begin{thebibliography}{10}

\bibitem{Christophe:2002}
C.~Abraham and B.~Cadre.
\newblock Asymptotic properties of posterior distributions derived from
  misspecified models.
\newblock {\em C. R. Math. Acad. Sci. Paris}, 335(5):495--498, 2002.

\bibitem{Christophe:2008}
C.~Abraham and B.~Cadre.
\newblock Concentration of posterior distributions with misspecified models.
\newblock {\em Ann. I.S.U.P.}, 52(3):3--14, 2008.

\bibitem{AbramowitsStegun}
M.~Abramowitz and I.~A. Stegun.
\newblock {\em Handbook of Mathematical Functions}.
\newblock Dover, New York, 1965.

\bibitem{AliprantisBorder:2006}
C.~D. Aliprantis and K.~C. Border.
\newblock {\em Infinite {D}imensional {A}nalysis: {A} {H}itchhiker's {G}uide}.
\newblock Springer, Berlin, third edition, 2006.

\bibitem{Arfken}
G.~B. Arfken and H.~J. Weber.
\newblock {\em Mathematical Methods for Physicists}.
\newblock Elsevier Academic Press, Burlington, Mass., sixth edition, 2005.

\bibitem{BahadurSavage}
R.~R. Bahadur and L.~J. Savage.
\newblock The nonexistence of certain statistical procedures in nonparametric
  problems.
\newblock {\em Ann. Math. Statist.}, 27(4):1115--1122, 1956.

\bibitem{Bernstein:1964}
S.~N. Bern{\v{s}}te{\u\i}n.
\newblock {\em Sobranie sochinenii. {T}om {IV}: {T}eoriya veroyatnostei.
  {M}atematicheskaya statistika. 1911--1946}.
\newblock Izdat. ``Nauka'', Moscow, 1964.

\bibitem{BertsimasPopescu:2005}
D.~Bertsimas and I.~Popescu.
\newblock Optimal inequalities in probability theory: a convex optimization
  approach.
\newblock {\em SIAM J. Optim.}, 15(3):780--804 (electronic), 2005.

\bibitem{Bogachev_Gauss}
V.~I. Bogachev.
\newblock {\em Gaussian Measures}.
\newblock Mathematical Surveys and Monographs, Vol.~62. American Mathematical
  Society, Providence, 1998.

\bibitem{CastaingValadier:1977}
C.~Castaing and M.~Valadier.
\newblock {\em Convex {A}nalysis and {M}easurable {M}ultifunctions}.
\newblock Lecture Notes in Mathematics, Vol. 580. Springer-Verlag, Berlin,
  1977.

\bibitem{ChangKempermanStudden}
F-C. Chang, J.~H.~B. Kemperman, and W.~J. Studden.
\newblock A normal limit theorem for moment sequences.
\newblock {\em Annals of Probability}, 21(3):1295--1309, 1993.

\bibitem{Christoffel}
E.~B. Christoffel.
\newblock \"{U}ber die {G}au\ss ische {Q}uadratur und eine {V}erallgemeinerung
  derselben.
\newblock {\em J. Reine Angew. Math.}, 55:61--82, 1858.

\bibitem{CormenLeisersonRivestStein}
T.~H. Cormen, C.~E. Leiserson, R.~L. Rivest, and C.~Stein.
\newblock {\em Introduction to Algorithms}.
\newblock MIT Press, Cambrige, Mass., 1990.

\bibitem{DiaconisFreedman:1986}
P.~Diaconis and D.~Freedman.
\newblock On the consistency of {B}ayes estimates.
\newblock {\em Ann. Statist.}, 14(1):1--67, 1986.
\newblock With a discussion and a rejoinder by the authors.

\bibitem{DiaconisFreedman:1998}
P.~W. Diaconis and D.~Freedman.
\newblock Consistency of {B}ayes estimates for nonparametric regression: normal
  theory.
\newblock {\em Bernoulli}, 4(4):411--444, 1998.

\bibitem{Donoho}
D.~L. Donoho.
\newblock One-sided inference about functionals of a density.
\newblock {\em Ann. Statist.}, 16(4):1390--1420, 1988.

\bibitem{Draper:1995}
D.~Draper.
\newblock Assessment and propagation of model uncertainty.
\newblock {\em J. Roy. Statist. Soc. Ser. B}, 57(1):45--97, 1995.
\newblock With discussion and a reply by the author.

\bibitem{Draper:2013}
D.~Draper.
\newblock Bayesian model specification: heuristics and examples.
\newblock In P.~Damien, P.~Dellaportas, N.~G. Polson, and D.~A. Stephens,
  editors, {\em Bayesian {T}heory and {A}pplications}. Oxford University Press,
  2013.

\bibitem{ForresterWarnaar}
P.~J. Forrester and S.~O. Warnaar.
\newblock The importance of the {S}elberg integral.
\newblock {\em Bull Amer. Math. Soc.}, 45(4):489--534, 2008.

\bibitem{GradshteynRyzhik}
I.~S. Gradshteyn and I.~M. Ryzhik.
\newblock {\em Table of Integrals, Series, and Products}.
\newblock Academic Press, San Diego, seventh edition, 2007.

\bibitem{Grunwald:2004}
P.~Gr{\"u}nwald and J.~Langford.
\newblock Suboptimal behavior of {B}ayes and {MDL} in classification under
  misspecification.
\newblock In {\em Learning theory}, volume 3120 of {\em Lecture Notes in
  Comput. Sci.}, pages 331--347. Springer, Berlin, 2004.

\bibitem{Gustafson:2001}
Paul Gustafson.
\newblock On measuring sensitivity to parametric model misspecification.
\newblock {\em J. R. Stat. Soc. Ser. B Stat. Methodol.}, 63(1):81--94, 2001.

\bibitem{HornJohnson2}
R.~A. Horn and C.~R. Johnson.
\newblock {\em Topics in Matrix Analysis}.
\newblock Cambridge University Press, Cambridge, 1991.

\bibitem{KarlinShapley}
S.~Karlin and L.~S. Shapley.
\newblock {\em Geometry of Moment Spaces}.
\newblock Memoirs of the American Mathematical Society, Number 12. American
  Math. Soc., 1953.

\bibitem{KarlinStudden:1966}
S.~Karlin and W.~J. Studden.
\newblock {\em Tchebycheff {S}ystems: {W}ith {A}pplications in {A}nalysis and
  {S}tatistics}.
\newblock Pure and Applied Mathematics, Vol. XV. Interscience Publishers John
  Wiley \& Sons, New York-London-Sydney, 1966.

\bibitem{Kleijn:2012}
B.~J.~K. Kleijn and A.~W. van~der Vaart.
\newblock The {B}ernstein-{V}on-{M}ises theorem under misspecification.
\newblock {\em Electron. J. Stat.}, 6:354--381, 2012.

\bibitem{Krein}
M.~G. Kre\u{\i}n.
\newblock The ideas of {P}. {L}. \u{C}eby\u{s}ev and {A}. {A}. {M}arkov in the
  theory of limiting values of integrals and their further development.
\newblock In E.~B. Dynkin, editor, {\em Eleven papers on Analysis, Probability,
  and Topology}, American Mathematical Society Translations, Series 2, Volume
  12, pages 1--122. American Mathematical Society, New York, 1959.

\bibitem{LeCam:1953}
L.~LeCam.
\newblock On some asymptotic properties of maximum likelihood estimates and
  related {B}ayes' estimates.
\newblock {\em Univ. California Publ. Statist.}, 1:277--329, 1953.

\bibitem{Lian:2009}
H.~Lian.
\newblock On rates of convergence for posterior distributions under
  misspecification.
\newblock {\em Comm. Statist. Theory Methods}, 38(11-12):1893--1900, 2009.

\bibitem{BayesOUQ}
H.~Owhadi, C.~Scovel, and T.~J. Sullivan.
\newblock Bayesian {B}rittleness.
\newblock {\em arXiv:1304.6772}, 2013.

\bibitem{OSSMO:2011}
H.~Owhadi, C.~Scovel, T.~J. Sullivan, M.~McKerns, and M.~Ortiz.
\newblock Optimal {U}ncertainty {Q}uantification.
\newblock {\em SIAM Review}, to appear, 2013.
\newblock arXiv:1009.0679 (published online on 2010).

\bibitem{Grunwald:2006}
Gr\"{u}nwald P.
\newblock Bayesian inconsistency under misspecification, 2006.

\bibitem{Ressel_some}
P.~Ressel.
\newblock Some continuity and measurability results on spaces of measures.
\newblock {\em Mathematica Scandinavica}, 40:69--78, 1977.

\bibitem{Rockafellar}
R.~T. Rockafellar.
\newblock {\em Convex Analysis}.
\newblock Princeton Univ. Press, Princeton, 1970.

\bibitem{SainteBeuve1974}
M.-F. Sainte-Beuve.
\newblock On the extension of von {N}eumann-{A}umann's theorem.
\newblock {\em J. Functional Analysis}, 17:112--129, 1974.

\bibitem{Saitoh}
S.~Saitoh.
\newblock {\em Integral transforms, reproducing kernels and their
  applications}.
\newblock Addison Wesley Longman, Essex, 1997.

\bibitem{Samaniego:2010}
F.~J. Samaniego.
\newblock {\em A comparison of the {B}ayesian and frequentist approaches to
  estimation}.
\newblock Springer Series in Statistics. Springer, New York, 2010.
\newblock \url{http://dx.doi.org/10.1007/978-1-4419-5941-6}.

\bibitem{Simon}
B.~Simon.
\newblock The {C}hristoffel–-{D}arboux kernel.
\newblock In {\em Perspectives in {PDE}, Harmonic Analysis and Applications, a
  volume in honor of VG Maz'ya's 70th birthday}, pages 295--335. American Math.
  Soc., 2008.

\bibitem{Skibinsky}
M.~Skibinsky.
\newblock The range of the (n + 1)th moment for distributions on [0, 1].
\newblock {\em Journal of Applied Probability}, 4(3):543--552, 1967.

\bibitem{Smith}
J.~E. Smith.
\newblock Generalized {C}hebychev inequalities: Theory and applications in
  decision analysis.
\newblock {\em Operations Research}, 43(5):807--825, 1995.

\bibitem{Spivak}
M.~Spivak.
\newblock {\em Calculus on Manifolds}.
\newblock Addison-Wesley, Reading, Massachusetts, 1965.

\bibitem{vonMises:1964}
R.~von Mises.
\newblock {\em Mathematical {T}heory of {P}robability and {S}tatistics}.
\newblock Edited and Complemented by Hilda Geiringer. Academic Press, New York,
  1964.

\bibitem{White:1982}
H.~White.
\newblock Maximum likelihood estimation of misspecified models.
\newblock {\em Econometrica}, 50(1):1--25, 1982.

\end{thebibliography}
\end{document}